\documentclass[reqno,11pt]{preprint}

\usepackage[marginparwidth=1in]{geometry}
\usepackage[full]{textcomp}
\usepackage[osf]{newtxtext}
\usepackage{cabin}
\usepackage{enumerate}
\usepackage{enumitem}
\usepackage{bbm}
\usepackage{dsfont}

\usepackage[varqu,varl]{inconsolata}
\usepackage[cal=boondoxo]{mathalfa}

\usepackage{comment}
\usepackage{hyperref}
\usepackage{breakurl}
\usepackage{mathrsfs}
\usepackage{booktabs}
\usepackage{caption}

\usepackage{tikz}
\usetikzlibrary{decorations.pathreplacing,decorations.markings}

\usepackage{amsmath}
\usepackage{amssymb}
\usepackage{mhequ}
\usepackage{mhsymb}
\usepackage{mhenvs}
\usepackage{microtype}

\usepackage{eufrak}

\usepackage{wasysym}
\usepackage{centernot}

\usepackage{upgreek}

\usepackage{oldgerm}

\newcommand{\Exp}{\mathbf{E}}
\newcommand{\Prob}{\mathbf{P}}

\newcommand{\bk}{\mathbf{k}}

\renewcommand{\R}{\mathbb{R}}

\renewcommand{\N}{\mathbb{N}}
\renewcommand{\L}{\mathbb{L}}
\renewcommand{\Q}{\mathbb{Q}}
\newcommand{\T}{\mathbb{T}}
\renewcommand{\Z}{\mathbb{Z}}
\renewcommand{\E}{\mathbb{E}}
\renewcommand{\P}{\mathbb{P}}

\newcommand{\cA}{\mathcal{A}}
\newcommand{\cB}{\mathcal{B}}
\newcommand{\cC}{\mathcal{C}}
\newcommand{\cD}{\mathcal{D}}
\newcommand{\cE}{\mathcal{E}}
\newcommand{\cF}{\mathcal{F}}
\newcommand{\cG}{\mathcal{G}}
\newcommand{\cH}{\mathcal{H}}
\newcommand{\cI}{\mathcal{I}}
\newcommand{\cJ}{\mathcal{J}}
\newcommand{\cK}{\mathcal{K}}
\newcommand{\cL}{\mathcal{L}}
\newcommand{\cM}{\mathcal{M}}
\newcommand{\cN}{\mathcal{N}}
\newcommand{\cO}{\mathcal{O}}

\newcommand{\cS}{\mathcal{S}}

\newcommand{\cg}{\mathcal g}
\newcommand{\ch}{\mathcal h}

\newcommand{\dd}{\mathrm{d}}
\newcommand{\dis}{\mathrm{dis}}

\newcommand{\Id}{\mathrm{Id}}
\newcommand{\loc}{\mathrm{loc}}

\newcommand{\SH}{\mathscr{H}}

\newcommand{\bd}{\mathbf{d}}

\newcommand{\fe}{\mathfrak{e}}
\newcommand{\fw}{\mathfrak{w}}

\newcommand{\fm}{\mathfrak{m}^{\tau,\delta}}
\newcommand{\tfm}{\tilde{\mathfrak{m}}^{\tau,\delta}}

\def\one{\mathrm{(I)}}
\def\two{\mathrm{(II)}}
\def\three{\mathrm{(III)}}

\newcommand{\vertiii}[1]{{\vert\kern-0.25ex\vert\kern-0.25ex\vert #1 
    \vert\kern-0.25ex\vert\kern-0.25ex\vert}}

\newcommand{\1}{\mathds{1}}

\newcommand{\di}{\mathrm{diag}}
\newcommand{\oD}{\mathrm{off}}
\newcommand{\oDi}{\mathrm{off_1}}
\newcommand{\ooDi}{\mathrm{off_2}}

\colorlet{darkblue}{blue!90!black}
\colorlet{darkred}{red!90!black}
\colorlet{darkgreen}{green!50!black}
\colorlet{darkyellow}{yellow!90!black}
\newcommand{\fock}[3]{
    \def\order{#2}
    \ifx\order\empty
        \Gamma_{#1} L^2_{#3}
    \else
        \Gamma_{#1} H^{#2}_{#3}
    \fi
}

\newcommand{\core}{\boldsymbol{\Gamma}\mathbf{H}_\tau}

\newcommand{\fockcore}[3]{
 \def\order{#2}
    \ifx\order\empty
        \boldsymbol{\Gamma}_{#1} \mathbf{L}^2_{#3}
    \else
        \boldsymbol{\Gamma}_{#1} \mathbf{H}^{#2}_{#3}
    \fi
}

\newcommand{\gen}{\cL^\tau}
\newcommand{\gensy}{\cS^\tau}
\newcommand{\gensyx}{\cL_0^{\fe_1}}

\newcommand{\gena}{\cA^{\tau}}
\newcommand{\genap}{\cA^{\tau}_+}
\newcommand{\genam}{\cA^{\tau}_-}

\newcommand{\genash}{\cA^{\sharp, \tau}}
\newcommand{\genafl}{\cA^{\flat, \tau}}
\newcommand{\genapsh}{\cA^{\sharp, \tau}_+}
\newcommand{\genapfl}{\cA^{\flat, \tau}_+}
\newcommand{\genamsh}{\cA^{\sharp, \tau}_-}
\newcommand{\genamfl}{\cA^{\flat, \tau}_-}

\newcommand{\geneff}{\cL^{\mathrm{eff}}}

\newcommand{\nonlinp}{\cK^{+, \tau}}
\newcommand{\nonlinm}{\cK^{-, \tau}}

\newcommand{\eff}{\mathrm{eff}}
\newcommand{\p}{\mathrm{p}}
\newcommand{\e}{\mathrm{e}}
\newcommand{\we}{\mathrm{w}}

\makeatletter
\def\DeclareSymbol#1#2#3{\expandafter\gdef\csname MH@symb@#1\endcsname{\tikzsetnextfilename{symbol#1}\tikz[baseline=#2,scale=0.15,draw=symbols]{#3}}\expandafter\gdef\csname MH@symb@#1s\endcsname{\scalebox{0.7}{\tikzsetnextfilename{symbol#1s}\tikz[baseline=#2,scale=0.15,draw=symbols]{#3}}}}
\def\<#1>{\csname MH@symb@#1\endcsname}
\makeatother

\definecolor{Red}{rgb}{1,0,0}
\definecolor{Blue}{rgb}{0,0,1}
\definecolor{Olive}{rgb}{0.41,0.55,0.13}
\definecolor{Yarok}{rgb}{0,0.5,0}
\definecolor{Green}{rgb}{0,1,0}
\definecolor{MGreen}{rgb}{0,0.8,0}
\definecolor{DGreen}{rgb}{0,0.65,0}
\definecolor{Yellow}{rgb}{1,1,0}
\definecolor{Cyan}{rgb}{0,1,1}
\definecolor{Magenta}{rgb}{1,0,1}
\definecolor{Orange}{rgb}{1,.5,0}
\definecolor{Violet}{rgb}{.5,0,.5}
\definecolor{Purple}{rgb}{.75,0,.25}
\definecolor{Brown}{rgb}{.75,.5,.25}
\definecolor{Grey}{rgb}{.7,.7,.7}
\definecolor{Black}{rgb}{0,0,0}
\definecolor{dr}{rgb}{0.8,0,0}
\definecolor{db}{rgb}{0,0,0.8}

\newcounter{list_size}
\newcommand{\connect}[4][dr]{
	\edef\myleftx{10000}
	\edef\myrightx{-10000}
	\edef\mycentery{0}
	\setcounter{list_size}{0}
	\foreach \j in {#2}{\stepcounter{list_size}};
	\edef\n{\arabic{list_size}}
	\ifthenelse{\n=1}
	{
		\draw (#3#2) -- (#4#2);
	}
	{
		\foreach \j in {#2} {
			\path (#3\j); \pgfgetlastxy{\XCoord}{\YCoord};
			\pgfmathsetmacro{\lx}{min(\myleftx,\XCoord)};
			\pgfmathsetmacro{\rx}{max(\myrightx,\XCoord)};
			\pgfmathsetmacro{\cy}{\mycentery+\YCoord/(2*\n)};
			\path (#4\j); \pgfgetlastxy{\XCoord}{\YCoord};
			\pgfmathsetmacro{\cy}{\cy+\YCoord/(2*\n)};
			\global\let\myleftx=\lx
			\global\let\myrightx=\rx
			\global\let\mycentery=\cy
		}
		\foreach \j in {#2} {
			\draw (#3\j) -- (#4\j);
		}
		\draw[thick,#1] (\myleftx pt,\mycentery pt) -- (\myrightx pt,\mycentery pt);
		\foreach \j in {#2} {
			\path (#3\j); \pgfgetlastxy{\XCoord}{\YCoord};
			\node at (\XCoord,\mycentery pt) [simple,fill=#1] {};
		}
	}
}

\newcommand{\eqlaw}{\stackrel{\mbox{\tiny law}}{=}}

\def\restriction#1#2{\mathchoice
              {\setbox1\hbox{${\displaystyle #1}_{\scriptstyle #2}$}
              \restrictionaux{#1}{#2}}
              {\setbox1\hbox{${\textstyle #1}_{\scriptstyle #2}$}
              \restrictionaux{#1}{#2}}
              {\setbox1\hbox{${\scriptstyle #1}_{\scriptscriptstyle #2}$}
              \restrictionaux{#1}{#2}}
              {\setbox1\hbox{${\scriptscriptstyle #1}_{\scriptscriptstyle #2}$}
              \restrictionaux{#1}{#2}}}
\def\restrictionaux#1#2{{#1\,\smash{\vrule height .8\ht1 depth .85\dp1}}_{\,#2}}

\let\PHI=\Phi
\let\Phi=\phi
\let\phi=\varphi
\let\epsilon=\varepsilon
\let\dis=\displaystyle

\title{Superdiffusive Central Limit Theorem for the Stochastic Burgers Equation at the critical dimension}

\begin{document}

\maketitle

\vspace{-2cm}

\noindent{\large \bf Giuseppe Cannizzaro$^1$,  Quentin Moulard$^2$, Fabio Toninelli$^2$}

\noindent{\small $^1$University of Warwick, UK\\
    $^2$Technical University of Vienna, Austria\\}
\noindent\email{giuseppe.cannizzaro@warwick.ac.uk,
quentin.moulard@tuwien.ac.at, \\
fabio.toninelli@tuwien.ac.at}
\newline

\bigskip\noindent
\begin{abstract}
  The Stochastic Burgers Equation (SBE) is a
  singular, non-linear Stochastic Partial Differential Equation (SPDE)
  that describes, on mesoscopic scales, the fluctuations of stochastic driven diffusive
  systems with a conserved scalar quantity. In space dimension $d=2$, the SBE is critical, being formally scale invariant under diffusive scaling. As such, it falls
  outside of the domain of applicability of the theories of Regularity
  Structures and paracontrolled calculus. 
  In apparent contrast with the formal scale invariance, we fully prove the conjecture first appeared in 
  [H. van Beijeren, R. Kutner, \& H. Spohn, 
  Phys. Rev. Lett., 1986]
  according to which the $2d$-SBE is
  logarithmically superdiffusive, i.e. its diffusion coefficient diverges like $(\log t)^{2/3}$ as $t\to\infty$, 
  thus removing subleading diverging multiplicative corrections in [D. De Gaspari \& L. Haunschmid-Sibitz, 
Electron. J. Probab., 2024] and in [H.-T. Yau, 
 Ann.
of Math., 2004] for $2d$-ASEP. We precisely identify the constant prefactor of the logarithm and show it 
is proportional to $\lambda^{4/3}$, for $\lambda>0$ the coupling constant, which, intriguingly, 
turns out to be exactly the same as for the \emph{one-dimensional} Stochastic Burgers/KPZ equation. 
More importantly, we prove that, under super-diffusive space-time rescaling, 
the SBE has an explicit Gaussian fixed point 
in the Renormalization Group sense, by deriving 
a superdiffusive central limit-type theorem for its solution. 

This is the first scaling limit result for a critical singular SPDE, beyond the weak coupling regime, 
and is obtained via a refined control, on all length-scales, 
of the resolvent of the generator of the SBE. We believe our methods 
are well-suited to study  other out-of-equilibrium driven diffusive systems at the critical dimension, 
such as $2d$-ASEP, which, we conjecture, have the same large-scale Fixed Point as SBE. 



\end{abstract}

\bigskip


\setcounter{tocdepth}{3}       
\tableofcontents

\section{Introduction}

In  recent years, the interest in singular Stochastic Partial
Differential Equations (SPDEs) has grown dramatically and the 
field has witnessed spectacular progress. Virtually any \emph{subcritical equation}, 
i.e.  whose non-linearity becomes weaker when \emph{zooming in at small scales}, 
can be given a rigorous path-wise meaning via  
the theory of Regularity Structures~\cite{Hai}, paracontrolled calculus~\cite{Para}, 
or the Flow approach~\cite{Kuppi, Duch}\footnote{Let us also mention the energy solution approach of~\cite{GJfirst,GPunique,GPP} which is though not path-wise and needs the explicit knowledge of the invariant measure, but is less technically demanding and is better suited to determine scaling-limits of interacting particle systems.}. Notable examples include the dynamical 
$\PHI^4_d$ model for $d<4$~\cite{Hai}, the Yang-Mills-Higgs equation in $d=2,3$~\cite{YM2d,YM} 
and the KPZ equation~\cite{HaiKPZ, GPKPZ}, just to mention a few. 
On the other extreme, the understanding of {\it supercritical equations}, 
i.e. whose non-linearity becomes weaker when \emph{zooming out at large scales}, 
is much less systematic but some progress has been recently made. 
The large-scale behaviour for the (non-)linear multiplicative stochastic heat equation~\cite{GRZ, GL, GHL}, 
the KPZ equation~\cite{CCM2,dunlap2020,CNN, LZ} and the 
Stochastic Burgers equation~\cite{CGT}, all in dimension $d\geq 3$, 
has been identified and convergence 
to the linear Stochastic Heat Equation (SHE) under diffusive scaling, proved. 

In contrast, \emph{critical SPDEs}, i.e. those which are formally scale-invariant, 
meaning that the non-linearity is left unchanged when {\it zooming either in or out}, 
are way less understood. 
Paradigmatic examples in dimension $d=2$ are the $2d$ KPZ equation, the Anisotropic KPZ equation 
and the $2d$ Stochastic Burgers Equation (SBE). While, in view of scale-invariance, naive 
scaling arguments fail to predict whether the effect of
non-linearity is relevant or not, non-rigorous
Renormalization Group techniques suggest that at large scales
their solution approaches a ``fixed point'', whose nature and scaling
exponents depend crucially on the symmetries of the equation at hand. 
For instance, for AKPZ and SBE,~\cite{W91} and~\cite{van1985excess} respectively 
predict that the diffusive scaling exponents should be modified by logarithmic
corrections but a Gaussian fixed point is still expected, while for the $2d$
KPZ equation the fixed point should be non-Gaussian~\cite{barabasi1995fractal} 
(and the scaling exponents are not known).  
Up to now, however, we are not aware of any singular, critical SPDE 
whose large-scale behaviour has been rigorously identified, 
except in the \emph{weak coupling regime}, where  the strength of the non-linearity
is artificially tuned and is forced to vanish as the scale of the system grows (see Section
\ref{sec:previous} for a discussion). 

In this work, we obtain the\emph{ first full scaling limit result} for a singular critical SPDE
 \emph{beyond} the weak coupling regime: we show that the $2d$ SBE 
 (with coupling constant of order $1$) 
converges  \emph{under logarithmically superdiffusive scaling}, to a 
linear Stochastic Heat Equation with explicit renormalised coefficients.  
\newline 

Before precisely stating our theorems, let us introduce the equation of interest 
and put it into context. In any dimension $d\geq 1$, 
the SBE  is formally defined on $\R_+ \times \R^d$ by
\begin{equ} \label{e:FormalSPDE}
\partial_t u = \frac{1}{2} \Delta u + \lambda \, (\mathfrak w \cdot \nabla) \, u^2 + \nabla \cdot \vec{\xi} \, ,
\end{equ}
where $\vec{\xi} = (\xi_1,\dots, \xi_d)$ is a d-dimensional vector-valued space-time white noise, 
and $\lambda > 0$ is the so-called \emph{coupling constant}. Here, $\fw$ is 
a unit vector in $\mathbb R^d$, but, by rotational invariance, the specific 
choice of $\fw$ is irrelevant.

The SBE was introduced in \cite{van1985excess} as an effective continuous model 
for the mesoscopic
fluctuations of driven diffusive systems with a single conserved
scalar quantity, such as the Asymmetric Simple Exclusion Process (ASEP),
the Zero-Range Process with asymmetric rates and other driven,
conservative, interacting particle systems. In fact, the three terms
at the right-hand side of \eqref{e:FormalSPDE} represent
respectively a diffusive term, the non-linear effect of the driving
field acting in direction $\fw$, and a microscopic noise of
conservative type. On the basis of the so-called mode-coupling
theory, \cite{van1985excess} predicts the large-scale behaviour of SBE in any dimension $d\geq 1$: 
in the subcritical dimension $d=1$, the authors conjecture strong superdiffusion
with a $t^{1/3}$ divergence of the
diffusion coefficient for large time $t$; in the critical dimension $d=2$, 
logarithmic enhancement of diffusion of order $(\log t)^{2/3}$; 
and, in the super-critical dimensions $d\geq 3$, classical 
diffusive behaviour. Let us point out that in $d=1$, the SBE is nothing but the 
space derivative of the KPZ equation which, by now, is mathematically well-understood (see~\cite{QS}, for a review): 
not only the $t^{1/3}$ behaviour was confirmed~\cite{BQS}, but the solution was shown 
to have a non-Gaussian and universal fixed point (the ``KPZ fixed
point'' \cite{KPZFP}), that also describes the large-scale behaviour of a large
class of one-dimensional driven diffusive processes such as ASEP~\cite{QS1} and
directed polymers in random environment~\cite{Virag}. 
In dimension $d\ge3$, the series of works \cite{LY, CLO,LOV} has
confirmed the conjectured normal diffusion of ASEP and the convergence
of the particle density fluctuations to a linear Stochastic Heat Equation at large
scales. A counterpart on the SPDE side is the recent \cite{CGT}, where
the same result is obtained for the SBE in $d\ge3$.
As for the critical dimension $d=2$, the breakthrough work \cite{Yau}
proved $(\log t)^{2/3}$ diffusivity for ASEP, while in~\cite{de2024log} 
the analogous was shown for the SBE. Let us stress though 
that in both cases, the upper and lower bounds on the diffusivity 
feature {\it diverging subleading corrections} and their statement is phrased 
in the {\it Tauberian sense}, i.e. in terms of Laplace transform.  
\medskip

The problem when analysing the SPDE \eqref{e:FormalSPDE} is that it is 
severely ill-defined due to the presence 
of the non-linearity. Since dimension $d=2$ is critical, it falls outside of the scope 
of any of the path-wise theories mentioned above which successfully addressed 
subcritical equations and, 
in fact, a local solution theory is not even 
expected.\footnote{In a direction orthogonal to ours, the noteworthy work~\cite{GPP} determines a local solution theory for a scaling-critical SPDE, the fractional SBE in $d=1$ but their approach does not apply to $d=2$ and, contrary to ours, their focus in on small-scales.} 
Our main results are the following. First, we determine the
\emph{sharp large-time asymptotic behaviour of the diffusion matrix}, 
including the explicit constant prefactor, thus  
significantly improving over both~\cite{Yau,de2024log} and 
fully addressing the long-standing conjecture of~\cite{van1985excess}. 
Second, we move well beyond that and obtain the (Gaussian) \emph{large-scale fixed point of the regularised SBE}, when space-time is rescaled in a
suitable, logarithmically superdiffusive way. 
\medskip

In addition to being
the first such result for critical SPDEs, our work opens intriguing
perspectives for other driven diffusive systems, including  interacting particle systems, as we will discuss more in detail at the end of Section \ref{sec:previous}.
In a broader perspective, this paper can be seen 
as an out-of-equilibrium counterpart of the celebrated work \cite{celebrated} 
that proves that the large-scale  limit of the Ising and $\PHI^4_d$ equilibrium measures, 
in the critical dimension $d=4$, is Gaussian, despite the presence of logarithmic corrections 
to the mean field critical exponents.

\subsection{The main results: asymptotic superdiffusion and  large-scale limit}

To give a meaning to the formal equation \eqref{e:FormalSPDE}, 
we regularise it with a smooth compactly supported, non-negative, even-symmetric
function $\rho$
whose mass equals $1$. Then, we consider 
\begin{equ} \label{e:SPDE1Strong}
\partial_t u = \frac{1}{2} \Delta u + \lambda \, \cN^1(u) + \nabla \cdot \vec{\xi}^{\,1} \, ,
\end{equ}
where $u=u(t, x)$, for $(t,x)\in\R_+\times\R^2$, 
$\vec{\xi}^{\,1} \eqdef \rho \ast \vec{\xi}\eqdef(\rho\ast \xi_1, \rho\ast\xi_2)$, and
\begin{equ}
  [e:NonLin1]
  \cN^1(f)\eqdef (\fw \cdot \nabla) \, \rho^{\ast 2} \ast f^2 \, .
\end{equ}
Note that, thanks to the convolution(s), the noise $\vec{\xi}^1$ is smooth 
and, as we will see in Proposition~\ref{p:Global}, the dynamics is globally well-posed 
and the solution $u$ smooth in space. Moreover, it is a Markov process 
(with respect to a suitable filtration) and stationary. 
Indeed, a fundamental property of~\eqref{e:FormalSPDE} is that, formally, it admits the law of a
spatial white noise $\eta$ on $\R^2$, denoted by $\mathbb P$, as a
stationary measure, which is the continuum analog of the
stationarity of Bernoulli i.i.d. measures for ASEP.
The regularisation in~\eqref{e:SPDE1Strong} has been chosen similarly 
to~\cite{FQ} (which is, however,  in $d=1$), in such a way that stationarity is 
preserved, which means that the law of the mollified space white noise 
$\eta^1 \eqdef \rho \ast \eta$, denoted by 
$\mathbb P^\rho$, is stationary for the solution to~\eqref{e:SPDE1Strong}. 
%
%

In order to study the large-scale behaviour of $u$, a convenient tool 
is the so-called diffusion matrix, which, for $t \geq 0$, is the $2 \times 2$ matrix whose entries are given by  
\begin{equ} \label{e:FormalDiffusivity}
D_{i,j}(t) = \frac{1}{ t} \int_{\R^2} x_i \, x_j \, \Big[S(t,x)-S(0,x)\Big] \, \mathrm{d}x \, ,
\quad S(t,x) \eqdef \mathbf{E}\big(u(0,0) \, u(t,x)\big) \,,
\end{equ}
with $\mathbf E$ denoting the expectation with respect to the product measure of $\eta^1$ 
and the space-time noise $\vec{\xi}^1$, and $u$ being the stationary solution of~\eqref{e:SPDE1Strong} 
(i.e. $u_0$ has the same law as $\eta^1$). 
Intuitively, if $e \in \R^2$ is an unit vector, one can interpret $\langle D(t) \, e, e \rangle$ as the ratio
$(L_e(t))^2/t$, with $L_e(t)$ the time-dependent
correlation length in the direction $e$ at time $t$, defined as the distance beyond which the correlation
$S(t,x)$ is negligible. While it is useful to keep this interpretation in mind, we will not try to make this connection rigorous. 

The integral defining $D(t)$ 
is convergent because $x \mapsto S(t,x)$ decays fast enough as $|x| \to +\infty$, 
as is proven in Section \ref{sec:preliminaries}. Further, upon taking $\lambda=0$ in~\eqref{e:SPDE1Strong}, 
which reduces the SBE to the linear 
stochastic heat equation, one has 
\begin{equ} 
  S(t,x)|_{\lambda=0} = \Big(\rho^{\ast 2} \ast \mathcal P_t\Big)
  (x)
\end{equ}
with $\mathcal P_t$ the usual heat semigroup,
and therefore
\begin{equ} 
D(t)|_{\lambda=0} =
\begin{pmatrix}
1 & 0 \\
0 & 1
\end{pmatrix} \, .
\end{equ}
In particular, if for large $t$ the diffusion matrix of a given system has $O(1)$ entries, 
then its large-scale behaviour is said to be diffusive, while if it grows with $t$ it is {\it superdiffusive} 
(and if it decreases, subdiffusive).  
Even though the formal scale invariance of
\eqref{e:FormalSPDE} could at first sight suggest the former, this turns out not 
to be the case. Our first main theorem 
determines the precise asymptotic behaviour of $D(t)$ as $t\to\infty$. 

\begin{theorem} \label{th:Diffusivity}
  Let $\lambda > 0$ and, without loss of generality, take $\fw=\mathfrak e_1$, where $(\fe_1, \fe_2)$ stands for the canonical basis of $\R^2$.
  Then as $t \to +\infty$ it holds
\begin{equ} \label{e:Diffusivity}
D(t) = \begin{pmatrix}
C_\eff(\lambda) \, (\log t)^{2/3} \, \big(1 + o(1)\big) & 0 \\
0 & 1
\end{pmatrix} \, ,
\end{equ}
where the constant $C_\eff(\lambda)$ is explicit and given by
\begin{equ}[e:EffConstant]
C_\eff(\lambda) \eqdef \Big(\dis \frac{3}{2 \pi}\Big)^{2/3} \, \lambda^{4/3}\,.
\end{equ}
\end{theorem}


As mentioned above, Theorem \ref{th:Diffusivity} substantially improves on the previously
known bounds \cite{de2024log} (and the analog for $2d$ ASEP in~\cite{Yau}). Our result is for $D$ itself and not 
for the Laplace transform of $t\mapsto t D(t)$, as was the case in the above-mentioned 
references. More importantly, it eliminates spurious
multiplicative, subdominant corrections of order $(\log\log\log t)^C$
in~\cite{de2024log} (respectively, of order $\exp[(\log\log \log t)^2]$ for
ASEP in \cite{Yau}) and it identifies the precise constant
$C_\eff(\lambda)$. Further, the ``$o(1)$'' in the first entry is explicit and our estimates actually 
show that $D_{1,1}(t) = C_\eff(\lambda) \log(t)^{2/3} + O(\log(t)^{1/2 + o(1)})$ (see \eqref{e:DiffFinalEstimate}). 

Let us also underline two important implications of~\eqref{e:Diffusivity} and~\eqref{e:EffConstant}. 
First, the effective diffusivity in direction $\fw$ is {\it independent of the 
microscopic diffusivity in the same direction}, 
which means that it is {\it fully produced by the non-linearity}. 
To see this, take $\fw=\fe_1$, $\nu_1,\nu_2>0$ and $Q$ 
to be the $2\times2$ diagonal matrix such that $Q_{i,j}=\nu_i\delta_{i,j}$, $i,j=1,2$. 
Consider the solution $v$ of 
\begin{equ}
\partial_t v = \frac{1}{2} \, \nabla \cdot Q \, \nabla v + \lambda \, \tilde\cN^1(v) + \nabla \cdot \sqrt{Q} \, \vec{\xi}^1
\end{equ}
where $\tilde\cN^1$ coincides with $\cN^1$ in~\eqref{e:NonLin1} but 
with the mollifier $\rho$ replaced by $(\nu_1\nu_2)^{-1/2}\rho(Q^{-1/2}\cdot)$. 
Then, easy scaling arguments show that $u(t,x)\eqlaw(\nu_1\nu_2)^{1/4}v(t, \sqrt{Q} x)$, 
for $u$ the solution to~\eqref{e:SPDE1Strong} with $\lambda$ replaced by $\lambda\nu_1^{-3/4}\nu_2^{-1/4}$, 
and the diffusion matrix associated to $v$, $D^v$ satisfies
\begin{equ}[e:ScalingDiffMat]
D^v_{i,j}(t)=(\nu_i\nu_j)^{1/2} D_{i,j}(t)\,.
\end{equ} 
Therefore, as $t\to\infty$,~\eqref{e:Diffusivity} and~\eqref{e:EffConstant} imply 
that $D^v_{1,2}(t)=D^v_{2,1}(t)=0$, $D^v_{2,2}=\nu_2$ and especially that 
$D^v_{1,1}(t)=C(\nu_1,\nu_2,\lambda)(\log t)^{2/3}(1+o(1))$ with 
\begin{equ}
C(\nu_1,\nu_2,\lambda)=\nu_2^{-1/3}C_\eff(\lambda)
\end{equ}
so that overall $D^v$ is {\it independent of $\nu_1$}. 

\begin{remark}\label{rem:scaling}
The scaling argument above 
can be reversed, to show that the lack of dependence of the diffusion matrix on the parameter $\nu_1$ forces the
 $\lambda^{4/3}$ dependence of $C_\eff(\lambda)$. In fact, by~\eqref{e:ScalingDiffMat}, 
 $C(\nu_1,\nu_2,\lambda)=\nu_1 \, C(1,1,\lambda \, \nu_1^{-3/4}\nu_2^{-1/4})$, 
which means that, assuming $C(\nu_1,\nu_2,\lambda)$  is independent of $\nu_1$, we get that 
$C_\eff(\lambda)=C(1,1,\lambda)$ is proportional to $\lambda^{4/3}$.

The proportionality $C_\eff(\lambda)\propto \lambda^{4/3}$ is quite
intriguing, especially because  exactly the same holds
for the \emph{one-dimensional} Stochastic Burgers equation\footnote{This follows immediately from elementary scaling properties of the 1d SBE, plus the fact that its diffusion coefficent grows proportionally to  the cube root of time}, despite the fact that the $t$-dependence of the diffusion coefficients are
completely different in the two cases (of order $t^{1/3}$ in one
dimension, in contrast with $(\log t)^{2/3}$ in two dimensions). 
The equality of the exponents of $\lambda$ in $d=1,2$ is \emph{not} a
coincidence, but rather a reflection of the fact
that the diffusion coefficient is asymptotically (for large $t$)
independent of the microscopic diffusion coefficient in the
direction where the non-linearity acts. 
\end{remark}

The second implication of Theorem~\ref{th:Diffusivity} is the {\it precise identification 
of the scaling} under which one can expect~\eqref{e:SPDE1Strong} to approach its Fixed Point.  
If one naively tries the classical diffusive rescaling, i.e. 
$u_{\mathrm{diff}}^{\tau}(t,x)\eqdef\tau^{1/2} \, u(\tau \, t, \tau^{1/2} \, x)$, 
then an argument similar to that used to deduce~\eqref{e:ScalingDiffMat}, 
shows that $D_{\mathrm{diff}}^{\tau}(t) = D(\tau t)$ which, for $t>0$ fixed and $\tau\to\infty$, 
blows up as $(\log\tau)^{2/3}$ by~\eqref{e:Diffusivity}. 
Thus, the diffusive scaling must be adjusted logarithmically 
and the above statement strongly suggests 
that one should instead look at
\begin{equ} \label{e:GoodScaling}
u^\tau(t,x) \eqdef \tau^{1/2} (\log\tau)^{1/6} \, u(\tau \, t, \tau^{1/2} R_\tau^{-1/2} \, x)  \, , \quad R_\tau \eqdef \begin{pmatrix}
(\log\tau)^{-2/3} & 0 \\
0 & 1
\end{pmatrix} \, ,
\end{equ}
as, under such scaling, the entries of the diffusion matrix are asymptotically $O(1)$ (see~\eqref{e:DiffRescaling}). 
In principle, Theorem~\eqref{th:Diffusivity} does not fix the size of the fluctuations, i.e. 
the prefactor in the definition of $u^\tau$, but this needs to be chosen  
in such a way that the law $\mathbb P^\tau$ of the mollified white noise $\rho_\tau \ast \eta$, 
with $\rho_\tau(\cdot) \eqdef \tau (\log\tau)^{1/3} \, \rho(\tau^{1/2} R_\tau^{-1/2} \cdot)$,
is a stationary measure for $u^\tau$. 

The equation solved by $u^\tau$ then becomes 
\begin{equ} \label{e:ScaledSPDE}
  \partial_t u^\tau = \frac{1}{2} \, \nabla \cdot R_\tau \, \nabla u + \frac{\lambda}{\sqrt{\log\tau}} \, \cN^\tau(u^\tau) + \nabla \cdot \sqrt{R_\tau} \, \vec{\xi}^\tau \, ,
\end{equ}
with $\vec{\xi}^{\tau} \overset{\mathrm{law}}{=} \rho_\tau \ast \vec{\xi}$ and 
$\cN^\tau(f) \eqdef \cN^{\rho_\tau}(f)$, see~\eqref{e:NonLin1}. 
Note that, for $\tau\to\infty$, $\rho_\tau$ is converging to a $\delta$-function so that, morally, 
identifying the large-scale limit of $u^\tau$  coincides with making sense of 
the singular critical SPDE~\eqref{e:FormalSPDE}. 

Let us also stress that all the terms at the right hand side of~\eqref{e:ScaledSPDE} 
describing the dynamics in the direction $\fe_2$, i.e. the Laplacian 
$\partial_2^2 u^\tau$ and the noise $\partial_2\xi^\tau_2$, 
are unaffected by the scaling, meaning that their coefficients are constant in $\tau$, 
and those in $\fe_1$, 
i.e. $\partial_1^2 u^\tau$, the non-linearity and $\partial_1\xi^\tau_1$, 
are all multiplied by $\tau$-dependent quantities that vanish as $\tau\to\infty$. 
The main challenge we have to face is that,  even if the 
vanishing coefficients should tame the growth caused by the non-linearity, 
they also lessen the regularising properties 
of the linear part of the equation, making it apriori unclear what to expect in the limit. 

This is precisely what is addressed by our second main result, which more broadly provides 
a full \emph{central limit theorem for $u^\tau$}. It shows that as $\tau\to\infty$, 
$u^\tau$ converges in finite-dimensional distributions to a Gaussian process, 
given by the solution of an {\it anisotropic linear Stochastic Heat Equation} with {\it explicit} 
{\it renormalised} coefficients. 

\begin{theorem}\label{th:cor}
Let $\Q$ be a probability measure, absolutely continuous with respect to $\mathbb P$ and 
for $\tau>0$, let $(u^\tau(t))_{t \geq 0}$ be the solution of \eqref{e:ScaledSPDE} 
with $\fw=\fe_1$ and initial condition $u^\tau(0) =  \rho_\tau \ast u_0$ for $u_0 \sim \mathbb Q$. 
Then, for any $k\in \mathbb N$, $0\leq t_1\le \dots \le t_k$ and $\phi_1, ..., \phi_k \in \cS(\R^2)$, 
\begin{equ}
\big(\langle u^\tau(t_1), \phi_1 \rangle,\cdots, \langle u^\tau(t_k), \phi_k \rangle\big)\xrightarrow[]{\tau \to +\infty}\big(\langle u^\eff(t_1), \phi_1 \rangle,\cdots, \langle u^\eff(t_k), \phi_k \rangle\big)
\end{equ}
where the convergence is in law and $(u^\eff(t))_{t \geq 0}$ is the process that starts at 
$u^\eff(0)=u_0$ and solves the anisotropic linear Stochastic Heat Equation 
given by 
\begin{equ}[e:EffectiveSBE]
\partial_t u^{\eff} = \frac{1}{2} \, \nabla \cdot D_{\eff} \, \nabla u^\eff + \nabla \cdot \sqrt{D_{\eff}} \, \vec{\xi} \, , \quad\text{for}\quad D_{\eff} \eqdef \begin{pmatrix}
C_\eff(\lambda) & 0 \\
0 & 1
\end{pmatrix} \, , 
\end{equ}
where $\vec \xi$ is a space-time white noise 
and the constant $C_\eff(\lambda)$ is that in \eqref{e:Diffusivity}.
\end{theorem}

The previous theorem fully characterises the large-scale behaviour of~\eqref{e:SPDE1Strong} 
and, more generally, it strongly indicates what to expect for other out-of-equilibrium driven diffusive systems 
with one conserved quantity at the critical dimension. Phenomenologically, it says that 
in the $\fe_2$ direction the original diffusion is unaltered (with the mollifier replaced by the 
$\delta$ function), while, in the $\fe_1$ direction, the linear dynamics vanishes 
and the non-linear one produces {\it a new noise and a new Laplacian} whose strength 
explicitly depends on the coupling constant $\lambda$, i.e. on the microscopic strength of the non-linearity. 

Even though the creation of new Laplacian/noise from the non-linearity has already been observed in the 
context of critical SPDEs (e.g.~\cite{CSZ} for KPZ,~\cite{AKPZweak} for AKPZ or, more explicitly,~\cite{CC} 
for the multiplicative SHE), 
this is the first time, up to the authors' knowledge, where this is coupled with a vanishing microscopic 
diffusivity.
\medskip

Theorem~\ref{th:cor} is a corollary of Theorem \ref{MT} below, 
which concerns the convergence of the semigroup of the solution $u^\tau$ of~\eqref{e:ScaledSPDE} 
to that of $u^\eff$.  
To properly state it, we need a few preliminary definitions. 
Let $\mathbf P_{u_0}^\tau$ and $\mathbf P_{u_0}^\eff$ be  
the laws (with respective expectations $\mathbf E_{u_0}^\tau$ and $\mathbf E_{u_0}^\eff$) 
of  $u^\tau$ and $u^\eff$ given their initial condition $u_0^\tau=u_0\ast\rho_\tau$ and $u_0$, 
and denote by $(P_t^\tau)_{t \geq 0}$ and $(P_t^{\mathrm{\eff}})_{t \geq 0}$ 
their Markov semigroups. 

Our analysis of $P^\tau$ and $P^\eff$ will be carried out on the spaces 
$\L^\theta(\P^\tau)$ (to deal with $u^\tau$) and 
$\L^\theta(\P)$ (to deal with $u^{\eff}$), for $\theta\in[1,\infty)$, 
where we recall that $\P$ and
$\P^\tau$ stand for the law of the spatial white noise $\eta$
and its mollified counterpart $\eta^\tau=\rho_\tau \ast \eta$. 
To relate $\mathbb L^\theta(\mathbb P^\tau)$ and 
$\mathbb L^\theta(\mathbb P)$, we define the isometric embedding  
$\iota_\tau \colon \L^\theta(\P^\tau) \to \L^\theta(\P)$ 
which sends the functional $\big(\eta^\tau \to F(\eta^\tau)\big)$ 
to the functional $\big(\eta \to F(\rho_\tau \ast \eta)\big)$ 
(for more on it see Lemma~\ref{l:Emb} below). 
With a slight abuse of notation, we will say that a sequence 
$(F^\tau)_{\tau > 0}$, with $F^\tau\in \L^\theta(\P^\tau)$, converges 
to $F\in\L^\theta(\P)$ if $\iota_\tau F^\tau \to F$ in $\L^\theta(\P)$ as $\tau \to +\infty$.

\begin{theorem} \label{MT}
For any $\theta \in [1, +\infty)$ and sequence $(F^\tau)_{\tau > 0}$ converging to $F$ in $\L^\theta(\P)$, 
$(P_t^\tau F^\tau)_{\tau > 0}$ converges to $P_t^{\mathrm{eff}} F$ in $\L^\theta(\P)$ 
uniformly over $t \geq 0$, that is to say
\begin{equ} \label{e:StrongCv}
\sup_{t \geq 0} \, \E\Big[\big|\iota_\tau P^{\tau}_t F^\tau - P^{\mathrm{eff}}_t F \big|^\theta\Big]\xrightarrow[]{\tau \to +\infty}  0 \, .
\end{equ}
As a consequence, for every $k \geq 1$, $(F_1^\tau)_{\tau > 0}, ..., (F_k^\tau)_{\tau > 0}$ converging to $F_1, ..., F_k$ in $\mathbb L^\theta(\mathbb P)$ for all $\theta \in [1, +\infty)$, and $0 \leq t_1 \leq ... \leq t_k$, we have 
\begin{equ} \label{e:FinDimCv}
\Exp^\tau_{u_0}\big[F_1^\tau\big(u^{\tau}(t_1)\big) \, ... \, F_k^\tau \big(u^{\tau}(t_k)\big)\big] \xrightarrow[]{\tau \to +\infty} \Exp^{\mathrm{eff}}_{u_0}\big[F_1\big(u^{\mathrm{eff}}(t_1) \big) \, ... \, F_k\big(u^{\mathrm{eff}}(t_k)\big)\big] \, ,
\end{equ}
where the convergence holds in  $\mathbb L^\theta\big(\mathbb P(\dd u_0)\big)$ for every $\theta \in [1, +\infty)$.
\end{theorem}

The above theorem is substantially stronger than Theorem \ref{th:cor}. 
In fact, \eqref{e:FinDimCv} says that the finite
dimensional distributions of the process $u^\tau$ converge to those of
the limit process not just in average, but in
$\L^\theta(\P)$ with respect to the initial
condition. In this sense, the convergence stated in~\eqref{e:FinDimCv} is stronger than both 
annealed and stable convergence (and that in~\cite[Definition 2.6]{KLO}), 
and is closer in spirit to a quenched central limit theorem, 
although we make no claim of almost-sure convergence. 
Note also that, because of the supremum in time in \eqref{e:StrongCv} due to the ergodicity of the limit process, 
we have a uniform control of the law of the process $u^\tau$, {\it over all time scales}, 
which rules out pathological deviations from the scaling limit. 
Moreover, while we formulated Theorem \ref{MT} for $\lambda$ of order $1$, 
this being the most interesting scenario, 
our techniques would allow to obtain the scaling limit of the process 
in the whole {\it crossover window} where $\lambda=\lambda(\tau)$ 
tends to $0$ with $\tau$ in any desired way, 
between the weak-coupling regime (see next section) and the case $\lambda=O(1)$. 
\medskip

The proof of Theorems~\ref{th:Diffusivity} and~\ref{MT} relies on a careful analysis 
of the resolvent of $u^\tau$. We will present the overall strategy and 
provide a thorough account of the novelties of the present work
in Section~\ref{sec:HeuRM}, 
after deriving the main properties of $u^\tau$ for fixed $\tau$ and studying the action of its 
generator $\gen$ on $\L^2(\P^\tau)$. For now, let us only point out that  
our arguments give a \emph{scale-dependent control of the effective diffusion operator}, in the
spirit of Renormalization Group. 
Indeed, what our proof shows is that, for $F\in\L^2(\P^\tau)$ sufficiently regular 
(think of $F=\eta(\phi)$ for some $\phi$ smooth), $(1-\gen)^{-1}F$ 
approaches $(1-\cD^\tau_\eff)^{-1}F$, which is the resolvent of the 
solution to the non-homogeneous Stochastic Heat Equation given by
\begin{equ}
\partial_t u_\eff^\tau=\tfrac12 \nabla\cdot D_\eff^\tau\nabla u_\eff^\tau + \nabla \cdot \sqrt{D^\tau_\eff}\vec{\xi}
\end{equ}
and $D_\eff^\tau$ is the operator whose Fourier multiplier is 
\begin{equ}[e:Diffk]
D_\eff^\tau(k)=\begin{pmatrix}
\dis (\log\tau)^{-2/3}+g^\tau\big(L^\tau(\tfrac12|\sqrt{R_\tau} k|^2)\big) & 0 \\
0 & 1
\end{pmatrix} \, .
\end{equ}
Above, $R^\tau$ is the matrix in~\eqref{e:GoodScaling}, $L^\tau$ is a function 
proportional to $x\mapsto\log(1+\tau/x)/\log\tau$ (see~\eqref{e:Leps} for the precise definition) 
and, interestingly, $g^\tau$ is the solution of the ODE
\begin{equation}
  \label{eq:ode}
\dot{g}^\tau(x)=\frac1{\sqrt{\nu_\tau+g^\tau(x)}},\qquad g^\tau(0)=0\,. 
\end{equation}
Inspecting~\eqref{e:Diffk} we see that for momenta $k$ of order $1$ 
and $\tau$ large, one has $D^\tau_\eff(k)\approx D_\eff$ with the latter defined as in \eqref{e:EffectiveSBE}, 
which in turn implies the proximity of $u^\tau$ and $u^\eff$. 
But more than that, since $L^\tau(x)$ \emph{is small when $x$ is large}
and $g^\tau$ satisfies an \emph{initial value} problem,
the effective diffusivity $g^\tau\big(L^\tau(\tfrac12|\sqrt{R_\tau} k|^2)\big)$ 
at scale $k$ \emph{encodes the cumulative effect of momenta larger than $k$}, 
that is, \emph{of smaller spatial scales}. In other words, the ODE \eqref{eq:ode} corresponds 
to what  would be called in Renormalization Group a 
\emph{flow equation for the scale-dependent coefficient $g^\tau$}.

\subsection{Relation with the literature and future perspectives}\label{sec:previous}

As alluded to earlier, most of the work on critical SPDEs (or more general out-of-equilibrium critical systems) 
concerns the so-called {\it weak coupling} 
or {\it intermediate disorder} scaling, which, in a sense, provides a bridge 
between subcritical and critical equations and can be thought of as a way to zoom into 
the singular part of the equation at hand. 
More precisely, starting from a given SPDE (say~\eqref{e:SPDE1Strong}), 
this amounts to introduce  \emph{a vanishing coupling constant} multiplying 
the nonlinearity, the noise or the initial condition (for~\eqref{e:SPDE1Strong}, 
set $\lambda=\lambda_\tau = \bar\lambda/\sqrt{\log \tau}$), 
and then take the usual diffusive scaling (i.e. $u\mapsto \sqrt{\tau}u(\tau t, \sqrt\tau x)$). 
Even though there is currently no systematic way to analyse such setting, 
several interesting results in a variety of different contexts have recently appeared. 
This is the case e.g. for the KPZ equation~\cite{CD,Gu2020,CSZ}, 
the AKPZ equation~\cite{CES,AKPZweak}, the SBE~\cite{CGT}, (self-)interacting diffusions~\cite{CG,yang2024weak}
the Allen-Cahn equation with weakly critical random initial data~\cite{GRZallen}, 
and, especially, the multiplicative (non-)linear Stochastic Heat Equation in $d=2$ 
~\cite{BC, CSZ1, DG, Tao}.  The latter in particular culminated in the
ground-breaking work~\cite{caravenna2023critical}, which first
constructed the scaling limit of directed polymers in random
environment in the intermediate disorder regime at the critical value
of the disorder strength, a process that now goes under the name
of critical ``$2d$-Stochastic Heat Flow'' (see also~\cite{Tsai}, for a
more recent characterisation).  As already stressed, the present work
goes beyond weak coupling as the vanishing coefficients at the right
hand side of~\eqref{e:ScaledSPDE} are not put there ``by hand''
but 
come from simply {\it zooming out at the correct scale}, which for the
SBE corresponds to~\eqref{e:GoodScaling}. 

Another critical, out-of-equilibrium model which has attracted significant attention due to
its physical importance in connection with turbulence, is the DCGFF,
i.e.  a singular SDE describing the motion of a Brownian particle in
two dimensions, subject to a divergence-free drift given by the curl
of a log-correlated random field (e.g. the $2d$ GFF). The analysis of
its mean-square displacement was initiated in~\cite{TothValko} and
later continued first in~\cite{CHT}, using functional analytical
techniques similar in flavour to those in the present work, and then
in~\cite{Morfe}, via stochastic homogeneisation arguments, where the
conjectured annealed logarithmic superdiffusivity (see~\cite{TothValko}) was established. 
The remarkable recent work~\cite{armstrong}, which proves the
analog of Theorem~\ref{th:cor} in this context, obtains the full conjecture 
of~\cite{LeDou} as it derives a
quenched invariance principle under a logarithmically adjusted scaling
(that said, the tools exploited therein are completely different from
ours and it is unclear whether they could be adapted to the SPDE
setting).  \medskip

Ultimately, the major achievement of the present paper is that it establishes a rigorous framework in which 
the large-scale behaviour of stationary space-time systems at criticality can be analysed.  
Moreover, in relation to the literature, we believe it raises a number of intriguing questions:
\begin{itemize}[noitemsep]
\item the convergence, under similar logarithmically adjusted 
  scaling, of  the anisotropic self-repelling polymer of~\cite{TothValko}, of  the density fluctuation profile of    $2d$
  ASEP and of other critical interacting particle systems, to a non-trivial Gaussian limit process.  We are not aware
  of any rigorous statement or even explicit conjecture in this sense;

\item the investigation of the role and universality of the $\lambda^{4/3}$  dependence \eqref{e:EffConstant} of the diffusion coefficient on the coupling constant  (see Remark~\ref{rem:scaling}) 
for the above-mentioned systems at {\it and above} the critical dimension;
\item the study of large-scale behaviour of other critical SPDEs with (non-)quadratic non-linearity, 
as the AKPZ equation or the stochastic Navier-Stokes equation, and, in the long-run, 
away from stationarity. 
\end{itemize}
At last, a very fascinating line of works~\cite{OW,MOW} has displayed compelling features 
of the DCGFF such as intermittency properties and an unexpected connection with a diffusion on 
the Lie group ${\bf SL}(2)$, 
and it would be interesting to see if analogous relations can be found in the present context or even 
more broadly for other systems at the critical dimension.

\subsection*{Organization of the paper}

The rest of the paper is organised as follows. In subsection~\ref{sec:notation}, we recall 
basic tools from Wiener space analysis. 
Section \ref{sec:thereg}
states basic results of the SBE~\eqref{e:ScaledSPDE} for fixed value of $\tau>0$
such as global well-posedness, Markovianity and stationarity. Moreover, 
the two-point correlation function is shown to decay polynomially fast 
and a Green-Kubo formula for the diffusion matrix 
in~\eqref{e:FormalDiffusivity} is rigorously derived, which we believe is of independent 
interest. The semigroup and the generator of~\eqref{e:ScaledSPDE}, together with 
the action of the latter on Fock spaces, are 
analysed in Section~\ref{sec:generator} where we further recall 
basic properties of the limit equation~\eqref{e:EffectiveSBE}. In Section~\ref{sec:HeuRM}, 
we provide a high-level roadmap that on the one hand outlines the overall strategy of proofs, 
and should thus guide the reader throughout, 
while on the other introduces the main novelties of the paper. 
Some of these are then detailed in Section
\ref{sec:Est} such as: a new 
splitting of the antisymmetric part of the generator that allows to
single out the singularity in Fourier and growth in the ``Wiener
chaos index'' (Section~\ref{sec:Asharp}); a (refined version of the) 
Replacement Lemma (Lemma~\ref{p:1stReplLemma}) and 
the Recursive Replacement Lemma (Lemma~\ref{p:2ndReplLemma}). 
Section~\ref{sec:CVRes} applies these tools to derive the needed estimates 
on the resolvent, as summarised in its main statement Theorem~\ref{thm:MainSec4}. 
Theorems~\ref{th:cor} and~\ref{MT}, 
concerning the superdiffusive central limit theorem are proven in Section~\ref{sec:CLT}, 
while Theorem~\ref{th:Diffusivity} about the diffusion matrix in Section \ref{sec:D}. 
At last, Appendices~\ref{a:BasicsSBE} and~\ref{a:Basics} contain the 
proofs of the statements in Section~\ref{sec:preliminaries}, 
while technical bounds and a general replacement estimate (which we deem interesting in its own right) 
can be found in Appendix~\ref{s:TechnicalEstimate}.

\subsection*{Notations and Function spaces}

Let $\fe_1$ and $ \fe_2$ be the canonical basis vectors of $\R^2$ and $|\cdot|$ be
the usual Euclidean norm.
For $p_1, ..., p_n \in \R^2$, we set $p_{1:n} \eqdef (p_1, ..., p_n) \in \R^{2n}$ and,
given a $2 \times 2$ matrix $M$, $M p_{1:n}\eqdef(M p_1, ..., M p_n)$.
We write (for $e \in \R^2$)
\begin{equ} \label{e:NotationsR2n}
|p_{1:n}| \eqdef \Big(\sum_{i=1}^n|p_i|^2\Big)^\frac{1}{2}\,,\qquad |e \cdot p_{1:n}| \eqdef \Big(\sum_{i=1}^n |e \cdot p_i|^2\Big)^\frac{1}{2} \,.
\end{equ}
For $\tau>0$ let the matrix $R_\tau$ and $\nu_\tau > 0$ be respectively given according to
\begin{equ}[e:nutau]
R_\tau\eqdef\begin{pmatrix}
\dis \nu_\tau & 0 \\
0 & 1
\end{pmatrix}\qquad\text{and}\qquad\nu_\tau \eqdef \frac{1}{\big(1 \vee \log\tau\big)^{2/3}} \, ,
\end{equ}
and note that it holds
\begin{equ} \label{e:NotationsQ}
|\sqrt{R_\tau} p_{1:n}|^2 = \nu_\tau |\fe_1 \cdot p_{1:n}|^2 + |\fe_2 \cdot p_{1:n}|^2 \, .
\end{equ}

Let $\cS(\R^d)$ be the Schwartz space of smooth functions whose derivatives decay faster
then any polynomial and $\cS'(\R^d)$ be its dual. For $\phi\in\cS(\R^d)$, we define its Fourier transform by
\begin{equ}
\cF(\phi)(p) = \hat \phi(p) \eqdef \frac{1}{(2 \pi)^{d/2}} \int_{\R^n} \phi(x) \, e^{- \iota \, p \cdot x} \, \dd x \, ,
\end{equ}
while, for $f\in\cS'(\R^d)$, it is given via duality.

Throughout the paper we fix a mollifier $\rho\in\cS(\R^2)$ (but the results are independent of it) which
is smooth, of mass $1$, even and compactly supported in a ball of radius $1$. For $\tau>0$ and $n\in\N$,
we set
\begin{equ}[e:Mollifiers]
\rho_\tau(\cdot) \eqdef \tau \, \nu_\tau^{-1/2} \, \rho(\tau^{1/2} R_\tau^{-1/2} \cdot)\,,\quad\text{and}\quad
\Theta_\tau(\cdot) \eqdef (2 \pi \hat{\rho}_\tau(\cdot))^2=(2 \pi \hat{\rho}(\tau^{-1/2} R_\tau^{1/2} \cdot))^2\,,
\end{equ}
and define the measure $\Xi_n^\tau$ on $\R^{2n}$ according to
\begin{equ}[e:RegMeas]
\Xi_n^\tau(\dd p_{1:n}) \eqdef \prod_{i=1}^n \Theta_\tau(p_i) \, \dd p_{1:n} \, ,\qquad p_{1:n}\in\R^{2n}
\end{equ}
with the convention that, for $\tau = \infty$, $\Theta_\tau \equiv 1$ and $\Xi_n=\Xi_n^\infty$ is the Lebesgue measure.

For $\tau> 0$ and $\gamma\in\R$, we introduce the $\tau$-dependent
(anisotropic) Sobolev spaces with respect to $\Xi_n^\tau$,
which are given by the completion of $\cS(\R^{2n})$ under the (semi-)norm\footnote{As the zero-set of $\Theta_\tau$ is of empty interior since $\Theta_\tau$ is analytic, these are actually norms. }
\begin{equ} \label{e:normHeps}
\|f\|_{H^\gamma_\tau(\R^{2n})}^2 \eqdef \int \Big(1 + \frac{1}{2} \, |\sqrt{R_\tau} p_{1:n}|^2\Big)^\gamma \, |\hat f(p_{1:n})|^2 \, \Xi_n^\tau(\dd p_{1:n}) \, .
\end{equ}
For $\gamma = 0$, this corresponds to the space $L_\tau^2(\R^{2n}) \eqdef L^2(\R^{2n}, \Xi_n^\tau)$,
whose scalar product will be denoted by $\langle \cdot, \cdot \rangle_{L^2_\tau(\R^{2n})}$.
\medskip

Finally, we will write $a \lesssim b$ if there exists a constant $C > 0$ independent of any quantity relevant for the result,
such that $a \leq C b$ and $a \asymp b$ if $a \lesssim b$ and $b \lesssim a$.
If we want to highlight the dependence of the constant $C$ on a specific quantity $Q$, we write instead $\lesssim_Q$.

\subsection{Gaussian white noise and Fock spaces} 
\label{sec:notation}

Let $\eta$ be a space white noise on $\R^2$, i.e. an isonormal Gaussian process
on $L^2(\R^2)$ (see~\cite[Definition 1.1.1]{Nualart})
and, for $\tau\geq 0$, set $\eta^\tau\eqdef\rho_\tau\ast\eta$ (with $\eta^\infty\equiv\eta$).
Then, $\eta^\tau$ is a centred Gaussian field with covariance
\begin{equ}
\E^\tau[\eta^\tau(\phi)\eta^\tau(\psi)] =\langle\phi,\psi\rangle_{L_\tau^2(\R^2)} = \int_{\R^{2}}\hat \phi(p)\,\overline{\hat \psi(p)}\,\Xi_1^\tau(\dd p)\,,\qquad \phi,\psi\in L^2_\tau(\R^2)\,.
\end{equ}
We will denote its law by $\P^\tau$ and by $\L^\theta(\P^\tau)$ the space of $\theta$-integrable random variable
with respect to $\P^\tau$. A random variable $F\in\L^\theta(\P^\tau)$ is said to be {\it cylinder} if it is
of the form $F=f(\eta^\tau(\phi_1),\dots,\eta^\tau(\phi_n))$ for $f\colon \R^n\to\R$ smooth and
growing at most polynomially at infinity, and $\phi_1,\dots,\phi_n\in\cS(\R^2)$.  
The set of cylinder random variables is dense in
$\L^\theta(\P^\tau)$, for $\theta<\infty$.
\medskip

According to~\cite[Theorem 1.1.1]{Nualart},
$\L^2(\P^\tau)$ admits the orthogonal
decomposition $\L^2(\eta^\tau) = \overline{\bigoplus_{n \geq 0} \SH_n^\tau}$, where
$\SH_n^\tau$ is the {\it $n$-th homogeneous Wiener chaos}, i.e.
the closure in $\L^2(\P^\tau)$ of
$\mathrm{Span}\big\{H_n(\eta^\tau(h)) \, | \, h \in \cS(\R^2), \, \|h\|_{L_\tau^2(\R^2)} = 1\big\}$
and $H_n$ is the $n$-th Hermite polynomial.
Further,~\cite[Theorem 1.1.2]{Nualart} ensures the existence of a canonical isometry
$I^\tau$ onto $\L^2(\P^\tau)$
which is the unique extension of the map that, for every $n\geq 0$,
assigns $h^{\otimes n}$ to $n! \, H_n(\eta^\tau(h)) \in \cH^\tau_n$ for
$h\in\cS(\R^2)$ with $L^2_\tau(\R^2)$-norm equal to $1$.
The space on which $I^\tau$ is defined is the Fock space $\fock{}{}{\tau}$,
and is given by the closure of $\oplus_{n\geq 0} \fock{n}{}{\tau}$,
where $\fock{n}{}{\tau}$ is the subset of
$L^2_\tau(\R^{2n})$ made of functions which are symmetric with respect to
permutations of their variables, endowed with the norm
\begin{equ}[e:normFock]
\|f\|_\tau^2=\|f\|_{\fock{}{}{\tau}}^2\eqdef \sum_{n\geq 0} n! \|f_n\|_{L^2_\tau(\R^{2n})}^2\,,
\end{equ}
for $f=(f_n)_{n\geq 0}\in \oplus_{n\geq 0}\fock{n}{}{\tau}$.
Note that, by construction, the restriction $I^\tau_n$ of $I^\tau$ to $\fock{n}{}{\tau}$ is an isometry with values in
$\SH_n^\tau$ and, again by~\cite[Theorem 1.1.2]{Nualart}, for any $F\in\L^2(\P^\tau)$
there exists $f=(f_n)_n\in\fock{}{}{\tau}$ such that $F=\sum_n I_n(f_n)$ and $\E^\tau[F^2]=\|f\|_\tau^2$.
\medskip

The bulk of our analysis will focus on various operators acting on (subspaces of) $\L^2(\P^\tau)$.
In view of the isometry $I^\tau$, we will abuse notation and identify them
with the corresponding operator acting on $\fock{}{}{\tau}$.
An important notion is that of {\it diagonal} operator that we now  define.

\begin{definition}\label{def:DiagOp}
An operator $\cD$ on $\fock{}{}{\tau}$ is said to be {\it diagonal}
if there exists a family of measurable kernels $\mathbf{d} = (\bd_n)_{n \geq 0}$
such that for all $n\geq 0$ and $\phi\in\fock{n}{}{\tau}\cap\cS(\R^{2n})$, it holds
$\cF{\cD \phi} (p_{1:n}) = \bd_n(p_{1:n}) \, \widehat \phi(p_{1:n})$ for (Lebesgue-almost-) every $p_{1:n} \in \R^{2n}$.
If $\cD^{(1)}$ and $\cD^{(2)}$ are diagonal, we write $\cD^{(1)} \leq \cD^{(2)}$ if their kernels satisfy
$\bd^{(1)} \leq \bd^{(2)}$ and say that $\cD^{(1)}$ is positive if $\cD^{(1)} \geq 0$.
Given a diagonal operator $\cD$ with kernel $\bd=(\bd_n)_{n \geq 0}$
and a measurable scalar function $f$, we define $f(\cD)$ to be the diagonal operator
whose kernel is $f(\bd)=(f \circ \bd_n)_{n \geq 0}$.
\end{definition}


\begin{definition}\label{def:NoOp}
We define the number operator $\cN$ to be the diagonal operator
acting on $\psi=(\psi_n)_n\in\fock{}{}{\tau}$ such that $\|\psi_n\|_\tau$ decays faster than any polynomial in $n$,
as $\cN\psi_n\eqdef n\psi_n$.
\end{definition}

In order to discuss the domain and well-posedness of the operators we will encounter,
it is convenient to introduce, for $\gamma\in\R$, the Sobolev space
corresponding to $\fock{}{}{\tau}$, which we denote by $\fock{}{\gamma}{\tau}$
and whose norm is the same as that in~\eqref{e:normFock}
but with $\|\cdot\|_{H^\gamma_\tau(\R^{2n})}^2$ replacing $\|\cdot\|_{L^2_\tau(\R^{2n})}^2$.
\medskip

To conclude this preliminary section, let us 
clarify the relation between $\L^\theta(\P^\tau)$ and $\L^\theta(\P)=\L^\theta(\P^\infty)$, for $\theta\in[1,\infty)$.
Recall that the operator $\iota^\tau \colon \L^\theta(\P^\tau) \to \L^\theta(\P)$ which maps the functional
$\eta^\tau \mapsto F(\eta^\tau)$ to the functional $\eta \mapsto F(\rho_\tau * \eta)$, is an isometric embedding.
In the next basic lemma we show that $\iota^\tau$ is an isometry which admits an inverse $j^\tau$.

\begin{lemma}\label{l:Emb}
For every $\theta \in [1, +\infty)$, the map $\iota^\tau \colon \mathbb L^\theta(\mathbb P^\tau) \to \mathbb L^\theta(\mathbb P)$ is an isometry which admits an inverse, that we set to be $j^\tau \eqdef (\iota^\tau)^{-1}$.
Furthermore, for $\theta = 2$, $\iota^\tau$ and $j^\tau$ are respectively the operators
from $\fock{}{}{\tau}$ to $\fock{}{}{}=\fock{}{}{\infty}$ and viceversa,
that for every $\Phi\in\fock{n}{}{\tau}$, $\psi\in\fock{n}{}{}$,
$n\in\N$, satisfy
\begin{equ}[e:iotaeps]
\cF\big(\iota^\tau(\Phi)\big)(p_{1:n})=\hat\Phi(p_{1:n})\prod_{i=1}^n \Theta_\tau(p_i)^{1/2}\,,\qquad \cF\big(j^\tau(\psi)\big)(p_{1:n})=\frac{\hat\psi(p_{1:n})}{\prod_{i=1}^n \Theta_\tau(p_i)^{1/2}}\,.
\end{equ}
\end{lemma}
\begin{proof}
As $\iota^\tau$ is an isometric embedding, to show that it is an isometry, it suffices to prove that
it is surjective. To do so, we will build its inverse, i.e. a map $j^\tau \colon \L^\theta(\P) \to \L^\theta(\P^\tau)$
such that $\iota^\tau \circ j^\tau = \mathrm{Id}$.
First, let $(\rho_\tau *)^{-1} \colon L^2(\R^2) \to L^2_\tau(\R^2)$ be given as
\begin{equ}
\cF\big((\rho_\tau *)^{-1} f\big)(p) \eqdef \frac{\hat{f}(p)}{\Theta_\tau(p)^{1/2}} \, ,
\end{equ}
which is well defined, as $\rho$ is compactly supported, its Fourier transform is analytic
and thus the set of points at which $\Theta_\tau$ vanishes has zero Lebesgue measure.
Furthermore, $(\rho_\tau *)$ and $(\rho_\tau *)^{-1}$ are both isometries
respectively from $L^2(\R^2)$ to $L^2_\tau(\R^2)$ and viceversa, and are inverses of one another.
Now, note that for $\phi \in L^2(\R^2)$, $\eta(\phi) = \eta^\tau\big((\rho_\tau *)^{-1} \phi\big)$,
that is to say $\eta = \eta^\tau \circ (\rho_\tau *)^{-1}$. Hence, the map
$j^\tau \colon \L^\theta(\P) \ni F(\eta) \mapsto F(\eta^\tau \circ (\rho_\tau *)^{-1}) \in \L^\theta(\P^\tau)$
is also well-defined, is an isometry and it is the inverse of $\iota^\tau$.

At last,~\eqref{e:iotaeps} can be immediately deduced by the definitions of $(\rho_\tau\ast)^{-1}$ and $j^\tau$.
\end{proof}

\section{Preliminaries, novelties and roadmap}\label{sec:preliminaries}

In this section, we present some preliminary results concerning the
solution to the regularised scaled SBE in~\eqref{e:ScaledSPDE}.
The statements  are essential, since they represent the foundations on which our analysis relies,
but basic, in that their proof uses classical
techniques from PDE theory and functional analysis.
This is the reason why we preferred to postpone most of the arguments (some of which quite involved)
to Appendices~\ref{a:BasicsSBE} and~\ref{a:Basics}, and could be skipped at first read.

%

After studying basic properties of the solution (Section~\ref{sec:thereg}),
its semigroup and generator (Section~\ref{sec:generator}),
in Section~\ref{sec:HeuRM} we provide a more detailed heuristic as to how the proof of the main result works
and what are the novel ideas behind it.

Throughout this section, the regularisation parameter $\tau>0$ is {\it fixed} and, unless otherwise specified,
in the estimates herein no uniformity in its value is claimed (or true).

\subsection{The regularised stochastic Burgers equation} \label{sec:thereg}


To rigorously discuss the equation of interest, let $(\cO,\cF,\Prob)$ be a probability space
supporting a couple of independent spatial white noise $\eta$ on $\R^2$ and
$2$-dimensional space-time white noise $\vec{\xi}=(\xi_1,\xi_2)$ on $\R \times \R^2$,
whose laws are respectively denoted by $\P$ and $P$, and $\Prob \eqdef \P \otimes P$.
Let $(\cF_t)_{t \geq 0}$ be the filtration in which $\cF_t$ is the (completed)
$\sigma$-algebra generated by $\eta$ and $\restriction{{\vec \xi} \,}{(-\infty,t]}$.
Let $\rho\in\cS(\R^2)$ be the mollifier (i.e. smooth, even-symmetric, compactly supported in a ball of radius $1$ and
with total mass equal to $1$) introduced above
and, for $\tau>0$,
$\rho_\tau$ be as in~\eqref{e:Mollifiers}. 
Let $\eta^\tau \eqdef  \eta\ast\rho_\tau$ and $\vec{\xi}^{\tau} \eqdef \vec{\xi}\ast\rho_\tau=(\xi_1\ast\rho_\tau,\xi_2\ast\rho_\tau)$
the rescaled white noises, whose joint law is $\Prob^\tau \eqdef \P^\tau \otimes P^\tau$.
\medskip

Now, we want to study the SPDE in~\eqref{e:ScaledSPDE} given by
\begin{equ}\label{e:SBEsystem}
\partial_t u^\tau = \frac{1}{2} \, \nabla \cdot R_\tau \, \nabla u^\tau + \lambda \, \nu_\tau^{3/4} \, \cN^\tau[u^\tau] + \nabla \cdot \sqrt{R_\tau} \, \vec{\xi}^\tau \, ,\qquad u^\tau_0 = \eta^\tau\,,
\end{equ}
where $u^\tau=u^\tau(t,x)=u^\tau_t(x)$ for $(t,x)\in\R_+ \times \R^2$ and $\cN^\tau$ is given as in~\eqref{e:NonLin1} 
but with $\rho_\tau$ in place of $\rho$,
and $R_\tau$, $\nu_\tau$ as in~\eqref{e:nutau}.
To do so, we need to introduce a
space in which $u^\tau$ lives and a suitable notion of solution. The former is
$C^\beta_\loc(\R_+, \CC^\alpha(\p_a))$ for $\beta,\alpha,a>0$, i.e.
the space of locally $\beta$-H\"older continuous functions in time with
values in the weighted Besov-H\"older space $\CC^\alpha(\p_a)$ of regularity $\alpha$. The precise
definition of $\CC^\alpha(\p_a)$ will be given in Appendix~\ref{a:Besov}
but, for now, the reader can think of the space of $\alpha$-H\"older continuous functions
whose norm restricted to a ball of radius $x$ grows at most as $\p_{-a}(x)=(1+|x|)^a$ (in particular this is true
for all derivatives of order smaller than $\lfloor\alpha\rfloor$).
Concerning the notion of solution, it will be convenient for us to work with mild solutions.

\begin{definition}\label{def:SolRegBurgers}
We say that $u^\tau$ is a mild solution of~\eqref{e:SBEsystem} driven by $(\eta^\tau,\vec{\xi}^{\tau})$
if $u^\tau$ is $(\cF_t)_{t \geq 0}$-adapted,
it belongs to $C_\loc(\R_+,\CC^\alpha(\p_a))$ $\Prob^\tau$-a.s., and for every $t \geq 0$ it holds
\begin{equ}[e:SBEregMild]
u^\tau_t = K^\tau_t \eta^\tau + \lambda\nu_\tau^{3/4} \int_0^t K^\tau_{t-s} \, \cN^\tau[u^\tau_s] \, \dd s+\int_0^t K^\tau_{t-s} \, \nabla\cdot \, R_\tau\vec{\xi}^{\tau}(\dd s) \, ,
\end{equ}
where the last summand is to be interpreted as a stochastic convolution in the sense of Walsh~\cite{Walsh},
and $K^\tau_t\eqdef e^{\frac{t}{2}\nabla \cdot R_\tau \, \nabla}$ is the heat semigroup associated to $ \frac12\nabla \cdot R_\tau \, \nabla$.
\end{definition}

\begin{remark}\label{rem:MildvsWeak}
In previous works (and especially~\cite{CGT}), the notion of solution considered is that of
analytically weak solution, which is obtained by multiplying both sides of~\eqref{e:SBEsystem} by a
function $\phi$, integrating in time and space, and requiring the equality to hold for a sufficiently large
class of $\phi$'s. In general, the notion of mild solution is stronger than that of weak but,
given the regularity of the noise (and the initial condition) for $\tau$ fixed, the two can be easily shown to be equivalent
arguing as in~\cite[Theorem 5.4]{DPZ}.
\end{remark}

The next proposition summarises the main properties of $u^\tau$ we are concerned with, i.e.
global well-posedness, Markovianity and stationarity. See Appendix~\ref{a:SBE}, for the proof. 

\begin{proposition}\label{p:Global}
Let $\tau>0$ be fixed. For any $\alpha > 0$ and $a > 0$,
the regularised stochastic Burgers equation~\eqref{e:SBEsystem} driven by
$(\eta^\tau,\vec{\xi}^{\tau})$ admits a unique mild solution
whose paths belong to $C_{\mathrm{loc}}(\R_+,\CC^\alpha(\p_a))$
(actually to $C_{\mathrm{loc}}^\beta(\R_+,\CC^\alpha(\p_a))$
for any $\beta \in [0,1/2)$) and path-by-path uniqueness holds.
Moreover, for every $\alpha,a, p, T > 0$ and $\beta \in [0,1/2)$,
there exists a positive constant $C=C(\tau,\alpha,\beta,a, p, T)<\infty$ such that
\begin{equ}[e:apriori]
\Exp^\tau\Big[\|u^\tau\|^p_{C^\beta([0,T], \cC^\alpha(\p_{a}))}\Big] \leq C \, .
\end{equ}
At last, $(u^\tau_t)_{t \geq 0}$ defines a $(\cF_t)_{t \geq 0}$-Markov process on $\CC^\alpha(\p_a)$,
which is stationary, that is, for any $t\geq 0$, $u^\tau_t\eqlaw \eta^\tau$; and
skew-reversible, in the sense that, for any $T\geq 0$,
$\{u^\tau_{T-t}\colon t\in[0,T]\}\eqlaw\{v^\tau_t\colon t\in[0,T]\}$ where $v^\tau$
is the unique mild solution to~\eqref{e:SBEsystem} driven by
$(\tilde\eta^\tau,\vec{\tilde\xi}^{\tau})\eqlaw(\eta^\tau,\vec{\xi}^{\tau})$ with $\lambda$ replaced by $-\lambda$.
\end{proposition}

In light of the previous proposition, we not only know that~\eqref{e:SBEsystem}
admits a unique mild solution $u^\tau$ driven by $(\eta^\tau,\vec{\xi}^{\tau})$ but also that $u^\tau$
is smooth (and stationary).
Therefore, we can define its two-point correlation function $S^\tau\colon\R_+\times\R^2\to\R$ according to
\begin{equ}[e:Corr]
S^\tau_t(x)= \Exp^\tau[u^\tau_t(x) \, u^\tau_0(0)]\,,\qquad (t,x)\in\R_+\times\R^2\,.
\end{equ}
Note that, since $u^\tau$ is centred, $S^\tau$ is indeed the covariance of $u^\tau_t(x)$ and $u^\tau_0(0)$.
Further, it is smooth and its derivatives, which, for $\bk\in\N^{2}$, will be denoted by
$\partial_\bk^{|\bk|} S^\tau=\partial_1^{k_1}\partial_2^{k_2}S^\tau$
(with the convention that if $k_i=0$ then there is no derivative in the $i$-th direction),
can be obtained by commuting the expectation and the derivative operator
(this is allowed as, locally, the sup-norm of the derivatives of $u^\tau_t$ have finite moments of all orders
since, in law, $u^\tau_t$  equals $\eta^\tau$ by stationarity).

In the next proposition, we continue our analysis of $u^\tau$ and show that its correlations, encoded via $S^\tau$, and
their spatial derivatives decay in space faster than any polynomial.

\begin{proposition}
\label{p:corrdecay}
  Let $S^\tau$ be defined according to \eqref{e:Corr}. Then, for every $T, \gamma > 0$ and $\bk\in\N^{2}$,
  there exists a positive constant $C=C (\tau,\gamma, T, \bk) < \infty$ such that for every $t \leqslant T$  
  and $x\in \R^2$. 
  \begin{equation}
    \label{e:Sdecays}
    | \partial_\bk^{|\bk|} S^\tau_t (x) | \leqslant C  \, | x |^{- \gamma} .
  \end{equation}
Similarly, for  $\bk\in\N^2$, we have
\begin{equ}[e:thesame]
\Big|\mathbf E^\tau[\partial_\bk^{|\bk|}\cM^\tau[u^\tau_t](x) \; u^\tau_0(0)]\Big|\vee \Big|\mathbf E^\tau[\partial_\bk^{|\bk|}\cM^\tau[u^\tau_t](x) \;  \cM^\tau[u^\tau_0](0)]\Big| \leqslant C \, | x |^{- \gamma} \,,
\end{equ}
where $\cM^\tau[f]$ is given by
\begin{equ}[e:cM]
\cM^\tau[f] \eqdef \rho_\tau^{\ast 2} \ast f^2 - \mathbb E^\tau\big[\rho_\tau^{\ast 2} \ast f^2 \, (0)\big].
\end{equ}
\end{proposition}


\begin{proof}
Throughout the proof, since $\tau$ is fixed, without loss of generality, we take it to be equal to $1$ and remove it from
the notation. Note that this gives $\nu_\tau=1$ and $R_\tau=\Id$, the identity matrix.
The result for general $\tau$ can then be deduced by scaling.
Further, we will denote by $C>0$ a constant only depending on
the stated quantities, and whose value may change from line to line.

Let us begin with a few preliminary definitions and observations.
Let $\varphi : \R^2 \to [0, 1]$ be a compactly supported smooth function such
  that $\varphi \equiv 1$ on $B_0  (1 / 8) \subset \R^2$, i.e. the ball
  centred at the origin of radius $1 / 8$, and $\varphi \equiv 0$ on $B_0  (1
  / 4)^c$, and set, for $x_0 \in \R^2$, $|x_0|>4$,
$\varphi_{x_0} (y) = \varphi (y / | x_0 |)$.
  Let $\eta'$ be a space white noise on $\R^2$ independent of $\eta$ and
  define $\tilde\eta$ according to
  \begin{equ}
\tilde\eta = \sqrt{\varphi_{x_0}} \eta' + \sqrt{1 -
    \varphi_{x_0}} \eta\,.
  \end{equ}
Note that $\tilde\eta$ is distributed according to a spatial white noise on $\R^2$ and,
setting as usual $\tilde\eta^1\eqdef\tilde\eta\ast\rho$, we have that $\tilde\eta^1$ coincides with $\eta^1$
on $B_0 (R)^c$ where $R \eqdef | x_0 | / 4 + 1$ (because $\rho$ has support in the ball of radius $1$), and
$\restriction{\tilde\eta^1}{B_0  (|x_0 | / 8 - 1)}$ is independent of
$\restriction{\eta^1}{B_0  (|x_0 | / 8 - 1)}$.

Let $\tilde{u}$ be the unique mild solution of \eqref{e:SBEsystem} driven by $(\tilde\eta^1,\vec{\xi})$
and $u$ that driven by $(\eta^1,\vec{\xi}^1)$, so that $\tilde{u}\eqlaw u$.
Set $v\eqdef u-\tilde u$ and $w\eqdef u+\tilde u$ and note that they satisfy~\eqref{e:EquationDifference}
with $X= K_t v_0=K_t (\eta^1-\tilde\eta^1)$.
Then, thanks to the previous observations, the $\bk$-th derivative, for $\bk\in\N^{2}$, of
the two-point correlation function $S$
at $x_0$ can be written as
%
%
  \begin{equ} \label{e:tispezzoindue}
   \partial_\bk^{|\bk|} S_t (x_0) = \mathbf{E} [\partial_\bk^{|\bk|}\tilde{u}_t (x_0)\, u_0 ( 0)] +
    \Exp [\partial_\bk^{|\bk|} v_t(x_0)\, u_0 (0)] = \Exp [\partial_\bk^{|\bk|} v_t(x_0)\, u_0 (0)]
  \end{equ}
  where we used that $\partial_\bk^{|\bk|}\tilde{u}$ is independent of
  $\restriction{\partial_\bk^{|\bk|}\eta^1}{B_0  (|x_0 | / 8 - 1)}$.

At this point, we need to control $\partial_\bk^{|\bk|} v_t(x_0)$.
To do so, we will first take advantage of~\eqref{e:apriori} to get rid of the
randomness of $w$, $\eta^1$ and $\tilde\eta^1$ so to be able to estimate it with
the solution to deterministic PDEs and ultimately prove that these decay sufficiently fast.
For $\delta>0$ and $R$ as above, define the event
\begin{equ} \label{e:EventE}
E=\Big\{\sup_{(t,x)\in[0,T]\times\R^2}|\p_\delta(x)\,\partial_\bk^{|\bk|}w_t(x)|\vee\sup_{x\in B_0(R)}|\partial_\bk^{|\bk|}v_0(x)|\leq R^\delta\Big\}
\end{equ}
and notice that thanks to~\eqref{e:apriori} and~\eqref{e:BoundSHE} (the latter with $t=0$),
the weighted norm of $w$ and the
local sup-norm of $v_0=\eta^1-\tilde\eta^1$ have finite moments of all orders,
which immediately implies that $\Prob(E^c)\lesssim R^{-\nu}$ for any given $\nu>0$. Therefore,
\begin{equs} 
    | \partial_\bk^{|\bk|} S_t (x_0) | &\leq \Exp[|  \partial_\bk^{|\bk|}v_t ( x_0) u_0 (0) | \1_{E^c}] + \Exp[| \partial_\bk^{|\bk|}v_t ( x_0) u_0 ( 0) | \1_{E}]\\
    & \lesssim R^{- \nu} \sqrt{\Exp [| \partial_\bk^{|\bk|}v_t ( x_0)|^2 |u_0 (
    0)|^2]} + \sqrt{\Exp[| \partial_\bk^{|\bk|}v_t ( x_0)|^2\1_{E}] \Exp[|u_0 ( 0) |^2]}\\
  & \lesssim R^{- \nu}+ \sqrt{\Exp[| \partial_\bk^{|\bk|}v_t ( x_0)|^2\1_{E}]}
  \label{e:Decay1}
\end{equs}
where, in the first step we applied Cauchy-Schwarz, while, in the last, again Cauchy-Schwarz on the first expectation
and then stationarity and finiteness of moments of $u_0(0)=\eta^1(0)$ and of $\partial_\bk^{|\bk|}\eta^1(0)$,
to conclude that the
bound is uniform in $t$ and $x_0$. Thus, we are left to show that, on $E$, $| \partial_\bk^{|\bk|}v_t ( x_0)|$
can be estimated by the right hand side of~\eqref{e:Sdecays}.

Recall the definition of $\cN^1[\cdot,\cdot]$ in~\eqref{e:BilinearForm},
that $v$ solves~\eqref{e:EquationDifference} and that, by construction, $v_0$ is compactly supported in a ball of
radius $R$. Then, on $E$, for every $(t,x)\in[0,T]\times\R^2$, we have
\begin{equs}[e:Decay2]
| \partial_\bk^{|\bk|}v_t ( x)|&\leq K_t| \partial_\bk^{|\bk|}v_0| ( x)+\int_0^t K_{t-s}\big|\partial_\bk^{|\bk|}\cN^1[v_s,w_s]\big|(x)\dd s\\
&\leq  K_t (R^\delta\1_{B_0(R)})(x_0)+R^\delta \int_0^t K_{t-s}\ast\big|\partial_\bk^{|\bk|}(\fw\cdot\nabla)\rho^{\ast2}\big|\ast(|v_s|\p_{-\delta})(x)\dd s\,.
\end{equs}
Let us consider first the case of $|\bk|=0$.
As $\rho^{\ast2}$ has
compact support and $\p_{-\delta}$ is a weight,~\eqref{e:PropWeight} ensures the existence
of a constant $c_\delta>0$ such that the convolution in the second summand above can be estimated by
\begin{equs}[e:Decay3]
\big|(\fe_1\cdot\nabla)\rho^{\ast2}\big|\ast(|v_s|\p_{-\delta})(z)&=\int \big|\partial_1\rho^{\ast2}\big|(z-y) |v_s(y)|\p_{-\delta}(y)\dd y\\
&\leq c_\delta p_{-\delta}(z)\int \big|\partial_1\rho^{\ast2}\big|(z-y) |v_s(y)|\dd y\\
&= p_{-\delta}(z)\int \PHI(z-y) [|v_s(y)|-|v_s(z)|]\dd y + m_\PHI  p_{-\delta}(z)|v_s(z)|
\end{equs}
where $\PHI\eqdef c_\delta |\partial_1\rho^{\ast2}|$ (which is compactly supported, positive and even)
and  $m_\PHI$ is its total mass.
Therefore,~\eqref{e:Decay2} and~\eqref{e:Decay3} imply that, uniformly in
$(t,x)\in[0,T]\times\R^2$, $v$ is bounded by the deterministic function $Z=Z_t(x)$ which solves
\begin{equ}[e:Z]
\partial_t Z_t(x)=\Big(\tfrac12\Delta +L_{\PHI,R^\delta \p_{-\delta}}\Big)Z_t(x)+m_\PHI R^\delta \p_{-\delta}(x) Z_t(x)\,,\quad Z_0(x)=R^\delta \1_{B_0(R)}(x)\,,
\end{equ}
for $L_{\PHI,R^\delta \p_{-\delta}}$ the operator acting on functions $u\colon\R^2\to\R$ as
\begin{equ}[e:L]
L_{\PHI,R^\delta \p_{-\delta}}u(x)\eqdef R^\delta \p_{-\delta}(x)\int\PHI(x-y) [u(y)-u(x)]\dd y\,.
\end{equ}
As detailed in Appendix~\ref{a:FK},
the advantage of singling out the operator $L_{\PHI,R^\delta \p_{-\delta}}$ is that
$\frac12 \Delta + L_{\PHI,R^\delta \p_{-\delta}}$ is the generator of a Markov
process on $\mathbb{R}^2$, and this allows to derive an explicit expression for the solution to~\eqref{e:Z}
via a Feynman-Kac formula, which in turn implies the (exponential) decay.
Indeed, it is not hard to see that Lemma~\ref{l:decayZ} applies to $Z$ with the choices of
$\PHI$ as above, $f=  R^\delta \p_{-\delta}$, $g=m_\PHI R^\delta \p_{-\delta}$ and $h=R^\delta\1_{B_0(R)}$.
Hence, recalling that $R=|x_0|/4+1$ (so that, since $h$ is supported in a ball of radius $r_h=R$, we have $|x_0|\geq 2r_h$ as $|x_0|> 4$, as required by Lemma~\ref{l:decayZ}),
we deduce that on the event $E$
\begin{equ}
|v_t(x_0)|\leq Z_t(x_0)\lesssim e^{-C |x_0|}
\end{equ}
which, once plugged into~\eqref{e:Decay1} with $|\bk|=0$, implies~\eqref{e:Sdecays}.

For $|\bk|\geq0$, we go back to~\eqref{e:Decay2} and upper-bound $|v|$ at the right hand
side by $Z$. Then, arguing as in~\eqref{e:Decay3},
we deduce that $|\partial_\bk^{|\bk|}v|$ is upper bounded by
the solution $Z^\bk$ of the non-homogeneous deterministic heat equation given by
\begin{equ}
\partial_t Z^\bk_t(x)=\frac12\Delta Z^{\bk}_t(x) + c_\delta R^\delta \p_{-\delta}(x)\Big(\big|\partial_\bk^{|\bk|}\partial_1\rho^{\ast2}\big|\ast Z_t\Big)(x)\,,\qquad Z_0(x)=R^\delta \1_{B_0(R)}(x)\,.
\end{equ}
Since the initial condition is compactly supported and, by Lemma~\ref{l:decayZ},
the second summand decays exponentially (outside the ball of radius $2R=|x_0|/2+2$),
it can be easily checked that also $Z^\bk_t(x_0)$ decays exponentially
in $|x_0|$ from which~\eqref{e:Sdecays} follows at once.
\medskip

The proof of~\eqref{e:thesame} requires only minor changes.
For instance, arguing as in~\eqref{e:tispezzoindue},
the first term at the right hand side can be rewritten as
\begin{equ}
  \mathbf E\Big[\partial_\bk^{|\bk|}(\cM^1[u_t](x_0)-\cM^1[\tilde u_t](x_0))u_0(0)\Big].
\end{equ}
On the other hand,
\begin{equ}
  \partial_\bk^{|\bk|}(\cM^1[u_t](x_0)-\cM^1[\tilde u_t](x_0))=\partial_\bk^{|\bk|}\rho^{\ast 2}\ast(u^2_t-\tilde u^2_t)(x_0)=
  \partial_\bk^{|\bk|}\rho^{\ast 2}\ast(v_tw_t)(x_0).
\end{equ}
At this point, one proceeds exactly as above, removing $|w|$ as in \eqref{e:Decay1} and
then bounding $|v|$ by $Z$.
We omit obvious details.
\end{proof}

A first consequence of Proposition \ref{p:corrdecay} is that the two-point correlation function $S^\tau$
has total mass equal to $1$, i.e.
\begin{equ}
  [e:afirst]
  \int S^\tau_t(x) \, \dd x=\int S^\tau_0(x) \, \dd x=\int
  \rho_\tau^{\ast 2}(x)\, \dd x=1
\end{equ}
independently of  $t\ge0$.
To see this, note first that, since $u^\tau$ is a (spatially smooth) mild solution to~\eqref{e:SBEsystem},
then it is standard to show that for every $(t,x)\in\R_+\times\R^2$ we have
\begin{equ}[e:SBEWeakMild]
u^\tau_t(x) = \eta^\tau(x) + \frac12\int_0^t \nabla \cdot R_\tau \, \nabla u^\tau_s(x) \ \dd s + \lambda\nu_\tau^{3/4} \int_0^t \cN^\tau[u^\tau_s](x) \, \dd s + \int_0^t \nabla \cdot \sqrt{R_\tau} \vec{\xi}^{\tau}(\dd s, x) \,
\end{equ}
which, by multiplying both sides by $u^\tau_0(0)=\eta^\tau(0)$ and taking expectation, gives
\begin{equ}
  \label{eq:Sintegrata}
  S^\tau_t(x)-S^\tau_0 (x)=\frac{1}2\int_0^t\nabla \cdot R_\tau\, \nabla S^\tau_s(x) \, \dd s+\lambda\nu_\tau^{3/4}\int_0^t \Exp\Big[\partial_1\cM^\tau[u^\tau_s](x) u^\tau_0(0)
  \Big]\dd s\,,
\end{equ}
where we used that the initial condition and the driving noise are independent.
Integrating with respect to $x$, all terms in the right-hand side vanish, as follows by integration by parts, which is
allowed by the decay properties in Proposition \ref{p:corrdecay}.
A similar argument delivers the second important consequence, namely a
Green-Kubo formula for the diffusivity matrix $D^\tau$ defined 
as in~\eqref{e:FormalDiffusivity} but with $u$ replaced by the rescaled process $u^\tau$ in~\eqref{e:GoodScaling}. 

\begin{proposition}\label{p:GK}
Let $S^\tau$ be the two-point correlation function in~\eqref{e:Corr}.
Then, for every $t\geq 0$, the diffusivity matrix $D^\tau$ in~\eqref{e:FormalDiffusivity} with $S^\tau$ in place of $S$ 
is well-defined and admits the
following Green-Kubo representation
\begin{equ}[e:Dnew]
 D^\tau_{1,1}(t) = 
\nu_\tau + \frac{2\lambda^2\nu_\tau^{3/2}}t\int_0^t \int_0^s\int_{\R^2}\Exp\Big[
    \cM^\tau[u^\tau_r](x) \cM^\tau[u^\tau_0](0) \Big] \, \dd x \, \dd r \, \dd s\,,
\end{equ}
where $\cM^\tau$ is defined in \eqref{e:cM}, while $D^\tau_{2,2}(t) = 1$ and $D^\tau_{1,2}(t)=D^\tau_{2,1}(t)=0$.
\end{proposition}

In previous works~\cite{de2024log,CET}, the right hand side of~\eqref{e:Dnew} (or the corresponding one for the AKPZ equation) with
$\tau=\nu_\tau=1$ was adopted as \emph{definition} of bulk diffusion coefficient,
but the equivalence with an expression similar to that in \eqref{e:FormalDiffusivity} was only argued
heuristically.
\begin{proof}
First, by~\eqref{e:Sdecays}, the integral at the right hand side of~\eqref{e:FormalDiffusivity} is absolutely convergent 
and thus the diffusion matrix $D^\tau$ is well-defined. 
On the other hand, the proof of~\eqref{e:Dnew} is essentially the same as in \cite[App. B]{de2024log},
except that the authors consider the bulk diffusion coefficient, which is a scalar quantity, and
the argument therein is formal, since it involves an unjustified integration by parts,
which holds in view of the decay of correlations shown in Proposition \ref{p:corrdecay}.

Start from \eqref{eq:Sintegrata}, multiply by $x_ix_j/t$, for $i,j\in\{1,2\}$, and integrate with respect to $x$, to obtain
  \begin{equs}[e:Dnew1]
    D_{i,j}^\tau(t)=&
    \frac{1}{2t}\int_0^t \int_{\R^2}x_ix_j\,\nabla \cdot R_\tau \, \nabla S^\tau_s(x)\dd x\dd s\\
    &+
    \frac{\lambda\nu_\tau^{3/4}}{t}\int_0^t\int_{\R^2}x_ix_j\,\Exp^\tau\Big[\partial_1\cM^\tau[u^\tau_s](x) (u^\tau_0(0)-u^\tau_s(0))
  \Big]\dd x\dd s
\end{equs}
where we used the fact that $\cM^\tau[u_s]$ is quadratic while $u^\tau_s(0)$ is linear in $u^\tau$,
and that, for every $s\geq 0$, $u^\tau_s\eqlaw\eta^\tau$ with $\eta^\tau$ Gaussian, so that
$\Exp[\cM^\tau[u^\tau_s](x) u^\tau_s(0)] = \E[\cM^\tau[\eta^\tau](x) \, \eta^\tau(0)] = 0$.
For the first term at the right hand side of~\eqref{e:Dnew1}, we apply integration by parts twice
and use that $\nabla \cdot R_\tau \nabla x_ix_j=2\nu_\tau \1_{i=j=1}+2\1_{i=j=2}$ 
together with~\eqref{e:afirst}.
As for the second, we use translation invariance, the tower property and stationarity to write
\begin{equs}
\Exp^\tau\Big[\partial_1\cM^\tau[u^\tau_s](x) (u^\tau_0(0)-u^\tau_s(0))
  \Big]=\E^\tau\Big[\partial_1\cM^\tau[\tilde\eta^\tau](0) \Exp^\tau_{\tilde\eta^\tau}[\tilde u^\tau_s(x)-\tilde u^\tau_0(x)]\Big]
\end{equs}
where $\tilde\eta^\tau \eqdef u_s^\tau \, (= \tilde u_0^\tau)$ is a regularised spatial white noise and $\tilde u$ is the time-reversed process,
i.e. $\tilde u_r\eqdef u_{s-r}$. Thanks to the skew-reversibility of $u$ stated in Proposition~\ref{p:Global},
in law, $\tilde u$ satisfies the same equation as~\eqref{e:SBEWeakMild} with $\lambda$ replaced by $-\lambda$ and
a (different) space-time white noise which is independent of $\tilde\eta$.
Hence, by plugging in such expression in place of the increment of $\tilde u$, integrating by parts
and using the independence of initial condition and driving noise, we obtain
\begin{equs}
  -\frac{\lambda^2\nu_\tau^{3/2}}{t}\int_0^t &\int_0^s\int_{\R^2}x_ix_j \Exp\Big[
\partial_1\cM^\rho[u^\tau_r](x) \partial_1\cM^\rho[u^\tau_0](0)
\Big]\dd r \dd s\dd x\\
&=\frac{\lambda^2\nu_\tau^{3/2}}{t}\int_0^t \int_0^s\int_{\mathbb  R^2}x_ix_j\mathbf E\Big[
\partial_1^2\mathcal M^\rho[u_r](x) \mathcal M^\rho[u_0](0)
\Big]\dd r \dd s\dd x.
\end{equs}
Another integration by parts, allowed by \eqref{e:thesame}, ensures that the above
coincides with the second summand in~\eqref{e:Dnew} if $i=j=1$ and is $0$ otherwise.
Putting first and second summands together, the statement follows.
\end{proof}

\subsection{The semigroup and the generator of SBE}\label{sec:generator}

In Proposition~\ref{p:Global}, we showed that, for every $\tau > 0$,~\eqref{e:SBEsystem} admits a
unique solution $u^\tau$ and that such solution is a stationary Markov process.
As such, it possesses
both a semigroup $P^\tau$ and a generator $\cL^\tau$, and all of the analysis detailed in the next sections
hinges on determining their behaviour as $\tau$ goes to $+\infty$.
Prior to that, we need to ensure that, for fixed $\tau>0$,
$P^\tau$ satisfies suitable basic properties and to study the
action of $\gen$ on its domain. 
\medskip

For $t\geq0$ and $F\in\L^1(\P^\tau)$, the semigroup $P^\tau$ of $u^\tau$ is given by
\begin{equ}[e:semigroup]
P_t^\tau F \eqdef \E^\tau\big[F(u^\tau_t) \, \big| \, \cF_0\big] \, ,
\end{equ}
which is clearly well-defined. 

\begin{lemma}\label{lem:Contractivity}
Let $\theta\in[1,\infty)$. Then, $(P^\tau_t)_{t\geq0}$ in~\eqref{e:semigroup}
defines a contractive strongly continuous semigroup on $\L^\theta(\P^\tau)$
such that, for every $F\in \L^\theta(\P^\tau)$, 
\begin{equ}[e:Contractivity]
\|P_t^\tau F\|_{\mathbb L^\theta(\mathbb P^\tau)} \leq \|F\|_{\mathbb L^\theta(\mathbb P^\tau)}
\end{equ}
holds uniformly in $\tau$ and also for $\theta=\infty$.
In particular,
the generator $\cL^\tau \colon \mathrm{Dom}(\gen) \subset \L^\theta(\P^\tau) \to \L^\theta(\P^\tau)$
associated to $P^\tau$ is closed and, for every $\mu>0$, 
its resolvent $(\mu-\gen)^{-1}$ is bounded uniformly in $\tau$ by $\mu^{-1}$ 
and has range given by $\mathrm{Dom}(\gen)$.
\end{lemma}
\begin{remark}
Even though for most of the paper we will focus on the case $\theta=2$,
we stated the previous result for generic $\theta\in[1,\infty)$ as the contractivity in~\eqref{e:Contractivity}
(which is uniform in $\tau>0$) will be essential
to upgrade the convergence of the {\it two}-point distribution of $u^\tau$ (i.e. of random vectors
of the form $(u^\tau_{t_1}(\phi_1), u^\tau_{t_2}(\phi_2)$) to that of {\it finite} dimensional
distributions (see Section~\ref{sec:CLT} for more details).
\end{remark}
\begin{proof}
The contractivity bound~\eqref{e:Contractivity} follows by a standard application of Jensen's inequality
for conditional expectations and stationarity of $u^\tau$. To see that $P^\tau$ is
a strongly continuous semigroup on $\L^\theta(\P^\tau)$ for $\theta\in[1,\infty)$, it suffices to
verify that, for a given bounded cylinder random variable $F$ (or $F$ in any dense subset of $\L^\theta(\P^\tau)$),
$P^\tau_tF$ converges to $F$ as $t\to0$ in $\L^\theta(\P^\tau)$.
This in turn is a consequence the continuity of $u^\tau$ in time
and smoothness of $F$.
At last, the properties of the generator listed in the statement
follow by classical Hille-Yosida theory (see e.g.~\cite[Theorem 1.2.6]{EK86}).
\end{proof}

We now want to study the action of $\gen$ on $\L^2(\P^\tau)$, which, by abusing notation,
is identified with the Fock space $\fock{}{}{\tau}$. To obtain an explicit
expression for such action, let us introduce the subset $\core$ of
$\fock{}{}{\tau}$ consisting of elements
$f = (f_n)_{n \geq 0} \in \fock{}{2}{\tau}$
such that the sequence $(\|f_n\|_{\fock{n}{2}{\tau}})_{n \geq 0}$
decays faster than any inverse power. Then, the following lemma, whose proof is 
postponed to Appendix~\ref{a:Basics}, holds. 

\begin{lemma}\label{l:ActionGen}
Let $\tau>0$ be fixed.
Let $\gensy$, $\genap$ and $\genam$ be the operators
such that $\gensy(\fock{0}{}{\tau})=\genap(\fock{0}{}{\tau})=\genam(\fock{1}{}{\tau})=0$, 
and that, for any $n\in\N$, map $\fock{n}{}{\tau}\cap\cS(\R^{2n})$ to
$\fock{n}{}{\tau}, \fock{n+1}{}{\tau}$ and $\fock{n-1}{}{\tau}$, respectively,
and whose action on $\psi\in\fock{n}{}{\tau}\cap\cS(\R^{2n})$ is given by
\begin{equs}
\cF(\gensy \psi)\,(p_{1:n})&\eqdef \cF\big((\nu_\tau\cL_0^{\fe_1}+\cL_0^{\fe_2})\psi\big)\,(p_{1:n})=-\frac{1}{2} \,  |\sqrt{R_\tau} p_{1:n}|^2 \, \hat\psi(p_{1:n})\qquad \label{e:gensy} \\
\cF(\genap \psi) \, (p_{1:n+1})
&\eqdef \lambda\nu_\tau^{3/4} \,\frac{2}{n+1} \sum_{1 \leq i < j \leq n+1}
\nonlinp_{p_i,p_j} \, \hat{\psi}\big(p_i + p_j, p_{1:n+1 \setminus \{i,j\}}\big) \, , \label{e:genap:fock}
\\
\cF(\genam \psi) \, (p_{1:n-1})
&\eqdef \lambda\nu_\tau^{3/4} \,
n \sum_{i=1}^{n-1}
\int_{\R^2} \nonlinm_{q, p_i-q} \, \hat{\psi} \big(q,p_i-q, p_{1:n-1\setminus\{i\}}\big)  \, \dd q \,  . \label{e:genam:fock}
\end{equs}
where 
\begin{equs}
\cF(\cL_0^{\fe_j}\psi)\,(p_{1:n})&\eqdef -\frac{1}{2}|\fe_j\cdot p_{1:n}|^2\hat\psi(p_{1:n})\,,\qquad j=1,2,\label{e:L0w}\\
\nonlinp_{q,q'} \eqdef -\frac{\iota}{2 \pi} \, [\fe_1 \cdot (q+q')] \, \Theta_\tau(q+q') \,& , \quad\nonlinm_{q,q'} \eqdef -\frac{\iota}{2 \pi} \, [\fe_1 \cdot (q+q')] \, \Theta_\tau(q)\Theta_\tau(q')\,,\quad \label{e:K+K-}
\end{equs}
and $\Theta_\tau$ is as in~\eqref{e:Mollifiers}.
Then, $\gen$ coincides with $\gensy+\genap+\genam$ on $\core$, which is contained in ${\rm Dom}(\gen)$.
Further, on $\core$, $\gensy$ and $\gena\eqdef\genap+\genam$ are respectively the
symmetric and skew-symmetric parts of $\gen$ and $(\cA^\tau_{\pm})^*=- \cA^\tau_{\mp}$.
At last, $\core$ is a core for $\gen$.
\end{lemma}
\begin{remark}
The importance of $\core$ is that, compared to the domain of $\gen$,
which is only implicitly given, its elements possess an explicit description and, since
it is a core, any property or statement regarding $\gen$ only need to be checked on it.
Such a feature will be used throughout the paper.
\end{remark}

  Before proceeding let us comment on the nature of the operators $\cS^\tau$, $\genap$ and $\genam$,
and provide some intuition regarding their action on $\fock{}{}{\tau}$.
The Fock space $\fock{n}{}{\tau}$ can be seen as the span of ``$n$-particle states'',
each particle carrying a momentum $p_i$, $i\in\{1,\dots,n\}$.
In this basis, while $\gensy$ is diagonal and its action simply consists of multiplying by ($(-1/2)$ times)
the sum of squared scaled momenta, $\genap$ and $\genam$ are not.
Indeed, $\genap$ takes one of the $n$ particles, say that of momentum $p_i+p_j$,
and {\it creates} two particles, whose momenta are $p_i$ and $p_j$. Dually,
$\genam$ takes two particles, say with momenta $p_i$ and $p_j$,
and merges (or {\it annihilates}) them into one of momentum $p_i+p_j$.
The kernels $\nonlinp,\nonlinm$ determine the rates at which these transitions occur.

The difficulty in analysing $\gena$ (and more specifically $\genap$) 
lies precisely in the fact that, even though $p_i+p_j$ might be small,
there is no mechanism, uniform in $\tau$, that prevents $p_i$ and $p_j$ from being arbitrarily large.
To see this, let $\phi$ be as regular as it gets, say in $\cS(\T^2)$, 
and consider the $\fock{}{-1}{\tau}$-norm of $\gena\phi=\genap\phi$ (the importance of such norm 
will become apparent in the next lemma). 
Then, it is not hard to see that 
\begin{equs}
\|\gena\phi\|_{\fock{}{-1}{\tau}}^2 &= \frac{\lambda^2}{4\pi^2}\int_{\R^2} \dd p \, |\fe_1\cdot p|^2|\hat\phi(p)|^2\Big(\nu_\tau^{3/2}\int\dd q \frac{\Theta_\tau(p)\Theta_\tau(q)\Theta_\tau(p-q)}{1+\frac12|\sqrt{R_\tau}(q,p-q)|^2}\Big)\\
&\gtrsim \nu_\tau\log\tau \int_{\R^2}\dd p \, |\fe_1\cdot p|^2|\hat\phi(p)|^2\label{e:explosion}
\end{equs}
which explodes as $\tau\to\infty$ since $\nu_\tau\log\tau\sim(\log\tau)^{1/3}$. 
\medskip

Finally, let us conclude with the following statement which will allow us to
perform apriori estimates on the resolvent and approximations thereof.

\begin{lemma} \label{lem:Apriori}
Let $\tau>0$ be fixed. Then $\mathrm{Dom}(\gen) \subset \fock{}{1}{\tau}$ and for $\psi \in \mathrm{Dom}(\gen)$ it holds
\begin{equ} \label{e:Apriori}
\|\psi\|_{\fock{}{1}{\tau}}^2 = \langle (1 - \gen) \, \psi, \psi\rangle_\tau \, .
\end{equ}
\end{lemma}
\begin{proof}
Thanks to Lemma~\ref{l:ActionGen},~\eqref{e:Apriori} is satisfied for $\psi \in \core$.
Indeed, using the skew-symmetry of $\gena$,  it holds
\begin{equ} \label{e:AprioriCore}
\|\psi\|_{\fock{}{1}{\tau}}^2 = \langle (1 - \gensy) \, \psi, \psi \rangle_\tau = \langle (1 - \gen) \, \psi, \psi \rangle_\tau \, ,
\end{equ}
which, by Cauchy-Schwarz, implies
\begin{equ} \label{e:BoundRes}
\|\psi\|_{\fock{}{1}{\tau}} \leq \|(1 - \gen) \, \psi\|_{\fock{}{-1}{\tau}} \, .
\end{equ}
Now, let $\psi \in \mathrm{Dom}(\gen)$, and $(\psi^n)_{n \geq 0}\subset\core$ be such that $\psi^n \to \psi$ and
$(1 - \gen) \, \psi^n \to (1 - \gen) \, \psi$, in $\fock{}{}{\tau}$ (the existence of such a sequence is due to
$\core$ being a core). Then, applying \eqref{e:BoundRes} to $\psi_n - \psi_m$,
we deduce that $(\psi_n)_{n \geq 0}$ is Cauchy in $\fock{}{1}{\tau}$.
As it already converges to $\psi$ in $\fock{}{}{\tau}$, the convergence must necessarily occur in
$\fock{}{1}{\tau}$ which implies that $\psi \in \fock{}{1}{\tau}$. Hence,
$\mathrm{Dom}(\gen) \subset \fock{}{1}{\tau}$. Upon applying~\eqref{e:AprioriCore} to $\psi_n$
and taking the limit in $n$,~\eqref{e:Apriori} follows.
\end{proof}

Let us conclude this section by collecting a few facts on the semigroup, $P^\eff$,
and the generator, $\geneff$, of the solution $u^\eff$ of the limiting (effective) equation in~\eqref{e:EffectiveSBE},
which we here recall
\begin{equ}[e:Eff]
\partial_t u^{\eff} = \frac{1}{2} \, \nabla \cdot D_{\eff} \, \nabla u^\eff + \nabla \cdot \sqrt{D_{\eff}} \, \vec{\xi} \, , \quad D_{\eff} \eqdef \begin{pmatrix}
C_\eff(\lambda) & 0 \\
0 & 1
\end{pmatrix} \,,
\end{equ}
where $C_\eff(\lambda)$ is as in~\eqref{e:EffConstant}.

\begin{lemma}\label{l:LimitGen}
The semigroup $(P^\eff_t)_{t\geq 0}$ of the solution $u^\eff$ to~\eqref{e:Eff}
is strongly continuous on $\L^\theta(\P)$ for $\theta\in[1,\infty)$ and contractive on $\L^\theta(\P)$ for $\theta\in[1,\infty]$.
Its generator $\geneff$ is closed and has bounded resolvent. Further,
$\geneff = C_\eff(\lambda) \, \cL_0^{\fe_1} + \cL_0^{\fe_2}$, where $C_\eff(\lambda)$ is
the constant given in~\eqref{e:EffConstant}
and the operators $\cL_0^{\fe_1}$, $\cL_0^{\fe_2}$ were defined in~\eqref{e:L0w}.

At last, $u^\eff$ is $\L^2(\P)$-mixing, in the sense that for all $F\in\L^2(\P)$, we have
\begin{equ} \label{e:EffErgodicity}
\lim_{t\to\infty}\|P_t^{\eff} F - \mathbb E[F]\|=0 \, .
\end{equ}
\end{lemma}


\begin{proof}
Basic properties of the semigroup and generator of $u^\eff$ are standard and can be proved, e.g., using the
same procedure as in Lemmas~\ref{lem:Contractivity} and~\ref{l:ActionGen}.
The $\L^2(\P)$-mixing property is well-known and straightforward to verify (a possible way to do so
is by using the explicit expression for the action of the semigroup on $\fock{}{}{}$).
%
%
\end{proof}

\subsection{Heuristics, novelties and roadmap}\label{sec:HeuRM}

Before delving into the details of the proof, we give here a
high-level roadmap that outlines the strategy and the main novelties, and
should guide the reader through the next sections.

As stated in Theorem~\ref{MT}, our goal is to prove
that the semigroup $(P^\tau_t)_{t\ge0}$ of the solution to the regularised SBE in~\eqref{e:SBEsystem} 
converges strongly to that of the effective limit equation~\eqref{e:Eff},
$(P^{\rm eff}_t)_{t\ge0}$. Via a suitable version of the classical
Trotter-Kato theorem, this follows if we are able to show strong
convergence of the resolvent $(\mu-\gen)^{-1}$ to $(\mu-\geneff)^{-1}$ for some $\mu>0$,
which amounts to prove that for every $\phi\in\L^2(\P)$, we have
\begin{equ}[e:StrongCvHeu]
(\mu - \gen)^{-1}  \phi \xrightarrow[\tau\to\infty]{\L^2(\P)} (\mu - \cL^{\mathrm{eff}})^{-1} \phi \,
\end{equ}
(see Proposition \ref{thm:StrongCvResL2} below for a more precise statement).
Establishing~\eqref{e:StrongCvHeu} requires a thorough understanding of its left-hand side,
whose structure is highly non-trivial, and a way to see how, as $\tau$ goes to $\infty$, the relatively
simple right hand side arises. Let us fix $\mu=1$ and, for the sake of the present informal discussion,
a smooth element in the first chaos, i.e. take $\phi\in\cS(\R^2)$, and
set $r^\tau\eqdef (1 - \gen)^{-1}  \phi$.
\medskip

A crucial observation is that the behaviour of the resolvent is intimately related to
the following operator fixed-point-equation
\begin{equ}[e:OpFP]
  \ch^{\tau}=-\genam \cH_\tau\genap\,, \qquad \text{with}\quad \cH_\tau\eqdef (1-\gensy+\ch_\tau)^{-1}.
\end{equ}
(For a heuristic as to where~\eqref{e:OpFP} comes from see~\cite[Section 5.1]{CT}.)
Indeed, assume that~\eqref{e:OpFP} admits a solution $\ch_\tau$ and that $\ch_\tau$ is diagonal in chaos,
i.e. it maps $\fock{n}{}{\tau}$ to itself. Now, define $\tilde r^\tau$ formally as 
\begin{equ}[e:seriesr]
\tilde r^\tau\eqdef \sum_{n\geq 0} (\cH_\tau\genap)^n\cH_\tau\phi\,,
\end{equ}
which means that
$\tilde r^\tau_0=0$, $\tilde r^\tau_1=\cH_\tau\phi$ and, when projected onto the $n$-th chaos component for $n>1$,
$\tilde r^\tau_n= (\cH_\tau\genap)^{n-1}\cH_\tau\phi=\cH_\tau\genap\tilde r^\tau_{n-1}$.
In particular, it satisfies
\begin{equ}[e:True]
\tilde r^\tau=\cH_\tau\genap\tilde r^\tau +\cH_\tau\phi\,.
\end{equ}
We claim that $\tilde r^\tau$ coincides with $r^\tau$.
Indeed, applying $(1-\gen)$ to $\tilde r^\tau$ we obtain
\begin{equs}
(1-\gen) \, \tilde r^\tau &= (1-\gensy+\ch_\tau) \, \tilde r^\tau-\genap\tilde r^\tau-\genam \tilde r^\tau-\ch_\tau \tilde r^\tau\\
&=\big[\cH_\tau^{-1}\tilde r^\tau-\genap\tilde r^\tau\big]+\big[-\genam \cH_\tau\genap-\ch_\tau\big]\tilde r^\tau-\genam\cH_\tau\phi
\end{equs}
where in the first step we added and subtracted $\ch_\tau\tilde r^\tau$ and in the second we used the definition
of $\cH_\tau$ and plugged~\eqref{e:True} into $-\genam \tilde r^\tau$.
Now, by~\eqref{e:True} the first term coincides with $\phi$, the second
is $0$ thanks to~\eqref{e:OpFP}, and the last vanishes since, by Lemma~\ref{l:ActionGen},
$\genam(\fock{1}{}{\tau})\equiv 0$. In other words, $(1-\gen)\tilde r^\tau=\phi=(1-\gen) r^\tau$ which,
by applying the resolvent at both sides, proves the claim.

What the previous shows is that the operator $\ch_\tau$ is the building block
at the core of $(1-\gen)^{-1}$ as it provides a (semi-){\it explicit representation}
for $r^\tau$. With it in hands,~\eqref{e:StrongCvHeu} would follow provided we could prove 
that
\begin{enumerate}[label=(\arabic*),noitemsep]
\item\label{i:point1}  $(r^\tau_n)_{n\geq 2}$ vanishes as $\tau\to0$ in $\fock{}{}{\tau}$, 
\item\label{i:point2}  $r^\tau_1=\cH_\tau\phi\in\fock{1}{}{\tau}$ is close to $(1-\geneff)^{-1}\phi$.
\end{enumerate}
\medskip

The problem with the above is that, even assuming
existence and uniqueness of the fixed point, $\ch^{\tau}$ is likely to be a very complicated
operator to which we have {\it no direct access}.
The way out, first explored in~\cite{CGT} (and~\cite{AKPZweak,CG}) in the so-called weak coupling
scaling 
is to identify an operator
$\cG^\tau$ that, on one hand, is an {\sl approximate} solution of~\eqref{e:OpFP},
and on the other has a sufficiently simple
form (in particular, it is diagonal in the sense of Definition~\ref{def:DiagOp}) that allows to
rigorously show (1) and (2) in an appropriate sense.
To move beyond weak-coupling though, this is still not enough.

The difficulty lies in the fact that the error in the approximate fixed point equation
(referred to as the ``Replacement Lemma'',~\cite[Lemma 2.5]{CGT}), while
small in $\tau$ in each given Fock space component  $\fock{n}{}{\tau}$, grows
polynomially in the chaos index $n$. Such growth is caused by the
``off-diagonal terms'' (see \eqref{e:diag+offdiag}) in scalar products involving $\gena$
as explained in the proof of Lemma~\ref{l:offDi}.
Let us remark that this problem is {\it structural} and somewhat intrinsic to
stationary non-reversible Markov processes (see~\cite{KLO}).
It lead to the introduction of the so-called {\it graded sector condition}
(see~\cite{Balint}) in the study of Gaussian
fluctuations for such systems and is responsible for the
sub-dominant corrections to the $(\log t)^{2/3}$-divergence of the diffusion coefficient
$D(t)$ of $2d$-ASEP \cite{Yau} and of the $2d$-SBE \cite{de2024log},
as well as those to the $(\log t)^{1/2}$ divergence of $D(t)$ of the AKPZ equation \cite{CET}.
\medskip

The main novelty of the present work is the identification of a way that allows
{\it both} to determine a diagonal approximate fixed point for~\eqref{e:OpFP} 
{\it and} to tame the growth in
chaos, thereby establishing~\eqref{e:StrongCvHeu}
in the strong coupling regime (thus beyond the range of applicability of the
graded sector condition, see Remarks~\ref{rem:GSCgenafl} and~\ref{rem:WeakvsStrong}),
and obtaining the precise asymptotic for $D(t)$, as well as a (superdiffusive) central limit theorem. 
To understand how the strategy came about, there are two 
aspects of the antisymmetric part 
of the generator that should be taken into account, namely, the {\it magnitude} and the {\it direction} of the momenta it produces.  

To see the former, note that $\gena$ (and more specifically $\genap$) 
is ``dangerous'' for two very
distinct reasons. As was discussed after Lemma
\ref{l:ActionGen}, the operator $\genap$, even when acting on
smooth functions with localised momenta, 
creates {\it new} momenta which can range from low to arbitrarily high.
The creation of high momenta causes the output to be very irregular in real space (or spread out in momentum space) and
is the mechanism responsible for the growth of $D(t)$ and, under the correct scaling, for the origin of $\geneff$.
When instead the new momenta are of the same magnitude as the original ones, then
new and old ``resonate'' and generate contributions which are lower order in $\tau$
but grow polynomially in the number of variables (i.e. the chaos number).

For the latter, note that even though $\gena$ takes as an input momenta from every direction, the singularity it produces 
is concentrated {\it only in the direction $\fe_1$}, 
which is the direction of the ``drift'', i.e. the non-linearity, of SBE~\eqref{e:SBEsystem}. 
To witness, the right hand side of~\eqref{e:explosion} (but also of every upper bound we will obtain 
on $\gena$) only involves the {\it anisotropic} $H^1$-norm of $\phi$, i.e. the $H^1$-norm in the direction $\fe_1$, 
which is $\|(-\gensyx)^{1/2}\phi\|_\tau=\frac12\|(\fe_1\cdot\nabla)\phi\|_\tau$.  
Further, in such direction, the linear part of the dynamics is vanishing 
at an order $\nu_\tau\ll \nu_\tau^{3/4}$, the latter being the coefficient of $\cN^\tau$ (and unsurprisingly 
the rate of convergence we derive is $\nu_\tau^{1/4}$). 
\medskip

To translate the previous observations in mathematical terms, we introduce two novel tools. 
First, we identify a splitting of $\genap$ (and, by duality, $\genam$) which decouples  
the singularity in Fourier and the growth in chaos (see Section~\ref{sec:Asharp}). 
More specifically, we write it as the sum of 
$\genapsh$ and $\genapfl$, with $\genapsh$ only producing high modes,
and no growth in the chaos number,
while $\genapfl$ keeping track of the latter and regular (actually, small) in Fourier. 
Second, we introduce a {\it Recursive Replacement Lemma} (see Lemma~\ref{p:2ndReplLemma}) 
which (morally) allows to recursively estimate each of the summands at the right hand side of~\eqref{e:seriesr}   
and precisely pins down the contribution of the anisotropic $\fock{}{1}{\tau}$-norm of the resolvent $r^\tau$, 
i.e. the $\fock{}{1}{\tau}$-norm in the direction $\fe_1$. 

With the previous at hand, the strategy of proof becomes
\begin{enumerate}[label=(\Roman*), noitemsep]
\item\label{i:pointI} find an approximate fixed point $\cg^\tau$ to~\eqref{e:OpFP} but with $\genapsh$ and $\genamsh$
in place of $\genap$ and $\genam$ (so to avoid growth in the chaos in the estimates);
\item\label{i:pointII} define the Ansatz for the resolvent $s^\tau$ as the right hand side of~\eqref{e:seriesr} but with
$\cG^\tau\eqdef (\mu-\gensy+\cg^\tau)^{-1}$ and $\genapsh$ instead of $\cH_\tau$ and $\genap$;
\item\label{i:pointIII} control the anisotropic $\fock{}{1}{\tau}$-norm of $s^\tau$, i.e. $\|(-\nu_\tau\gensyx)^{1/2} s^\tau\|_\tau$, 
with the Recursive Replacement Lemma and 
show that it vanishes in the limit for $\tau\to\infty$. This will show both that the contribution of $\genafl s^\tau$
is small and that $s^\tau$ is sufficiently close to the resolvent applied to $\phi$;
\item\label{i:pointIV} verify that $r^\tau$ satisfies~\ref{i:point1} and~\ref{i:point2} above, thus proving~\eqref{e:StrongCvHeu}.
\end{enumerate}
The proof of these points will be carried out in Sections~\ref{sec:Est} and~\ref{sec:CVRes}.
In the former, we will introduce: the splitting $\genash$/$\genafl$,
the approximate fixed point (point~\ref{i:pointI}); the Recursive Replacement Lemma; 
and the main estimates needed in the next.
Section~\ref{sec:CVRes} puts these tools to practice and addresses points~\ref{i:pointII}-\ref{i:pointIV}.

\section{Estimates on the generator acting in Fock spaces}\label{sec:Est}

This section is devoted to the analysis of the antisymmetric part $\gena$ of the generator $\gen$. 
We will rigorously present the novel tools mentioned in Section~\ref{sec:HeuRM}
and obtain the bounds we will need in order to establish~\eqref{e:StrongCvHeu}. 

\subsection{Diagonal and off-diagonal contributions}

As should be clear by~\eqref{e:OpFP}, a crucial role will be played by operators of the form
$-\genam \cD \genap$ with $\cD$ a non-negative diagonal operator on $\fock{}{}{\tau}$ 
according to Definition~\ref{def:DiagOp}, with kernel $\bd = (\bd_n)_{n \geq 0}$ 
and that we assume to be well-defined on $\core$ (this is the case, e.g., if $\cD$ is bounded). 
Notice first that, since by Lemma~\ref{l:ActionGen} $-\genam = (\genap)^*$ on $\core$ and diagonal 
operators are always self-adjoint (on their natural domain), it holds 
\begin{equ} \label{e:Duality1}
\langle - \genam \cD \genap \Phi, \psi \rangle_\tau = \langle \cD \genap \Phi, \genap \psi \rangle_\tau \quad\text{and}\quad \langle - \genam \cD \genap \Phi, \Phi \rangle_\tau = \|\cD^{1/2} \genap \Phi\|_\tau^2
\end{equ}
for $\Phi,\psi\in\core$. 

Let us now introduce the  {\it diagonal} and {\it off-diagonal terms} of operators $- \genam \cD \genap$, with $\cD$ as above. 
For $\Phi, \psi \in\core$, the former is given by 
\begin{equs}[e:diag]
\langle-\genam\cD\genap\Phi,\psi\rangle_{\tau,\di} \eqdef \sum_n &n! \, \int \Xi_n^\tau(\dd p_{1:n}) \, \hat\Phi_n(p_{1:n}) \, \overline{\hat\psi}_n(p_{1:n})\\
&\times \sum_{i=1}^n (\fe_1\cdot p_i)^2 \, \frac{\lambda^2 \nu_\tau^{3/2}}{2 \pi^2} \int_{\R^2} \dd q \, \bd_{n+1}(q, p_i-q, p_{1:n\setminus\{i\}}) \, \cJ_{q,p_1-q}^\tau
\end{equs}
for $\cJ_{q,q'}^\tau\eqdef \Theta_\tau(q)\Theta_\tau(q')\Theta_\tau(q+q')$, while
the others, referred to as the {\it off-diagonal terms} of type $1$ and $2$ respectively,
are
\begin{equs}
\langle - \genam \cD \genap \, \Phi, \psi\rangle_{\tau, \oDi}\eqdef& \sum_n n! \, c_{\oDi}(n) \ \lambda^2 \nu_\tau^{3/2} \int \Xi_{n+1}^\tau(\dd p_{1:n+1}) \, \bd_{n+1}(p_{1:n+1}) \, \nonlinp_{p_1,p_2} \, \overline{\nonlinp_{p_1,p_3}}\\
&\qquad \times \hat{\Phi}(p_1+p_2, p_{1:n+1\setminus\{1,2\}}) \ \overline{\hat \psi}(p_1+p_3, p_{1:n+1\setminus\{1,3\}})\,,  \label{e:offdiag1} \\
\langle - \genam \cD \genap \, \Phi, \psi\rangle_{\tau, \ooDi}\eqdef& \sum_n n! \, c_{\ooDi}(n) \ \lambda^2 \nu_\tau^{3/2} \int \, \Xi_{n+1}^\tau(\dd p_{1:n+1})\,\bd_{n+1}(p_{1:n+1}) \, \nonlinp_{p_1,p_2} \, \overline{\nonlinp_{p_3,p_4}} \\
&\qquad\times \hat{\Phi}(p_1+p_2, p_{1:n+1\setminus\{1,2\}}) \ \overline{\hat{\psi}}(p_3+p_4, p_{1:n+1\setminus\{3,4\}}) \label{e:offdiag2}
\end{equs}
where $c_{\oD_1}(n)\eqdef 4 n (n-1)$ and $c_{\oD_2}(n)\eqdef n (n-1) (n-2)$, 
and $\nonlinp_{q,q'}$ is as in~\eqref{e:K+K-}. In the next lemma, we show that ~\eqref{e:diag}-\eqref{e:offdiag2} indeed provide a decomposition of $- \genam \cD \genap$.

\begin{lemma}\label{l:offDi}
Let $\cD$ be a non-negative diagonal operator on $\fock{}{}{\tau}$ which is well-defined on $\core$. 
Then, the operator $-\genam\cD\genap$ preserves the chaos, in the sense
that it commutes with the number operator $\cN$, and for every $\Phi,\psi\in\core$ we have 
\begin{equ}[e:diag+offdiag]
\langle-\genam\cD\genap\Phi,\psi\rangle_\tau=\langle-\genam\cD\genap\Phi,\psi\rangle_{\tau,\di}+\sum_{i=1}^2\langle-\genam\cD\genap\Phi,\psi\rangle_{\tau,\oD_i}
\end{equ}
where the summands at the right hand side are respectively given by~\eqref{e:diag},~\eqref{e:offdiag1} and~\eqref{e:offdiag2}.

As a consequence, for every $n\in\N$, we have 
\begin{equ} \label{e:CSoffdiag}
\Big|\langle - \genam \cD \genap \, \Phi_n, \psi_n\rangle_{\tau, \oD_i}\Big| \lesssim n^2 \langle-\genam\cD\genap\Phi_n,\Phi_n\rangle_{\tau,\di}^{1/2} \langle-\genam\cD\genap\psi_n,\psi_n\rangle_{\tau,\di}^{1/2} \, ,
\end{equ}
for all $\Phi,\psi\in\core$. 
\end{lemma}

\begin{proof}
  The fact that $-\genam\cD\genap$ preserves the chaos follows
  immediately from Lemma~\ref{l:ActionGen},
  while~\eqref{e:diag+offdiag} can be argued as in~\cite[Lemma
  3.6]{CET} so that we will only sketch the proof.
  Consider the first equality in~\eqref{e:Duality1}.  By~\eqref{e:genap:fock}, the Fourier
  kernels of $\genap \Phi$ and $\genap \psi$ contain a sum over
  indices $i<j$, so that, when multiplying them and expanding the
  product, we obtain a sum over $i<j$, $i'<j'$. The diagonal part then
  corresponds to the sum over $i=i'<j=j'$, the off-diagonal of type
  $1$ to that over $|\{i,j\}\cap\{i',j'\}|=1$ while the off-diagonal
  of type $2$ to that over $|\{i,j\}\cap\{i',j'\}|=0$.
  Equations~\eqref{e:offdiag1} and~\eqref{e:offdiag2} are then obtained via a
  change of variables and are a consequence of the fact that $\Phi$
  and $\psi$ are symmetric with respect to permutation of their
  variables. The constants
  $c_{\oD_i}(n)$, $i=1,2$, are then simple combinatorial factors
     arising from counting the number of possible choices of
    $i<j,i'<j'$ satisfying $|\{i,j\}\cap\{i',j'\}|=1$ or
  $|\{i,j\}\cap\{i',j'\}|=0$. At last,~\eqref{e:CSoffdiag} is a
  straightforward application of Cauchy-Schwarz and the bound $c_{\oD_i}(n) = O(n^3)$.
\end{proof}

The expressions~\eqref{e:diag},~\eqref{e:offdiag1} and~\eqref{e:offdiag2} are quite suggestive
in regards to the expected contributions from these terms. 
An inspection of the first hints at where the 
diffusivity in the direction $\fe_1$ comes from. Indeed, its right hand side features an integral in $q$ 
which {\it involves neither $\phi$ nor $\psi$},
so that its convergence as $\tau\to\infty$ is totally independent of them. In particular, their
decay in Fourier plays no role, as we saw in~\eqref{e:explosion} for $\cD=(1-\cS^\tau)^{-1}$.
Further, if we assume for a moment that, for the right choice of $\cD$, such integral
were a constant $D$ independent of $p_{1:n}$,
then we would have $\langle-\genam\cD\genap\Phi,\psi\rangle_{\tau,\di}=\langle(-D\cL_0^{\fe_1})\Phi,\psi\rangle_{\tau}$. 
Our goal is to show that for the (approximate) fixed point to~\eqref{e:OpFP}, 
in the limit for $\tau\to\infty$ such integral indeed converges to a constant 
and that such constant coincides with the first entry of $D_\eff$ in~\eqref{e:EffectiveSBE}.

On the other hand, the value of $c_{\oD_i}(\cN)$, $i=1,2$, in~\eqref{e:offdiag1} and~\eqref{e:offdiag2}, 
and the trivial bound~\eqref{e:CSoffdiag} say that off-diagonal terms can be controlled by the diagonal one, 
times a factor growing like the square of the chaos index.
This bound alone is way too brutal for our purposes 
since, as hinted at above, we expect the effective diffusivity to arise from the diagonal terms alone. 
Thus, the off-diagonal terms should actually be {\it negligible} in the large-scale limit 
but, as we will discuss in next section, to see this we need to perform 
a much more refined analysis of the operators $-\genam \cD \genap$.

\subsection{Singling out momentum and chaos singularities}\label{sec:Asharp}

In order to set apart the issues coming from the singularity in Fourier and the chaos-growth, we now introduce
a splitting of the operator $\gena$ which is one of the main novelties of the present work.
For any $n\in\N$ and $\psi\in\fock{n}{}{\tau}\cap\cS(\R^{2n})$, we define the operators $\genapsh$ and $\genamsh$
according to
\begin{equs}
\cF(\genapsh \psi) \, (p_{1:n+1}) &\eqdef \lambda \nu_\tau^{3/4} \frac{2}{n+1} \sum_{1 \leq i < j \leq n+1} \nonlinp_{p_i,p_j} \, \chi_{p_i, p_j}^{\sharp,\tau}(p_i+p_j, p_{1:n+1\setminus\{i,j\}}) \\
&\qquad\qquad\qquad\qquad\qquad\qquad\qquad\qquad\times\hat{\psi}(p_i+p_j, p_{1:n+1\setminus\{i,j\}})\label{e:A+sharp}\\
\cF(\genamsh \psi) \, (p_{1:n-1})
&\eqdef \lambda \nu_\tau^{3/4}  \,
n \sum_{i=1}^{n-1}
\int_{\R^2} \chi_{q,p_i-q}^{\sharp,\tau}(p_{1:n-1}) \, \nonlinm_{q, p_i-q} \, \hat{\psi} \big(q,p_i-q, p_{1:n-1\setminus\{i\}}\big) \dd q \qquad\label{e:A-sharp}
\end{equs}
and set $\genash\eqdef\genapsh+\genamsh$, where $\chi^\sharp$ is given by
\begin{equ}\label{e:ChiSharp}
\chi_{q,q'}^{\sharp,\tau}(p_{1:N}) \eqdef \1_{|\sqrt{R_\tau} q| \wedge |\sqrt{R_\tau} q'| > 2 \, |\sqrt{R_\tau} p_{1:N}|} \, .
\end{equ}
Upon setting $\chi^{\flat, \tau} \eqdef 1 - \chi^{\sharp, \tau}$, we also introduce the operator $\genafl\eqdef \genapfl+\genamfl$,
where $\genapfl$ and $\genamfl$ are given as in~\eqref{e:A+sharp} and~\eqref{e:A-sharp} but with $\chi^\flat$
in place of $\chi^\sharp$. Before discussing their features, let us state in the next lemma
their basic properties.

\begin{lemma}\label{l:Asharp}
Let us fix $\tau > 0$. Then the operators $\genash$ and $\genafl$ are continuous on $\core$
and their sum coincides with $\gena$. Further, on $\core$,
they are skew-symmetric, and
$(\genapsh)^\ast= - \genamsh$, $(\genapfl)^\ast= - \genamfl$ (in particular \eqref{e:Duality1} remains true upon 
replacing $\gena_{\pm}$ by $\genash_{\pm}$ or $\genafl_{\pm}$).
At last, for any bounded diagonal operator $\cD$ satisfying the assumptions of Lemma~\ref{l:offDi},
both $-\genamsh\cD\genapsh$ and $-\genamfl\cD\genapfl$ preserve the chaos, and
the same decomposition into diagonal and off-diagonal terms given in~\eqref{e:diag+offdiag}
holds with $\chi^\sharp$ and $\chi^\flat$ appropriately placed at the right hand sides of~\eqref{e:diag}-\eqref{e:offdiag2}.
\end{lemma}
\begin{proof}
Continuity can be argued as in Lemma~\ref{l:GSC}, while the other properties
follow immediately by the definition of the operators and straightforward computations.
\end{proof}

The advantage of the previous decomposition is that
$\genash$ and $\genafl$, while singular in Fourier and chaos respectively, 
are ``regular'' in the other. To substantiate such claims, we first show that
$\genash$ does not produce any off-diagonal contribution (and thus its bounds will not display any  
dependence in $\cN$). 

\begin{lemma} \label{l:SharpDiag}
Let $\cD$ satisfy the assumptions of Lemma~\ref{l:offDi}. Then,
the operator $-\genamsh \, \cD \, \genapsh$ is a non-negative diagonal operator
according to Definition~\ref{def:DiagOp} on $\core$.
In other words, for any $\Phi \in \core$ and $\psi \in \fock{}{}{\tau}$ it holds
\begin{equ} \label{e:SharpDiag}
\langle - \genamsh \cD \genapsh \, \Phi, \psi\rangle_\tau = \langle - \genamsh \cD \genapsh \, \Phi, \psi\rangle_{\tau,\di}\,.
\end{equ}
\end{lemma}
\begin{proof}
Let us begin by showing~\eqref{e:SharpDiag}. Thanks to Lemma~\ref{l:offDi} (and Lemma~\ref{l:Asharp}), 
it follows provided $\langle - \genamsh \cD \genapsh \, \Phi, \psi\rangle_{\tau,\oD_i} = 0$ for $i=1,2$. 
We will focus on the case $i=1$ as the other is completely analogous. 
By definition, we have 
\begin{equs}
\langle - \genamsh \cD \genapsh \, \Phi, \psi\rangle_{\tau, \oDi}=& \sum_n n! \, c_{\oDi}(n) \ \lambda^2\nu_\tau^{3/2} \int \Xi_{n+1}^\tau(\dd p_{1:n+1}) \, \bd_{n+1}(p_{1:n+1}) \, \nonlinp_{p_1,p_2} \, \overline{\nonlinp_{p_1,p_3}}\\
&\qquad \times \chi^{\sharp,\tau}_{p_1,p_2}(p_1+p_2, p_{1:n+1\setminus\{1,2\}}) \, \chi^{\sharp,\tau}_{p_1,p_3}(p_1+p_3, p_{1:n+1\setminus\{1,3\}})\\
&\qquad \times \hat{\Phi}(p_1+p_2, p_{1:n+1\setminus\{1,2\}}) \ \overline{\hat \psi}(p_1+p_3, p_{1:n+1\setminus\{1,3\}})\,.
\end{equs}
Now, the product of the $\chi^\sharp$'s can be different from $0$ only if they are both 
equal to $1$. By~\eqref{e:ChiSharp}, this holds provided that
\begin{equ}
|\sqrt{R_\tau} p_1| \wedge |\sqrt{R_\tau} p_2| > 2 \,|\sqrt{R_\tau} (p_1+p_2,p_{1:n+1\setminus\{1,2\}})| \geq |\sqrt{R_\tau} p_3|
\end{equ}
and
\begin{equ}
|\sqrt{R_\tau} p_2| \wedge |\sqrt{R_\tau} p_3| > 2 \, |\sqrt{R_\tau}(p_2+p_3,p_{1:n+1\setminus\{2,3\}})| \geq |\sqrt{R_\tau} p_1|
\end{equ}
 i.e. $|\sqrt{R_\tau} p_1| > |\sqrt{R_\tau} p_3| > |\sqrt{R_\tau} p_1|$ which is impossible. Hence,~\eqref{e:SharpDiag} follows. 
 
Since the second line at right hand side of~\eqref{e:diag} defines a non-negative Fourier multiplier
(as it only depends on $p_{1:n}$) and $\genapsh$, $\genamsh$ and $\cD$ are well-defined on 
$\core$, we conclude that $- \genamsh \cD \genapsh$ defines
a diagonal operator according to Definition~\ref{def:DiagOp} on $\core$, so that the proof of the statement is complete.
\end{proof}

On the other hand, the operator norm of $\genafl$ from $\fock{}{1}{\tau}$ to $\fock{}{-1}{\tau}$ is bounded 
{\it uniformly in $\tau$} (but the bound depends polynomially on the number operator $\cN$), 
which, together with~\eqref{e:explosion}, says that the singularity in Fourier is fully captured by $\genash$. 
The next proposition states an even more refined estimate. 

\begin{proposition} \label{p:GSCgenafl}
There exists a constant $C>0$ independent of $\tau$ such that for every $\Phi \in \core$ it holds
\begin{equ} \label{e:GSCgenafl}
\|\genafl_{\pm} \Phi\|_{\fock{}{-1}{\tau}} \leq C \lambda \, \|\cN (- \nu_\tau \gensyx)^{1/2} \Phi\|_\tau \, .
\end{equ}
\end{proposition}

\begin{remark}\label{rem:GSCgenafl}
  The bound~\eqref{e:GSCgenafl} on $\genafl$ is reminiscent of the
  so-called Graded Sector Condition (GSC) of~\cite[Section 2.7.4]{KLO}
  but with three differences. 
  First, the GSC refers to the \emph{full} operator $\gena_\pm$, while in \eqref{e:GSCgenafl} only $\genafl_\pm$ appears.
Second, the GSC
   would require the
  power of $\cN$ to be strictly less than $1$ or $\lambda$ to be
  sufficiently small (see~\cite{Balint}).  Arguing as in~\cite[Lemma
  2.4]{CGT}, it should be possible to improve~\eqref{e:GSCgenafl} and
  show that instead of $\cN$ one has $\sqrt{\cN}$ at its right hand
  side. On the other hand, the GSC will play no role in our analysis, because anyway it \emph{does not hold independently of $\tau$ for the full operator $\gena$, but at best only for $\genafl$} (see also Remark \ref{rem:WeakvsStrong}). For this reason, we refrained
  from optimising \eqref{e:GSCgenafl} and limited ourselves to use the trivial
  bound~\eqref{e:CSoffdiag}.  Third, and more importantly, the
  operator at the right hand side of \eqref{e:GSCgenafl}  is {\it not} the symmetric part of
  $\gen$, i.e. $\cS^\tau$, as is usual for GSC, {\it but only its
    component in the $\fe_1$-direction}, i.e. $(- \nu_\tau \gensyx)$,
  that is in a sense much smaller than $-\cS^\tau$.  As mentioned in
  Section~\ref{sec:HeuRM}, this fact will be crucial for us (see
  Section~\ref{sec:2ndReplLemma}).  Note that, if $\phi$ is smooth,
  say in $\cS(\R^{2n})$ for some $n\in\N$, then
  $\| (- \nu_\tau \gensyx)^{1/2}
  \phi\|_\tau\lesssim_{\phi,n}\sqrt{\nu_\tau}\sim(\log\tau)^{-1/3}$
  which vanishes as $\tau\to\infty$ (see~\eqref{e:explosion} for
  comparison).
\end{remark}

\begin{proof}
Let us consider the diagonal bounded operators $\cD^{(1)}=(1-\cS^\tau)^{-1}$ and, 
for $\kappa>0$, $\cD^{(2)}=(\kappa-\nu_\tau\gensyx)^{-1}$. 
A simple duality argument connected to the fact that $\genamfl=-(\genapfl)^\ast$ 
ensures that it suffices to prove that for all $\Phi\in\core$ and $(i,j) \in \{(1,2), (2,1)\}$, 
uniformly over $\kappa>0$ we have 
\begin{equ}[e:MainHere]
\|(\cD^{(i)})^{1/2}\genapfl(\cD^{(j)})^{1/2}\Phi\|_\tau\lesssim \|\cN\phi\|_\tau\,.
\end{equ}
Since both $\cD^{(1)}$ and $\cD^{(2)}$ are bounded operators, they are well-defined on $\core$, 
which implies that~\eqref{e:Duality1} holds. Hence, by~\eqref{e:diag+offdiag} and~\eqref{e:CSoffdiag} 
we get  
\begin{equs}[e:Intermezzo]
\,&\|(\cD^{(i)})^{1/2}\genapfl(\cD^{(j)})^{1/2}\Phi\|^2_\tau\lesssim  \langle-\genamfl\cD^{(i)}\genapfl\cN(\cD^{(j)})^{1/2}\Phi,\cN(\cD^{(j)})^{1/2}\Phi\rangle_{\tau,\di}\\
&= \frac{\lambda^2}{2 \pi^2}\sum_n n! \,n^2 \int \Xi_n^\tau(\dd p_{1:n})\, |\hat\Phi_n(p_{1:n})|^2\times \\
&\quad\times \bd^{(j)}_n(p_{1:n})\nu_\tau^{3/2} \sum_{i=1}^n(\fe_1\cdot p_i)^2 \int_{\R^2} \dd q \, \bd^{(i)}_{n+1}(q, p_i-q, p_{1:n\setminus\{i\}}) \chi^\flat_{q,p_i-q}(p_{1:n})\, \cJ_{q,p_1-q}^\tau
\end{equs}  
where we further used that $\cN$ and $-\genamfl\cD\genapfl$ commute and the definition of the 
diagonal part in~\eqref{e:diag}. 
To estimate the last line above, note that $\chi^\flat$ is bounded above by the 
sum of the indicator functions over $\{|\sqrt{R_\tau} q| < 2 |\sqrt{R_\tau} p_{1:n}|\}$ and 
$\{|\sqrt{R_\tau} (p_i-q)| < 2 |\sqrt{R_\tau} p_{1:n}|\}$, but the two resulting summands coincide as can be shown by 
a simple change of variables. Further, the triangle inequality easily gives 
\begin{equ}
\bd^{(i)}_{n+1}(q, p_i-q, p_{1:n\setminus\{i\}})\lesssim \bd^{(i)}_{n}(p_{1:n})\,.
\end{equ}
Putting these observations together, we obtain
 \begin{equs}
 \bd^{(j)}_n(p_{1:n})&  \nu_\tau^{3/2}\sum_{i=1}^n(\fe_1\cdot p_i)^2\int_{\R^2} \dd q \, \bd^{(i)}_{n+1}(q, p_i-q, p_{1:n\setminus\{i\}}) \chi^\flat_{q,p_i-q}(p_{1:n})\, \cJ_{q,p_1-q}^\tau\\
 &\lesssim   \bd^{(j)}_n(p_{1:n})\bd^{(i)}_{n}(p_{1:n})\nu_\tau^{3/2}\sum_{i=1}^n(\fe_1\cdot p_i)^2 \ \int_{\R^2} \dd q \, \1_{|\sqrt{R_\tau} q| < 2 |\sqrt{R_\tau} p_{1:n}|} \\
 &\lesssim \bd^{(j)}_n(p_{1:n})\bd^{(i)}_{n}(p_{1:n}) \nu_\tau|\fe_1\cdot p_{1:n}|^2 |\sqrt{R_\tau}p_{1:n}|^2\lesssim 1\,,
 \end{equs}
the second to last step being a consequence of the fact that, for any $r>0$, the volume of $\{|\sqrt{R_\tau} q|\leq r\}$ 
is $r^2\,\nu_\tau^{-1/2}$, while  the last follows from  the definition of $\cD^{(1)}$, $\cD^{(2)}$, recalling that $\{i,j\} = \{1,2\}$. 
Plugging the above back into~\eqref{e:Intermezzo},~\eqref{e:MainHere} and thus the statement follow. 
\end{proof}

\subsection{The approximate fixed point $\cG_M^\tau$ and the Replacement Lemmas}
\label{sec:apprG}

In this section, we focus on operators involving $\genash$ only and aim at refining 
our understanding of operators of the form $- \genamsh \cD \genapsh$ for carefully   
chosen diagonal $\cD$. More precisely, our goals are to first identify an approximate 
fixed point $\cG^\tau$ to~\eqref{e:OpFP} (with $\genash_\pm$ in place of
$\gena_\pm$) which is diagonal and bounded, and then to obtain suitable estimates 
on $\cG^\tau \genapsh$ as this will appear in our ansatz for 
the solution to the resolvent equation (see~\ref{i:pointII} and~\eqref{e:True}/\eqref{e:AnsatzFormula}). 

Let us begin by introducing the candidate approximate fixed point $\cG^\tau_M$, whose 
definition requires a few notations. 
Let $L^\tau$ be the function on $\R_+$ given by 
\begin{equ} \label{e:Leps}
L^\tau(x) \eqdef \frac{\lambda^2 \nu_\tau^{3/2}}{\pi} \log \left(1 + \frac{\tau}{1+x}\right) \, .
\end{equ}
and, for a measurable function $f \colon \R_+ \to \R_+$, set $-\cL_0^{\fe_1}[f]$ 
to be the diagonal operator  
\begin{equ} \label{e:Lwf}
-\cL_0^{\fe_1}[f] \eqdef -f(L^\tau(- \cS^\tau)) \, \gensyx \, ,
\end{equ}
whose Fourier multiplier is $f(L^\tau\big(\tfrac12|\sqrt{R_\tau} p_{1:n}|^2\big)) \, \tfrac{1}{2} |\fe_1 \cdot p_{1:n}|^2$ 
for any $p_{1:n}\in\R^{2n}$ and $n\in\N$.

For $M\geq 0$, let $g^\tau_M$ be the function on $\R_+$ defined as 
\begin{equ} \label{e:gtauM}
g^\tau_M(x) \eqdef (M \nu_\tau) \vee g^\tau(x)\,
\end{equ}
where 
\begin{equ}\label{e:gtau}
g^{\tau}(x) \eqdef\Big(\tfrac{3}{2} \, x + \nu_\tau^{3/2}\Big)^{2/3} - \nu_\tau\,,
\end{equ}
and $\cg_M^\tau$ and $\cG_M^\tau$ be the diagonal operators 
\begin{equ}[e:cgG]
\cg_M^\tau \eqdef -\gensyx[g_M^\tau] \, , \qquad \text{and} \qquad \cG_M^\tau \eqdef \big(1 - \gensy + \cg_M^\tau\big)^{-1}\,,
\end{equ}
which are the approximate fixed point candidates alluded to earlier. 

Before stating the lemma that will make the previous claim rigorous, let us  
comment on the structure (and the advantages) of the operators $\cg_M^\tau$ and $\cG_M^\tau$. 
What they provide is a {\it scale-by-scale} description of how the limiting operator arises. 
Indeed, the function $g^\tau$ appearing in their definition is what determines the size of the diffusivity 
and, as we will see in the proof of the next statement (see also~\cite[Remark 2.6]{CGT}), 
it must be chosen in such a way that  it solves the ODE
\begin{equ}[e:ReplODE]
\dot y (\ell) = \frac{1}{\sqrt{\nu_\tau + y(\ell)}}\,,\qquad y(0) = 0\,.
\end{equ}
If $p_{1:n}\in\R^{2n}$ and $M\geq 0$ are given, 
then the Fourier multiplier of $\cg^\tau_M$ satisfies 
\begin{equs}
\cF(\cg^\tau_M)(p_{1:n})&=g^\tau_M(L^\tau\big(\tfrac12|\sqrt{R_\tau} p_{1:n}|^2\big)) \, \tfrac{1}{2} |\fe_1 \cdot p_{1:n}|^2\\
&= \bigg\{\Big[\Big(\tfrac{3}{2} \, L^\tau\big(\tfrac12|\sqrt{R_\tau} p_{1:n}|^2\big) + \nu_\tau^{3/2}\Big)^{2/3} - \nu_\tau\Big]\vee (M\nu_\tau)\bigg\}\,\tfrac{1}{2} |\fe_1 \cdot p_{1:n}|^2\qquad\\
&\xrightarrow[]{\tau\to\infty} \quad \frac{C_\eff(\lambda)}{2}  |\fe_1 \cdot p_{1:n}|^2\label{e:PointwiseConv}
\end{equs} 
with $C_\eff(\lambda)$ as in~\eqref{e:EffConstant}, 
so that the right hand side coincides with the Fourier multiplier of the limiting operator in the direction $\fe_1$. 
As will be made more precise in the next section, 
this last observation addresses point~\ref{i:point2} in Section~\ref{sec:HeuRM}

On the other hand, if $-\cS^\tau\gtrsim\tau$ then $L^\tau(-\cS^\tau)$  (i.e. $\ell$ in~\eqref{e:ReplODE}) is vanishingly small and 
$g^\tau(L^\tau(-\cS^\tau))$ is smaller than $\nu_\tau$ so that it does not provide any 
relevant (or new) information.  The parameter $M$ provides an extra degree of freedom 
which controls the ratio between $\nu_\tau$ and $g^\tau_M$ for very small $\ell$, that is, for very small spatial scales. 
We will see in Lemma~\ref{p:2ndReplLemma} that its role is to artificially force the approximate fixed point
to enter some perturbative regime, which will ultimately make our analysis easier. 
In the spirit of Renormalization Group, the somewhat arbitrary parameter $M$, while technically useful, does not influence the large-scale  properties of the system, notably the diffusion matrix for momenta of order $1$, that is independent of $M$.


We are ready for the first of the two main results of this section, which precisely quantifies  
how close $\cg^\tau_M$ and $\cG_M^\tau$ are from solving~\eqref{e:OpFP}.

\begin{lemma}[The Replacement Lemma]\label{p:1stReplLemma}
For $M\geq 0$ and $\tau>0$, let $\cg^\tau_M$ and $\cG^\tau_M$ be the operators defined in~\eqref{e:cgG}. Then, 
there exists a constant $C=C(\lambda,M)>0$ such that for every $\Phi\in\core$ it holds
\begin{equ} \label{e:1stReplLemma}
\|-\genamsh \cG_M^\tau \genapsh \Phi - \cg_M^\tau \Phi\|_{\fock{}{-1}{\tau}} \leq C \, \|(- \nu_\tau \gensyx)^{1/2} \Phi\|_\tau \, .
\end{equ}
\end{lemma}

\begin{remark}\label{rem:1stReplLemma}
In the proof, the dependence of the constant on $\lambda$ and $M$ is explicitly tracked and it turns out to be  proportional to
$\lambda^2 \frac{1 \vee \log(M)}{1 \vee M^{1/2}} + M$.
\end{remark}


\begin{remark} \label{rem:BeyondWC}
Compared to previous works in which statements as the above have been exploited, 
Lemma~\ref{p:1stReplLemma} displays two main differences. The first lies in the 
fact that the norm of $\Phi$ at the right hand side is {\it not} the whole of the $\fock{}{1}{\tau}$-norm 
but only part of it, and more precisely the part in the $\fe_1$ direction (i.e. the anisotropic $\fock{}{1}{\tau}$-norm). 
Second,~\eqref{e:1stReplLemma} does not
display any growth in chaos, i.e. its right hand side is {\it independent}
of the number operator $\cN$. In all cases previously considered, the
error in the Replacement Lemma
contained an exploding polynomial factor of the form $\cN^\gamma$,
for some $\gamma>0$, which limited its applicability to the weak-coupling regime.
\end{remark}

\begin{proof}
Using orthogonality of chaoses and the fact that $- \genamsh \cG_M^\tau \genapsh$ and $\cg_M^\tau$ preserve the chaos, 
it is easy to see that it is enough to prove \eqref{e:1stReplLemma} for $\Phi \in \fock{n}{}{\tau} \cap \core$ 
and any $n \geq 1$. Moreover, the duality between $\fock{}{-1}{\tau}$ and $\fock{}{1}{\tau}$ 
ensures that we only need to establish the existence of $C=C(\lambda,M)>0$ such that 
for all $\Phi, \psi \in \fock{n}{}{\tau} \cap \core$ it holds
\begin{equ} \label{e:1stReplLemmaGoal}
\Big|\langle - \genamsh \cG_M^\tau \genapsh \Phi, \psi \rangle_\tau - \langle \cg_M^\tau \Phi, \psi \rangle_\tau\Big| \leq C \, \|(- \nu_\tau \gensyx)^{1/2} \Phi\|_\tau \|(- \nu_\tau \gensyx)^{1/2} \psi\|_\tau \, .
\end{equ}
Now, by Lemma~\ref{l:SharpDiag} the first summand reduces to its diagonal part (see \eqref{e:diag})
so that, recalling the definitions of $\cg_M^\tau$ and $\cG_M^\tau$ in~\eqref{e:cgG}, we have  
\begin{equs}\label{e:1stReplFormula}
\Big|\langle - &\genamsh \cG_M^\tau\genapsh \Phi, \psi \rangle_\tau -\langle \cg_M^\tau \Phi, \psi \rangle_\tau\Big| \leq n! \int \Xi_n^\tau(\dd p_{1:n}) \, |\hat\Phi(p_{1:n})| \, |\hat\psi(p_{1:n})| \times  \\
&\times \sum_{i=1}^n \frac12 |\fe_1 \cdot p_i|^2 \, \bigg|\frac{\lambda^2 \nu_\tau^{3/2}}{\pi^2} \int_{\R^2}  \frac{\cJ_{q,p_i-q}^\tau \, \chi_{q, p_i-q}^\sharp(p_{1:n})}{1 + \Gamma_i^\tau + g_M^\tau(L^\tau(\Gamma_i^\tau)) \Gamma_i^{\fe_1}}\dd q - g_M^\tau\big(L^\tau\big(\tfrac12|\sqrt{R_\tau} p_{1:n}|^2\big)\big)\bigg|\,,
\end{equs}
where we adopted the same notations as in~\ref{e:NotationsRepl}, i.e. for $i\in\{1,\dots,n\}$, 
\begin{equ}
\Gamma^\tau_i \eqdef \frac{1}{2} |\sqrt{R_\tau} (q,p_i-q,p_{1:n\setminus\{i\}})|^2 \, , \quad
\Gamma_i^{\fe_1} \eqdef \frac{1}{2} |\fe_1 \cdot (q,p_i-q,p_{1:n\setminus\{i\}})|^2 \, .
\end{equ}
By Remark~\ref{rmk:ReplTech}, we can apply Lemma~\ref{lem:ReplTech} with the function $\PHI$ in its statement 
given by $\PHI(\gamma, \gamma', \ell) \eqdef (\gamma + g_M^\tau(\ell) \gamma')^{-1}$, so that~\eqref{e:ReplTech} gives 
\begin{equs}[e:1stReplLemInt]
\bigg|&\frac{\lambda^2 \nu_\tau^{3/2}}{\pi^2} \int_{\R^2}  \frac{\cJ_{q,p_i-q}^\tau \, \chi_{q, p_i-q}^\sharp(p_{1:n})}{1 + \Gamma_i^\tau + g_M^\tau(L^\tau(\Gamma_i^\tau)) \Gamma_i^{\fe_1}}\dd q \\
&- \int_0^{L^\tau\big(\tfrac12|\sqrt{R_\tau} p_{1:n}|^2\big)}  \int_0^{\nu_\tau^{-1}} \frac{1}{\pi \sqrt{\varsigma (1 - \nu_\tau \varsigma)}} \, \frac{1}{1+g_M^\tau(\ell) \, \varsigma} \dd \varsigma\dd\ell\bigg| \lesssim \lambda^2 \frac{1 \vee \log(M)}{1 \vee M^{1/2}} \,  \nu_\tau \, .
\end{equs}
Let us now focus on the second term at the left hand side and, to abbreviate the notations, 
set $\ell^\tau\eqdef L^\tau\big(\tfrac12|\sqrt{R_\tau} p_{1:n}|^2\big)$. We first compute the integral in $\varsigma$ 
using the change of variables $\varsigma=(\nu_\tau(1+t^2))^{-1}$, so that 
\begin{equ}[e:1stReplIdentity]
\int_0^{\nu_\tau^{-1}} \frac{1}{\pi \sqrt{\varsigma (1 - \nu_\tau \varsigma)}} \, \frac{1}{1+g_M^\tau(\ell) \, \varsigma} \dd \varsigma =  \frac{1}{\sqrt{1 + g_M^\tau(\ell)}} \, ,
\end{equ}
and we are left to control the difference between the integral on $[0,\ell^\tau]$ of the latter, and $g_M^\tau(\ell^\tau)$. 
To do so, we will separately upper and lower bound~\eqref{e:1stReplIdentity} 
in terms of $g_M^\tau(\ell^\tau)$. 
Recall that $g^\tau$ in~\eqref{e:gtau} is clearly pointwise smaller than $g^\tau_M$ in~\eqref{e:gtauM} 
and was chosen so that it solves the ODE~\eqref{e:ReplODE}. Hence, 
\begin{equs}
\int_0^{\ell^\tau} \frac{\dd \ell}{\sqrt{1 + g_M^\tau(\ell)}} \leq \int_0^{\ell^\tau} \frac{\dd \ell}{\sqrt{1 + g^\tau(\ell)}} = g^\tau(\ell^\tau) \leq g_M^\tau(\ell^\tau) \, . 
\end{equs}
On the other hand, let $\bar\ell\geq0$ be such that $g^\tau(\bar\ell) = M \nu_\tau$. 
Then, $g_M^\tau \equiv g^\tau$ on $[\bar\ell, +\infty)$, and we can lower bound~\eqref{e:1stReplIdentity} as
\begin{equs}[e:LB]
\int_0^{\ell^\tau} \frac{\dd \ell}{\sqrt{1 + g_M^\tau(\ell)}} \geq \int_{\bar\ell\wedge \ell^\tau}^{\ell^\tau} \frac{\dd \ell}{\sqrt{1 + g^\tau(\ell)}} = g^\tau(\ell^\tau) - g^\tau(\bar\ell \wedge \ell^\tau) \geq g_M^\tau(\ell^\tau) - M \nu_\tau \, .
\end{equs}
As a consequence, we deduce 
\begin{equ}
\bigg|\int_0^{L^\tau\big(\tfrac12|\sqrt{R_\tau} p_{1:n}|^2\big)} \frac{\dd \ell}{\sqrt{1 + g_M^\tau(\ell)}} - g_M^\tau(L^\tau\big(\tfrac12|\sqrt{R_\tau} p_{1:n}|^2\big))\bigg| \leq M \nu_\tau \,. 
\end{equ}
Plugging the latter, together with~\eqref{e:1stReplIdentity} and~\eqref{e:1stReplLemInt}, into~\eqref{e:1stReplFormula}, 
~\eqref{e:1stReplLemmaGoal} follows by a simple application of Cauchy-Schwarz, and the proof of the statement 
is thus complete. 
\end{proof}

What the previous lemma gives is a quantitative bound on the distance between 
the operators $-\genamsh\cG^\tau_M\genapsh$ and $\cg^\tau_M$. 
For~\eqref{e:1stReplLemma} to be effective, we need its right hand side to go to $0$ 
and, as can be seen from Remark~\ref{rem:1stReplLemma}, such smallness 
cannot come from the constant $C$ (which instead grows linearly in $M$). 
Thus, we must hope that, for the functions of interest to us, namely the ansatz anticipated in point~\ref{i:pointII}, 
the anisotropic $\fock{}{1}{\tau}$-norm vanishes in $\tau$ for $M$ sufficiently large. 

In the next lemma, we state the second main result of this section 
which is the core of the proof of Proposition~\ref{p:AnisoH1Ansatz}, that in turn shows that this is indeed the case. 

\begin{lemma}[The Recursive Replacement Lemma] \label{p:2ndReplLemma}
Let $f \colon \R_+ \to \R_+$ be a locally Lipschitz function 
such that for every $x > 0$ where its derivative exists it holds
\begin{equ}[e:Assf]
f(x) \leq C_f \, \big(\nu_\tau + g_M^\tau(x)\big)\quad\text{and}\quad |f'(x)| \leq C_f \, \frac{\nu_\tau + g_M^\tau(x)}{x} \, ,
\end{equ}
for some $C_f > 0$. Then, there exists a constant $ C^{(1)}_{M}=C^{(1)}_{M}(\lambda)>0$ independent of $f$ 
and such that $\lim_{M\to\infty} C^{(1)}_{M}=0$, for which for every $\Phi \in \core$ we have 
\begin{equ} \label{e:2ndReplLemma}
\big\|(-\gensyx[f])^{1/2} \, \cG_M^\tau \genapsh \Phi\big\|_\tau^2 \leq \big\|(-\gensyx[\cI_\tau(f) +  \nu_\tau  C_f C^{(1)}_{M}])^{1/2} \Phi\big\|_\tau^2 \, ,
\end{equ}
where $\cG^\tau_M$ is as in~\eqref{e:cgG} and $\cI_\tau$ is the linear operator acting on $\cC(\R_+, \R_+)$ as
\begin{equ}[e:Itau]
\cI_\tau(f) \, (x) \eqdef \frac{1}{2} \int_0^x \frac{f(\ell)}{(\nu_\tau + g^\tau(\ell))^{3/2}} \, \dd \ell \, . 
\end{equ}
\end{lemma}

\begin{remark}\label{rem:2ndReplLemma}
In the proof, the dependence of the constant $C^{(1)}_{M}$ on $\lambda$ and $M$ is explicitly tracked and it turns out to be proportional to
$\lambda^2 \tfrac{1 \vee \log M}{1 \vee M^{1/2}}$.
\end{remark}

\begin{remark}\label{rem:1/2}
The $1/2$ in the expression of $\cI_\tau$ in~\eqref{e:Itau}, and more precisely the fact that it is strictly less than $1$, 
will play a crucial role in our analysis (see Remark~\ref{rem:1/22}). 
As we will see in the proof, it comes from differentiating~\eqref{e:1stReplIdentity} 
with respect to $g^\tau_M$, and the square-root therein is a direct consequence of the anisotropy of the SBE. 
Indeed, for ``isotropic''  models, as the AKPZ equation, the power would be one 
(as shown in~\cite[Eq. (3.29)]{AKPZweak}).  
%
%
\end{remark}

\begin{proof}
Arguing as in Lemma~\ref{p:1stReplLemma}, it suffices to prove the statement 
for $\Phi\in\fock{n}{}{\tau} \cap \core$ and any $n \geq 1$. Furthermore, by Lemma~\ref{l:SharpDiag} 
and the fact that both $\cG^\tau_M$ and $-\gensyx[f]$ are non-negative diagonal (see~\eqref{e:Lwf} for the latter), 
we have 
\begin{equs}
\big\|(-\gensyx[f])^{1/2} \, &\cG_M^\tau \genapsh \Phi\big\|_\tau^2=\langle - \gensyx[f] \, (\cG_M^\tau)^2 \, \genapsh \Phi, \genapsh \Phi \rangle_{\tau, \di} \\
=& \, n! \, \int \Xi_n^\tau(\dd p_{1:n}) \, |\hat\Phi_n(p_{1:n})|^2
\, \times \\
&\times \sum_{i=1}^n \frac{1}{2} |\fe_1\cdot p_i|^2 \, \frac{\lambda^2 \nu_\tau^{3/2}}{\pi^2} \int_{\R^2} \dd q \, \frac{f(L^\tau(\Gamma_i^\tau)) \Gamma_i^{\fe_1}}{1 + \Gamma_i^\tau + g_M^\tau(L^\tau(\Gamma_i^\tau)) \Gamma_i^{\fe_1}} \, \cJ_{q,p_i-q}^\tau \chi^\sharp_{q,p_i-q}(p_{1:n})
\end{equs}
where we adopted the same notations as in~\eqref{e:NotationsRepl}. 

Once again, we apply Lemma~\ref{lem:ReplTech}, this time with the function $\PHI$ given by 
$\PHI(\gamma, \gamma', \ell) = \frac{f(\ell) \gamma'}{(1 + g_M^\tau(\ell) \gamma')^2}$, 
which, by Remark~\ref{rmk:ReplTech}, satisfies~\eqref{e:ReplTechHyp1} and~\eqref{e:ReplTechHyp2}. Hence, we get 
\begin{equs}
\bigg|&\frac{\lambda^2 \nu_\tau^{3/2}}{\pi^2} \int_{\R^2} \dd q \, \frac{f(L^\tau(\Gamma_i^\tau)) \Gamma_i^{\fe_1}}{1 + \Gamma_i^\tau + \Gamma_i^{\fe_1} g_M^\tau(L^\tau(\Gamma_i^\tau))} \cJ_{q,p_i-q}^\tau \chi^\sharp_{q,p_i-q}(p_{1:n}) \\
&- \int_0^{L^\tau\big(\tfrac12|\sqrt{R_\tau} p_{1:n}|^2\big)} \int_0^{\nu_\tau^{-1}} \frac{1}{\pi \sqrt{\varsigma (1 - \nu_\tau \varsigma)}} \, \frac{f(\ell) \varsigma}{(1+g_M^\tau(\ell) \, \varsigma)^2}\dd \varsigma \dd \ell \bigg| \leq \nu_\tau \, C_f \lambda^2 \frac{1 \vee \log(M)}{1 \vee M^{1/2}} \, .
\end{equs}
We are left to compute the second integral at the left hand side. Set $\ell^\tau\eqdef L^\tau\big(\tfrac12|\sqrt{R_\tau} p_{1:n}|^2\big)$. 
The integral in $\varsigma$ can be immediately deduced by~\eqref{e:1stReplIdentity} by differentiating with respect to $g^\tau_M(\ell)$ 
(which is a fixed number for given $\ell$), so that we obtain 
\begin{equs}
\int_0^{\ell^\tau} \int_0^{\nu_\tau^{-1}} \frac{1}{\pi \sqrt{\varsigma (1 - \nu_\tau \varsigma)}} \, \frac{f(\ell) \varsigma}{(1+g_M^\tau(\ell) \, \varsigma)^2}\dd \varsigma \dd \ell &= \frac{1}{2}\int_0^{\ell^\tau} \, \frac{f(\ell)}{(1 + g_M^\tau(\ell))^{3/2}}\dd \ell \\
&\leq \frac{1}{2}\int_0^{\ell^\tau}  \frac{f(\ell)}{ (1 + g^\tau(\ell))^{3/2}}\dd \ell  = \cI(f) \, (\ell^\tau) \, ,
\end{equs}
and the proof is concluded. 
\end{proof}

The reason why we dubbed the previous Lemma {\it Recursive} is because 
we will use~\eqref{e:2ndReplLemma} {\it recursively} on the chaos components 
of our ansatz (see the proof of Proposition~\ref{p:AnisoH1Ansatz}), while {\it Replacement} is motivated by the fact that 
to leading order we expect the quantity at the left hand side of~\eqref{e:2ndReplLemma} to behave as that at the right hand side 
{\it without} the correction $\nu_\tau C_f C^{(1)}_{M}$, and thus we would like to replace the former by the latter. 
That said, a matching lower bound, which would fully justify the name ``replacement'', 
is only true at the cost of a larger (negative) additive error (in the spirit of~\eqref{e:LB}) and, in our setting, an upper bound 
suffices. 

An important observation is that while $C^{(1)}_{M}$ vanishes as $M\to\infty$, 
the constant $C$ in~\eqref{e:1stReplLemma} explodes (see Remark~\ref{rem:1stReplLemma}). 
For our purposes though, it will be enough to choose $M$ sufficiently large so that 
$C^{(1)}_{M}<1/2$ (as this will help stabilize the iteration $f\mapsto \cI_\tau(f)+\nu_\tau  C_f C^{(1)}_{M}$)
but {\it fixed}, thus not creating problems when applying Lemma~\ref{p:1stReplLemma}. 
\medskip

We conclude this section with a lemma whose goal is to control 
the norm of operators of the form 
$\cG^\tau_M\genapsh$ and $\genamsh\cG^\tau_M$ as these play a crucial 
role in addressing point~\ref{i:point1} in Section~\ref{sec:HeuRM}. 

\begin{proposition} \label{p:GA+Norm}
There exist a constant $ C^{(2)}_M= C^{(2)}_M(\lambda)  > 0$ 
which converges to $0$ as $M\to\infty$,
such that for every $\Phi \in \core$ it holds
\begin{equ} \label{e:GA+Norm}
\|\cG_M^\tau \genapsh \Phi\|_\tau \leq C^{(2)}_M \|\Phi\|_\tau \, .
\end{equ}
Furthermore, for any $\alpha\in[0,1)$ there exists a constant $C = C(\lambda,\alpha)>0$, for which 
\begin{equ} \label{e:GA+SmoothNorm}
\|\cG_M^\tau \genapsh \1_{-\cS^\tau \leq \tau^{\alpha}} \, \Phi\|_\tau \vee  \|\genamsh \cG_M^\tau \1_{-\cS^\tau \leq \tau^{\alpha}} \, \Phi\|_\tau \leq C \, \nu_\tau^{3/4}\,  \|\Phi\|_\tau \,,
\end{equ}
where $\1_{- \cS^\tau \leq \tau^{\alpha}}$ stands for the diagonal operator whose kernel is $\1_{|\sqrt{R_\tau} p_{1:n}|^2/2 \leq \tau^{\alpha}}$.
\end{proposition}

\begin{remark}\label{rem:GA+norm}
As for the replacement lemmas, 
the dependence of $C^{(2)}_M$ and $ C$ on $M,\lambda$ and $\tau$ 
can be explicitly tracked and they can respectively be taken to be proportional to 
$\lambda/(1\vee M^{3/4})$ and $(1-\beta)^{-1/2} + \lambda\nu_\tau^{-3/2}\tau^{-(\beta-\alpha)/2}$, 
for any $0\leq\alpha<\beta<1$.  
\end{remark}

\begin{proof}
As in the proof of Replacement Lemmas~\ref{p:1stReplLemma} and~\ref{p:2ndReplLemma}, 
it suffices to prove the statement for $\Phi\in\fock{n}{}{\tau} \cap \core$ and any $n \geq 1$. 
Furthermore, by Lemma~\ref{l:SharpDiag}, the
fact that both $\cG^\tau_M$ and $\gensyx[f]$ are diagonal (see~\eqref{e:Lwf} for the latter), 
and formula~\eqref{e:diag}, we have 
\begin{equs}\label{e:GA+NormProof}
\|\cG_M^\tau \genapsh \Phi\|_\tau^2 &= \langle - \genamsh (\cG_M^\tau)^2 \genapsh \Phi, \Phi \rangle_{\tau, \di}\\
&\lesssim n! \int \Xi_n^\tau(\dd p_{1:n}) \, |\hat{\Phi}(p_{1:n})|^2 \lambda^2 \sum_{i=1}^n |\fe_1 \cdot p_i|^2 \nu_\tau^{3/2}\int_{\R^2} \frac{\dd q}{\big(1 + \Gamma_i^\tau + g_M^\tau(L^\tau(\Gamma_i^\tau)) \, \Gamma_i^{\fe_1}\big)^2}
\end{equs}
where we adopted the same notations as in~\eqref{e:NotationsRepl}. 
Let us now focus on the inner integral which will be bounded in two different ways 
to obtain~\eqref{e:GA+Norm} and~\eqref{e:GA+SmoothNorm}, respectively. 

For the former, since $g^\tau_M\geq M\nu_\tau$ (see~\eqref{e:gtau}) we have 
\begin{equs}
\int_{\R^2} \frac{\dd q}{\big(1 + \Gamma_i^\tau + g_M^\tau(L^\tau(\Gamma_i^\tau)) \, \Gamma_i^{\fe_1}\big)^2}&\lesssim \int_{\R^2}\frac{\dd q}{\big(\nu_\tau(M\vee 1)|\fe_1\cdot p_{1:n}|^2+|\sqrt{R_\tau}q|^2+M\nu_\tau|\fe_1\cdot q|^2\big)^2}\\
&\lesssim \frac{1}{\nu_\tau^{3/2}(1\vee M^{3/2})|\fe_1\cdot p_{1:n}|^2}
\end{equs}
where the last estimate follows by \eqref{e:Fact1} in Lemma~\ref{l:technical-integral}. Then, plugging the previous 
into~\eqref{e:GA+NormProof}, we obtain~\eqref{e:GA+Norm}. 

Let us now turn to~\eqref{e:GA+SmoothNorm}. The bound on the term containing $\genamsh$ follows 
from that on $\genapsh$ and a duality argument. Indeed, duality ensures that it suffices to 
control the $\fock{}{}{\tau}$-norm of $\1_{- \cS^\tau \leq \tau^{1/2}} \cG^\tau_M \genapsh \Phi$ 
via the right hand side of~\eqref{e:GA+SmoothNorm}. But, thanks to the definition of $\chi^\sharp$ in~\eqref{e:ChiSharp}, 
we have 
\begin{equ}
\1_{- \cS^\tau \leq \tau^{\alpha}} \cG^\tau_M \genapsh = \1_{- \cS^\tau \leq \tau^{\alpha}} \cG^\tau_M \genapsh \1_{- \cS^\tau \leq \tau^{\alpha}}
\end{equ}
so that we can apply~\eqref{e:GA+SmoothNorm} for $\genapsh$ to conclude. 

To prove the estimate on $\genapsh$, we split the integral in~\eqref{e:GA+NormProof} in two 
regions, corresponding to $|q|^2\leq \tau^{\beta}$ and $|q|^2>\tau^\beta$, for some given $\beta$ such that $\alpha<\beta<1$. 
In the first, we have 
$\Gamma_i^\tau=\frac12|\sqrt{R_\tau}(q,p_i-q, p_{2:n})|^2\lesssim \tau^{\beta}$, 
which implies that $g^\tau_M(L^\tau(\tfrac12|\sqrt{R_\tau}(q,p_i-q, p_{2:n})|^2))\gtrsim \lambda^{4/3}(1-\beta)^{2/3}$. 
Hence, 
\begin{equs}
\int \frac{\1_{\{|q|^2\leq \tau^{\beta}\}}\dd q}{\big(1 + \Gamma_i^\tau + g_M^\tau(L^\tau(\Gamma_i^\tau)) \, \Gamma_i^{\fe_1}\big)^2}&\lesssim \int\frac{\dd q}{\big(|\sqrt{R_\tau}q|^2+\lambda^{4/3}(1-\beta)^{2/3}\big(|\fe_1\cdot q|^2+|\fe_1\cdot p_{1:n}|^2\big)\big)^2}\\
&\lesssim \frac{1}{\lambda^{2}(1-\beta)|\fe_1\cdot p_{1:n}|^2}
\end{equs}
where in the last step we used \eqref{e:Fact1} from Lemma~\ref{l:technical-integral} 
with $A$ and $M$ therein respectively given by $A^2=\lambda^{4/3}(1-\beta)^{2/3}|\fe_1\cdot p_{1:n}|^2$ 
and $M=\nu_\tau^{-1}\lambda^{4/3}(1-\beta)^{2/3}$. 
For the other instead, we keep only the term $|\sqrt{R_\tau}q|^4 \gtrsim \nu_\tau^2 |q|^4$ at the denominator and obtain
\begin{equs}
\int \frac{\1_{\{|q|^2> \tau^{\beta}\}}\dd q}{\big(1 + \Gamma_i^\tau + g_M^\tau(L^\tau(\Gamma_i^\tau)) \, \Gamma_i^{\fe_1}\big)^2}\lesssim \int \frac{\1_{\{|q|^2> \tau^{\beta}\}}\dd q}{\nu_\tau^2|q|^4}\lesssim \frac{1}{\nu_\tau^2\tau^{\beta}}\leq \frac{1}{\nu_\tau^3\tau^{\beta-\alpha} |\fe_1\cdot p_{1:n}|^2} 
\end{equs}
the last bound being a consequence of $\tau^\alpha\geq |\sqrt{R_\tau}p_{1:n}|^2\gtrsim \nu_\tau|\fe_1\cdot p_{1:n}|^2$. 
Therefore,~\eqref{e:GA+SmoothNorm} follows by estimating~\eqref{e:GA+NormProof} 
via the two equations above and thus, the proof of the proposition is complete. 
%
%
%
%
%
%
%
\end{proof}

\section{The Resolvent and its ansatz}\label{sec:CVRes}

Now that we collected all the necessary bounds, we turn to the analysis of the resolvent and aim at 
obtaining the estimates needed to prove Theorems~\ref{th:Diffusivity} and~\ref{th:cor}. 
To properly state them, let $\phi$ be an element of $\fock{n_0}{}{\tau}$ for some fixed $n_0 \in \N$  
and $r^\tau$ be the resolvent applied to $\phi$, i.e.
\begin{equ}[e:res]
r^\tau \eqdef (1 - \gen)^{-1} \phi \, .
\end{equ}
Note that, if $n_0=0$, then $r^\tau$ coincides with $\phi$ (and there is thus nothing to prove) so that in what follows we will 
focus on the case of $n_0>0$. 
The main result of the present section is summarised in the following theorem. 

\begin{theorem}\label{thm:MainSec4}
Let $n_0\in\N$, $n_0\geq 1$, and $M$ be sufficiently large so that 
the constants $C^{(1)}_{M}$ and $C^{(2)}_M$ in~\eqref{e:2ndReplLemma} and~\eqref{e:GA+Norm} 
satisfy
\begin{equ} \label{e:AssumptionM}
C^{(1)}_{M}\vee C^{(2)}_M \leq \frac{1}{2}\, .
\end{equ}
Then, there exists a constant $C=C(\lambda,M,n_0)>0$ such that 
for all $\phi\in\fock{n_0}{}{\tau}$, we have 
\begin{equs}
\|r^\tau\|_\tau&\leq C\Big( \vertiii{\phi}_{\tau,M}+\|\cG^\tau_M\phi\|_\tau\Big)\,, \label{e:L2}\\
\Big|\|r^\tau\|_{\fock{}{1}{\tau}}^2-\|(\cG^\tau_M)^{1/2}\phi\|^2_\tau\Big|&\leq C\, \vertiii{\phi}_{\tau,M}\Big(\vertiii{\phi}_{\tau,M}+\|(\cG^\tau_M)^{1/2}\phi\|_\tau\Big)\label{e:H1}
\end{equs}
where $r^\tau$ is defined according to~\eqref{e:res} and the (semi-)norm $\vertiii{\cdot}$ is given by
\begin{equ}[e:semi]
\vertiii{\phi}_{\tau,M}\eqdef \nu_\tau^{1/4}\sqrt{|\log\nu_\tau|}\|(-\gensyx)^{1/2}\cG^\tau_M\phi\|_\tau+\|\genamsh\cG^\tau_M\phi\|_\tau\,.
\end{equ}
Furthermore, if there exists $\alpha$ such that $\mathrm{Supp}(\cF \phi) \subset \{- \cS^\tau \leq \tau^{\alpha}\}$, 
then there exists a constant $ C= C(\lambda,\alpha)>0$ for which 
\begin{equ}[e:L2Compact]
\|r^\tau-\cG^\tau_M\phi\|\leq C\,\nu_\tau^{1/4}\sqrt{|\log\nu_\tau|}\;\|\phi\|_\tau\,.
\end{equ}
\end{theorem}

As outlined in Section~\ref{sec:HeuRM}, the proof of the above statement is based on a 
approximation of the resolvent with a suitably chosen ansatz $s^\tau$ (point~\ref{i:pointII} of the roadmap). The latter is introduced 
in Section \ref{sec:Ansatz} where its main properties together with a control 
over the $\fock{}{1}{\tau}$-norm of the difference between $r^\tau$ and $s^\tau$ are derived. Most of the estimates 
we obtain therein are in terms of $\|(- \nu_\tau \gensyx)^{1/2} s^\tau\|_\tau$ whose control 
is only possible via the Recursive Replacement Lemma~\ref{p:2ndReplLemma}. 
This is detailed in Section \ref{sec:2ndReplLemma} (point~\ref{i:pointIII}), at the end of which 
all the results derived are put together and the proof of Theorem~\ref{thm:MainSec4} is completed.

\subsection{The ansatz and its properties}

\label{sec:Ansatz}

As explained in Section~\ref{sec:HeuRM}, the resolvent is a complicated object whose structure
intimately depends on the fixed point operator to~\eqref{e:OpFP}. The ansatz that we will now introduce
shares the same structure, i.e. that in~\eqref{e:True}, but with two twists: we replace $\cH^\tau$
with $\cG^\tau_M$ given in~\eqref{e:cgG}, and the whole of $\genap$ with $\genapsh$ in~\eqref{e:A+sharp}. 

Let $\phi$ be an element of $\fock{n_0}{}{\tau}$ for some fixed $n_0\ge 1$. 
We define the ansatz $s^\tau$ for the resolvent equation in~\eqref{e:res} according to 
\begin{equ}[e:AnsatzFormula]
s^\tau \eqdef \sum_{k=0}^{+\infty} (\cG_M^\tau \genapsh)^k \, \cG_M^\tau \phi \, .
\end{equ}
Equation~\eqref{e:AnsatzFormula} provides the full chaos decomposition of $s^\tau$ as, for any $k$, 
$(\cG_M^\tau \genapsh)^k \, \cG_M^\tau\in\fock{n_0+k}{}{\tau}$. 
Let us first show that, for fixed $\tau > 0$, the series at the right hand side is absolutely convergent in $\core$ so that 
$s^\tau$ is indeed a well-defined element of the core, and derive an expression for the action of the generator on it. 

\begin{lemma}\label{l:AnsatzBasic}
Let $\phi\in\fock{n_0}{}{\tau}$ for some $n_0\geq 1$, and assume that~\eqref{e:AssumptionM} holds. Then, 
for $\tau > 0$ fixed, the sum at the right hand side of~\eqref{e:AnsatzFormula}
is absolutely convergent in $\core$, so that $s^\tau$ is indeed well defined.
Moreover, it satisfies
\begin{equ} \label{e:AnsatzEquation}
s^\tau = \cG_M^\tau \genapsh s^\tau + \cG_M^\tau \phi \, ,
\end{equ}
and
\begin{equ} \label{e:GenAnsatz}
(1 - \gen) \, s^\tau = \phi + (-\genamsh \cG_M^\tau\genapsh -\cg_M^\tau) s^\eps - \genafl s^\eps - \genamsh \cG_M^\tau \phi \, . 
\end{equ}
\end{lemma}

\begin{proof}
Assume first that the sum in~\eqref{e:AnsatzFormula} indeed converges absolutely in $\core$, and let us prove \eqref{e:AnsatzEquation} and \eqref{e:GenAnsatz}. For \eqref{e:AnsatzEquation}, notice that
\begin{equ}
s^\tau = \cG_M^\tau \phi + \sum_{k \geq 0} (\cG_M^\tau \genapsh)^{k+1} \, \cG_M^\tau \phi = \cG_M^\tau \phi + \cG_M^\tau \genapsh s^\tau \, ,
\end{equ}
where in the last step we are able to collect $\cG_M^\tau \genapsh$ from the infinite sum, 
which is allowed as this operator is continuous on $\core$ and the sum is absolutely convergent therein. 

For \eqref{e:GenAnsatz}, by adding and subtracting $\cg_M^\tau s^\tau$ from $(1-\gen)s^\tau$, we obtain
\begin{equs}
(1-\gen) \, s^\tau &= (1 - \gensy + \cg_M^\tau) s^\tau -\genapsh s^\tau - \genamsh s^\tau - \cg_M^\tau s^\tau -\genafl s^\eps \\
&= (\cG_M^\tau)^{-1} s^\tau - \genapsh s^\tau - \genamsh s^\tau - \cg_M^\tau s^\tau - \genafl s^\tau \\
&= \phi + \big(-\genamsh \cG_M^\tau\genapsh-\cg_M^\tau\big) s^\tau - \genamsh \cG_M^\tau \phi - \genafl s^\tau \, ,
\end{equs}
the second step being a consequence of the definition of $\cG^\tau_M$ in~\eqref{e:cgG} and 
the last of two applications of~\eqref{e:AnsatzEquation}, one for the first two terms and the other for the third and fourth. 

Let us now argue that the series at right hand side of~\eqref{e:AnsatzFormula} converges absolutely in $\core$.
First, by Proposition~\ref{p:GA+Norm} and~\eqref{e:AssumptionM}, 
$\|\cG_M^\tau \genapsh\|_{\fock{}{}{\tau} \to \fock{}{}{\tau}} \leq 1/2$. Hence, a $k$-times iteration of this bound gives
\begin{equ} \label{e:AnsatzWellPosed1}
\|s^\tau_{n_0+k}\|_\tau = \|(\cG_M^\tau \genapsh)^k \cG_M^\tau \phi\|_\tau \leq 2^{-k} \|\cG_M^\tau \phi\|_\tau \leq 2^{-k} \|\phi\|_\tau \, .
\end{equ}
Now, since, by definition, it can be directly checked that $s^\tau_{n_0+k+1}=\cG_\eps \genapsh s^\tau_{n_0+k}$, 
we have 
\begin{equ} \label{e:AnsatzWellPosed2}
\|s^\tau_{n_0+k+1}\|_{\fock{}{2}{\tau}} = \|\cG_\eps \genapsh s^\tau_{n_0+k}\|_{\fock{}{2}{\tau}} \leq \|\genapsh s^\tau_{n_0+k}\|_\tau \lesssim_{\lambda, \tau} (n_0+k)^{3/2} \, \|s^\tau_{n_0+k}\|_\tau \, ,
\end{equ}
where in the last step we exploited~\eqref{e:ContA} which, even though stated (and proved) for $\gena_\pm$,
also holds for $\genash_\pm$ (with the same proof). Combining \eqref{e:AnsatzWellPosed1} and \eqref{e:AnsatzWellPosed2} concludes the proof.
\end{proof}

The (rather) explicit structure of $s^\tau$ allows to efficiently estimate both its $\fock{}{}{\tau}$ and  $\fock{}{1}{\tau}$-norms, 
as the next proposition shows. 

\begin{proposition} \label{p:L2Ansatz}
Assumte that~\eqref{e:AssumptionM} holds. 
Then, there exists a constant $C=C(\lambda,M)>0$ such that 
for all $\phi\in\fock{n_0}{}{\tau}$ for some $n_0\in\N$, we have 
\begin{equs} \label{e:L2AnsatzDiff}
\|s^\tau\|_\tau &\leq 2 \, \|\cG_M^\tau \phi\|_\tau \, ,\\
\Big|\|s^\tau\|_{\fock{}{1}{\tau}}^2 - \|(\cG_M^\tau)^{1/2} \phi\|_\tau^2\Big| &\leq C \|(- \nu_\tau \gensyx)^{1/2} s^\tau\|_\tau^2 \,  \label{e:H1Ansatz}
\end{equs}
where $s^\tau$ is defined according to~\eqref{e:AnsatzFormula}. 
Furthermore,  if $\phi$ is such that $\mathrm{Supp}(\cF \phi) \subset \{- \cS^\tau \leq \tau^{\alpha}\}$ for some $\alpha\in[0,1)$, 
then there exists a constant $C=C(\lambda,\alpha)>0$ independent of $\phi$ for which 
\begin{equ} \label{e:L2Ansatz}
\|s^\tau - \cG_M^\tau \phi\|_\tau \leq \nu_\tau^{3/4}\, C \|\phi\|_\tau \, ,
\end{equ}
\end{proposition}

\begin{remark}\label{rem:L2Ansatz}
As the proof shows, the constant $C$ in~\eqref{e:H1Ansatz} is the same at that of the Replacement Lemma, i.e. 
the one for which~\eqref{e:1stReplLemma} holds, and that in~\eqref{e:L2Ansatz} coincides with that of~\eqref{e:GA+SmoothNorm}. 
\end{remark}

\begin{proof}
As noted in the proof of Lemma~\ref{l:AnsatzBasic}, Proposition~\ref{p:GA+Norm} and \eqref{e:AssumptionM} 
imply that $\|\cG_M^\tau \genapsh\|_{\fock{}{}{\tau} \to \fock{}{}{\tau}} \leq 1/2$. Hence, 
\begin{equ}
\|s^\tau\|_\tau\leq \sum_k\|(\cG_M^\tau \genapsh)^k \cG_M^\tau \phi\|_\tau \leq \sum_k2^{-k} \|\cG_M^\tau \phi\|_\tau\leq 2\|\cG_M^\tau \phi\|_\tau \, ,
\end{equ}
from which~\eqref{e:L2AnsatzDiff} follows. 

For~\eqref{e:H1Ansatz}, Lemma~\ref{l:AnsatzBasic} ensures that $s^\tau\in\core\subset{\rm Dom}(\gen)$, 
so that~\eqref{e:Apriori} gives 
\begin{equ}[e:AprioriAnsatz]
\|s^\tau\|_{\fock{}{1}{\tau}}^2 = \langle (1-\gen) s^\tau, s^\tau \rangle_\tau \, .
\end{equ}
Plugging~\eqref{e:GenAnsatz} in the above formula, we obtain 
\begin{equ}
\|s^\tau\|_{\fock{}{1}{\tau}}^2 = \langle \phi, s^\tau\rangle_\tau + \langle (-\genamsh \cG_M^\tau \genapsh - \cg_M^\tau) s^\tau, s^\tau\rangle_\tau - \langle \genafl s^\tau, s^\tau\rangle_\tau - \langle \genamsh \cG_M^\tau \phi, s^\tau \rangle_\tau \, .
\end{equ}
Now, the third and fourth terms are $0$: the
third since $\genafl$ is skew-symmetric on $\core$ by Lemma~\ref{l:Asharp}; the fourth since $s^\tau \in \overline{\bigoplus_{n \geq n_0} \fock{n}{}{\tau}}$
while $\genamsh \cG^\tau_M \phi \in \fock{n_0-1}{}{\tau}$ and these spaces are orthogonal. 
For the first, we use again the orthogonality of Fock spaces with different indices and~\eqref{e:AnsatzFormula}, and deduce 
\begin{equ}
\langle \phi, s^\tau \rangle_\tau = \langle \phi, s_{n_0}^\tau \rangle_\tau = \langle \phi, \cG_M^\tau \phi\rangle_\tau = \|(\cG_M^\tau)^{1/2} \phi\|_\tau^2 \, .
\end{equ}
Thus we are left to control the second term, on which we apply~\eqref{e:1stReplLemma} 
in the Replacement Lemma~\ref{p:1stReplLemma}, and get \eqref{e:H1Ansatz}. 
\medskip

At last, we turn to~\eqref{e:L2Ansatz} whose proof is similar to that of~\eqref{e:L2AnsatzDiff}. 
By definition
\begin{equ}
\|s^\tau-\cG^\tau_M\phi\|_\tau\leq \sum_{k\geq1}\|(\cG_M^\tau \genapsh)^k \cG_M^\tau \phi\|_\tau\leq \|(\cG_M^\tau \genapsh) \cG_M^\tau \phi\|_\tau
\end{equ}
where we used that $\|\cG_M^\tau \genapsh\|_{\fock{}{}{\tau} \to \fock{}{}{\tau}} \leq 1/2$. 
Now, since $\cG^\tau_M$ is diagonal, the Fourier transforms of 
$\phi$ and $\cG^\tau_M\phi$ have the same support, so that~\eqref{e:GA+SmoothNorm} holds. 
Further exploiting that $\cG^\tau_M\leq 1$, we obtain 
\begin{equ}
\|(\cG_M^\tau \genapsh) \cG_M^\tau \phi\|_\tau\lesssim \nu_\tau^{3/4}\|\cG_M^\tau \phi\|_\tau\leq \nu_\tau^{3/4}\|\phi\|_\tau
\end{equ}
which concludes the proof. 
\end{proof}

\begin{remark}
  The bound~\eqref{e:L2Ansatz} shows in particular that, for $\phi$
  smooth, $s^\tau$ is asymptotically equal to $\cG_M^\tau \phi$ in
  $\fock{}{}{\tau}$, which means that $s^\tau$ is well-approximated by
  the $k=0$ summand at the right hand side
  of~\eqref{e:AnsatzFormula}. In $\fock{}{1}{\tau}$, the
  situation is drastically different: while the $\fock{}{1}{\tau}$ of the whole
    $s^\tau$ \emph{does} converge to a positive constant for $\tau\to\infty$, and this  
    (as can be seen from \eqref{e:H1Ansatz}; and this will play a crucial role in obtaining
    the limiting diffusion matrix), it can be shown that 
    the $\fock{}{1}{\tau}$ norm of each individual summand in~\eqref{e:AnsatzFormula}  tends to zero with $\tau\to\infty$. 
    In fact, one can prove that
    the summands that indeed contribute are those whose $k$ diverges with $\tau$ as $\log
    \log\tau$. 
\end{remark}

To conclude this section, we derive an apriori estimate on the difference of $s^\tau$ and $r^\tau$ in
 $\fock{}{1}{\tau}$ (and thus also in $\fock{}{}{\tau}$).

\begin{proposition}\label{p:H1Ansatz}
There exists a constant $C=C(\lambda,M)>0$ such that for any 
$\phi\in\fock{n_0}{}{\tau}$ for some $n_0 \geq 1$, we have 
\begin{equ}
\|r^\tau - s^\tau\|_{\fock{}{1}{\tau}} \leq C  \, \|\cN \, (- \nu_\tau \gensyx)^{1/2} s^\tau\|_\tau + \|\genamsh \cG^\tau_M \phi\|_\tau  \label{e:H1AnsatzRes}
\end{equ}
where $r^\tau$ and $s^\tau$ are respectively given as in~\eqref{e:res} and~\eqref{e:AnsatzFormula}, 
and $\cN$ is the number operator in Definition~\ref{def:NoOp}. 
\end{proposition}

\begin{remark}
The constant $C$ in~\eqref{e:H1AnsatzRes} can be taken to be proportional to $\lambda + \lambda^2 + M$. 
\end{remark}

\begin{proof}
Note that $r^\tau - s^\tau \in \mathrm{Dom}(\gen)$, so that Lemma \ref{lem:Apriori} gives
\begin{equ}[e:AprioriAnsatzRes]
\|r^\tau - s^\tau\|_{\fock{}{1}{\tau}}^2 = \langle (1 - \gen) (r^\tau - s^\tau), r^\tau - s^\tau \rangle_\tau \, .
\end{equ}
Thanks to the definition of $r^\tau$ and~\eqref{e:GenAnsatz}, we can compute $(1-\gen)(r^\tau-s^\tau)$ 
and plug it into the above expression. What we obtain is the sum of three terms 
\begin{equ}
-\langle [(- \genamsh) \cG_M^\tau \genapsh - \cg_M^\tau] s^\tau, r^\tau - s^\tau\rangle_\tau+ \langle \genafl s^\tau, r^\tau-s^\tau\rangle_\tau+\langle \genamsh \cG_M^\tau \phi, r^\tau - s^\tau \rangle_\tau
\end{equ}
that we denote by $\one$, $\two$ and $\three$, and separately estimate. For the first, we apply the Replacement Lemma~\ref{p:1stReplLemma} which gives
\begin{equ}[e:boundI]
|\one| \leq C \|(- \nu_\tau \gensyx)^{1/2} s^\tau\|_\tau \, \|r^\tau-s^\tau\|_{\fock{}{1}{\tau}} \, .
\end{equ}
Proposition~\ref{p:GSCgenafl} allows to control $\two$ as
\begin{equ}[e:boundII]
|\two| \lesssim \lambda \, \|\cN (- \nu_\tau \gensyx)^{1/2} s^\tau\|_\tau \, \|r^\tau-s^\tau\|_{\fock{}{1}{\tau}} \, .
\end{equ}
At last, we simply use Cauchy-Schwarz on $\three$, thus obtaining 
\begin{equ}[e:boundIII]
|\three| \leq \|\genamsh \cG_M^\tau \phi\|_\tau \, \|r^\tau-s^\tau\|_\tau \leq \|\genamsh \cG_M^\tau \phi\|_\tau \, \|r^\tau-s^\tau\|_{\fock{}{1}{\tau}} \, .
\end{equ}
As a consequence,~\eqref{e:H1AnsatzRes} follows upon 
plugging~\eqref{e:boundI}, ~\eqref{e:boundII} and \eqref{e:boundIII} at the right hand side of~\eqref{e:AprioriAnsatzRes}, 
and dividing both sides by $\|r^\tau-s^\tau\|_{\fock{}{1}{\tau}}$. 
\end{proof}

\subsection{Estimating the anisotropic norm of the ansatz and proof of Theorem~\ref{thm:MainSec4}}
\label{sec:2ndReplLemma}

For the estimates in the previous section to be effective and allow us to deduce Theorem~\ref{thm:MainSec4}, 
we need to control the $\|w(\cN) \, (- \nu_\tau \gensyx)^{1/2} \cdot\|_\tau$-norm of the ansatz $s^\tau$ 
in~\eqref{e:AnsatzFormula}, for suitable weights $w\colon\N\to\R_+$, 
and prove that it is small as $\tau \to +\infty$. 
This is the content of the next proposition. 

\begin{proposition} \label{p:AnisoH1Ansatz}
Assume that~\eqref{e:AssumptionM} holds. 
Let $w \colon \N \to \R_+$ be such that, for every $k \geq 0$, $w(k+1) \leq 2 \, w(k)$. 
Then, for every $\phi \in \fock{n_0}{}{\tau}$ for some $n_0 \geq 1$, 
\begin{equs} \label{e:AnisoH1Ansatz}
\|[w(\cN - n_0)]^{1/2} \, (- \nu_\tau \gensyx)^{1/2} s^\tau\|_\tau \leq 2 \|(- \gensyx[f^\tau_w])^{1/2} \, \cG_M^\tau \phi\|_\tau \, , 
\end{equs}
where 
\begin{equs}
f_w^\tau(x) \eqdef \nu_\tau \sum_{k\geq 0}w(k)\frac{1}{k!}\Big(\log\sqrt{\frac{\nu_\tau + g^\tau(x)}{\nu_\tau}}\Big)^k
\label{e:WeightFunction}
\end{equs}
and $g^\tau$ is the function in~\eqref{e:gtau}. 
%
%
\end{proposition}

\begin{remark} \label{rmk:FunctionWeight}
The reason why we imposed $w(\cdot+1) \leq 2 w(\cdot)$ is to ensure that $f_w$ in~\eqref{e:WeightFunction} satisfies 
the assumptions of the Recursive Replacement Lemma~\ref{p:2ndReplLemma} with constant $C_{f_w} = w(0)$ 
as can be immediately checked via a direct computation.
Notice also that, if $w$ is non-decreasing then $\cI_\tau(f_w) + w(0) \, \nu_\tau \leq f_w$ where $\cI_\tau$ 
is the map given in~\eqref{e:Itau}. 
\end{remark}

\begin{remark}\label{rem:H1}
Even though for the sake of the present paper, it would have been sufficient to consider $w$ given by the identity, 
we preferred to state Proposition~\ref{p:AnisoH1Ansatz} in broader generality 
as it reveals deeper properties of the ansatz and partially,  by~\eqref{e:H1AnsatzRes}, 
of the resolvent~\eqref{e:res}. 
Indeed, the bound~\eqref{e:AnisoH1Ansatz} shows that the anisotropic $\fock{}{1}{\tau}$-norm of the former 
{\it decays in chaos exponentially fast} as can be checked by taking $w$ to be $w(k)=c^k$ for $c\leq 2$,
in which case
  \begin{equation}
    \label{eq:fexp}
    f_w^\tau(x) = \nu_\tau^{1-\frac c2}\Big(\nu_\tau+g^\tau(x)\Big)^{\frac c2}
  \end{equation}
and that, upon choosing $c<2$, the weighted norm {\it vanishes as $\tau\to\infty$}, as one sees from \eqref{eq:fexp}. 
Even though~\eqref{e:H1AnsatzRes} per se does not allow us to deduce the same at the level of the resolvent 
(and we will not need it), we suspect that this is indeed the case. 
Note that weighted (in $\cN$) apriori bounds on the resolvent are highly non-trivial and, the only 
available ones, strongly rely on the graded sector condition (see~\cite{KLO}), 
which, as pointed out in Remark~\ref{rem:GSCgenafl}, 
does not hold in our context. 
\end{remark}

The proof of the above proposition is based on two main ingredients:  
the explicit iterative structure of the Ansatz $s^\tau$ and 
the Recursive Replacement Lemma~\ref{p:2ndReplLemma}. 
Before delving into the details, let us provide a brief heuristic that on the one hand shows how they come into play 
and on the other sheds some light on the role of the map $\cI_\tau$ in~\eqref{e:Itau}. 

Thanks to~\eqref{e:AnsatzFormula}, for every $N\geq 0$, the $(N+n_0)$-chaos 
component of $s^\tau$ is given by 
\begin{equ}
s^\tau_{N+n_0}=(\cG^\tau_M\genapsh)^N\cG^\tau_M\phi\,, 
\end{equ}
which implies
\begin{equs}
\| (- \nu_\tau \gensyx)^{1/2} s^\tau_{N+n_0}\|^2_\tau&=\| (- \nu_\tau \gensyx)^{1/2} (\cG^\tau_M\genapsh)^N\cG^\tau_M\phi\|_\tau^2\\
&=\Big\| \Big[(-\gensyx[\nu_\tau])^{1/2} \cG^\tau_M\genapsh\Big](\cG^\tau_M\genapsh)^{N-1}\cG^\tau_M\phi\Big\|_\tau^2\qquad\;\;\label{e:Itera1}
\end{equs}
where we have taken $w\equiv 1$ and used that, by~\eqref{e:Lwf}, $-\gensyx[\nu_\tau]=-\nu_\tau\gensyx$. 
To estimate the norm at the right hand side we apply~\eqref{e:2ndReplLemma} with $f\equiv\nu_\tau$ 
and, neglecting the error for the sake of this heuristic, we obtain 
\begin{equs}
\|(- \nu_\tau \gensyx)^{1/2} s^\tau_{N+n_0}\|^2_\tau\leq \| (-\gensyx[\cI_\tau(\nu_\tau)])^{1/2}(\cG^\tau_M\genapsh)^{N-1}\cG^\tau_M\phi\|_\tau^2\,,
\end{equs}
where $\cI_\tau$ is as in~\eqref{e:Itau} and the right hand side has precisely the same shape of~\eqref{e:Itera1} 
but with $\cI_\tau(\nu_\tau)$ and $N-1$ 
in place of $\nu_\tau$ and $N$. In other words, we can iterate the above procedure $N-1$ times, 
ultimately getting 
\begin{equ}[e:N+n0comp]
\| (- \nu_\tau \gensyx)^{1/2} s^\tau_{N+n_0}\|^2_\tau\leq \| (-\gensyx[\cI^{\circ N}_\tau(\nu_\tau)])^{1/2}\cG^\tau_M\phi\|_\tau^2
\end{equ}
for $\cI^{\circ N}_\tau$ the $N$-fold composition of the map $\cI_\tau$ with itself. 
On the other hand, $\cI^{\circ N}_\tau(\nu_\tau)$ is nothing but the $N$-step Picard iteration 
for the ODE
\begin{equ}[e:ODE2]
\dot{y}(\ell)=\frac12\frac{y(\ell)}{(\nu_\tau+g^\tau(\ell))^{3/2}}\;,\qquad y(0)=\nu_\tau\,,
\end{equ}
(which is the linearisation of~\eqref{e:ReplODE}) whose unique solution is 
\begin{equ}[e:final]
y(x)=\nu_\tau\sqrt{\frac{g^\tau(x)+\nu_\tau}{\nu_\tau}}=\sqrt{\nu_\tau}\sqrt{g^\tau(x)+\nu_\tau}\,.
\end{equ}

\begin{remark}\label{rem:1/22}
Here we see how crucial is the fact that the coefficient of the integral at the right hand side of~\eqref{e:Itau} 
is $1/2<1$. Indeed, the $1/2$ therein corresponds to the square-root inside the logarithm in~\eqref{e:WeightFunction} 
which in turn results in the square-root in the second term of~\eqref{e:final}. 
Since $g^\tau(x)+\nu_\tau\lesssim 1$, we get $y(x)\lesssim \nu_\tau/\sqrt{\nu_\tau}=\sqrt{\nu_\tau}\to0$ 
thus suggesting that the $\|(- \nu_\tau \gensyx)^{1/2} \cdot\|_\tau$-norm of the $(N+n_0)$-chaos component 
of the ansatz $s^\tau$ vanishes at rate $\nu_\tau^{1/4}$. 
\end{remark}

To make the previous argument rigorous, we need to be able to sum over $N$ the bound~\eqref{e:N+n0comp}, 
introduce the dependence on the number operator via the weight $w$ and 
take into account the error in~\eqref{e:2ndReplLemma}. 
This boils down to a stability type result for the Picard iteration of~\eqref{e:ODE2} 
which reduces to understand how $\cI_\tau$ iterates. This is precisely what the 
following lemma explains. 

\begin{lemma} \label{lem:Pk}
For $k\in\N$, let $p^\tau_k \colon \R_+ \to \R_+$ be the function given by
\begin{equ} \label{e:DefPk}
p_k^\tau(x) \eqdef \nu_\tau \, \frac{1}{k!} \Big(\log\sqrt{\frac{\nu_\tau + g^\tau(x)}{\nu_\tau}}\Big)^k \, ,
\end{equ}
where $g^\tau$ is defined according to~\eqref{e:gtau}. Then, 
$p^\tau_k$ is smooth, non-negative and non-decreasing and, for all $x>0$, it holds 
\begin{equ} \label{e:DecayPk}
\sum_{k \geq 0} 2^k \, p^\tau_k(x) \leq \nu_\tau + g^\tau(x) \, , \quad \sum_{k \geq 0} 2^k \, (p^\tau_k)'(x) \leq \frac{2}{3} \frac{\nu_\tau + g^\tau(x)}{x}\,.
\end{equ}
Furthermore, for every $k\geq 0$, we have 
\begin{equ} \label{e:RelationPk}
\cI_\tau(p_k^\tau) = p_{k+1}^\tau \, .
\end{equ}
\end{lemma}

The statement follows by basic computations, and thus  its proof is omitted.
We now have all the elements we need in order to prove Proposition~\ref{p:AnisoH1Ansatz}. 

\begin{proof}[of Proposition~\ref{p:AnisoH1Ansatz}]
We claim that for every $0 \leq m \leq N$ it holds
\begin{equ} \label{e:GoalAnisoH1Ansatz}
\sum_{k = m}^{N} w(k) \, \|(- \nu_\tau \gensyx)^{1/2} s^\tau_{n_0+k}\|^2
\leq \, \|(- \gensyx[f_{m,N}])^{1/2} \, s_{n_0+m}^\tau\|^2 \, ,
\end{equ}
where  $f_{m,N} \eqdef 2 \sum_{k=0}^{N-m} w(k+m) \, p^\tau_k$ and $p^\tau_k$ 
is as in~\eqref{e:DefPk}. Given the claim, then~\eqref{e:AnisoH1Ansatz} follows 
by taking $m=0$ and $N\to\infty$. 

To prove the claim, we fix $N \geq 0$ and proceed by induction from $m = N$ to $m = 0$. 
The case $m=N$ is obvious, 
so assume~\eqref{e:GoalAnisoH1Ansatz} holds for $m+1$, with $m\in\{0,\dots, N-1\}$. 
Then, by the induction hypothesis, we have 
\begin{equs}[e:Induction2ndReplLemma]
\sum_{k = m}^{N} &w(k) \, \|(- \nu_\tau \gensyx)^{1/2} s^\tau_{n_0+k}\|^2 \\
&\leq w(m) \|(- \nu_\tau \gensyx)^{1/2} s^\tau_{n_0+m}\|^2 + \|(-\gensyx[f_{m+1,N}])^{1/2} s^\tau_{n_0+m+1}\|^2 \\
&= w(m) \|(- \nu_\tau \gensyx)^{1/2} s^\tau_{n_0+m}\|^2 + \Big\|\Big[(-\gensyx[f_{m+1,N}])^{1/2} \, \cG_M^\tau \genapsh\Big] s^\tau_{n_0+m}\Big\|^2 \, .
\end{equs}
where in the last step we used~\eqref{e:AnsatzFormula}. 
In order to use the Recursive Replacement Lemma \ref{p:2ndReplLemma} to bound the second term, 
we need to check that $f_{m+1,N}$ satisfies its assumptions. 
Since $w(\cdot+1) \leq 2 \, w(\cdot)$, we have 
\begin{equs}
f_{m+1,N}(x) \leq \sum_{k \geq 0} 2^{k+1} w(m) \, p^\tau_k(x) \leq 2 w(m) \, (\nu_\tau + g^\tau(x)) \, ,\qquad \text{for all $x > 0$, }
\end{equs}
the last bound being a consequence of~\eqref{e:DecayPk}, and similarly for $f_{m+1,N}'$. Hence,~\eqref{e:Assf} 
holds for $f_{m+1,N}$ with $C_{f_{m+1,N}}=2 w(m)$. 
Thus, plugging~\eqref{e:2ndReplLemma} into~\eqref{e:Induction2ndReplLemma}, we deduce 
\begin{equs}
\sum_{k = m}^{N} &w(k) \, \|(- \nu_\tau \gensyx)^{1/2} s^\tau_{n_0+k}\|^2\\
&\leq  w(m) \|(- \nu_\tau \gensyx)^{1/2} s^\tau_{n_0+m}\|^2+\|(-\gensyx[\cI_\tau(f_{m+1,N}) +\nu_\tau w(m) ])^{1/2} \, s^\tau_{n_0+m}\|^2\\
&= \|(-\gensyx[\cI_\tau(f_{m+1,N}) + 2\nu_\tau w(m)])^{1/2} \, s^\tau_{n_0+m}\|^2\,,
\end{equs}
where we also used~\eqref{e:AssumptionM} to bound $\nu_\tau  C_{f_{m+1,N}} C^{(2)}_{M,\tau}\leq \nu_\tau w(m)$. 
But now, thanks to~\eqref{e:RelationPk},  $\cI_\tau(f_{m+1,N}) + 2\nu_\tau w(m) 
=f_{m,N}$, which concludes the proof of the induction and therefore of \eqref{e:AnisoH1Ansatz}.
%
\end{proof}

Thanks to the previous proposition and the results in the previous subsection, we 
can now prove Theorem~\ref{thm:MainSec4}. 

\begin{proof}[of Theorem~\ref{thm:MainSec4}]
We will first prove~\eqref{e:L2} and~\eqref{e:H1}, for which we take $\phi\in\fock{n_0}{}{\tau}$, for some $n_0\geq 1$. 
Notice first that
\begin{equ}[e:L2Proof1]
\|r^\tau\|_\tau\leq 
\|r^\tau-s^\tau\|_{\fock{}{1}{\tau}}+\|s^\tau\|_\tau\lesssim \|\cN \, (- \nu_\tau \gensyx)^{1/2} s^\tau\|_\tau + \|\genamsh \cG^\tau_M \phi\|_\tau +\|\cG_M^\tau \phi\|_\tau
\end{equ}
where in the last step we used~\eqref{e:H1AnsatzRes} and~\eqref{e:L2AnsatzDiff}. 
The triangle inequality together with the basic estimate $a\leq\sqrt{|a^2-b^2|}+b$, which holds 
for any $a,b>0$, implies 
\begin{equs}
\Big|\|r^\tau&\|_{\fock{}{1}{\tau}}^2-\|(\cG^\tau_M)^{1/2}\phi\|^2_\tau\Big|\leq \Big|\|r^\tau\|_{\fock{}{1}{\tau}}^2-\|s^\tau\|_{\fock{}{1}{\tau}}^2\Big|+\Big|\|s^\tau\|_{\fock{}{1}{\tau}}^2-\|(\cG^\tau_M)^{1/2}\phi\|^2_\tau\Big|\\
&\leq \|r^\tau-s^\tau\|_{\fock{}{1}{\tau}}\big(\|r^\tau-s^\tau\|_{\fock{}{1}{\tau}}+2\|s^\tau\|_{\fock{}{1}{\tau}}\big)+\Big|\|s^\tau\|_{\fock{}{1}{\tau}}^2-\|(\cG^\tau_M)^{1/2}\phi\|^2_\tau\Big|\\
&\lesssim \Big(\|r^\tau-s^\tau\|_{\fock{}{1}{\tau}}+\sqrt{|\|s^\tau\|_{\fock{}{1}{\tau}}^2-\|(\cG^\tau_M)^{1/2}\phi\|^2_\tau|}\Big)^2+\|r^\tau-s^\tau\|_{\fock{}{1}{\tau}}\|(\cG^\tau_M)^{1/2}\phi\|_\tau
\end{equs}
and thus, by~\eqref{e:H1AnsatzRes} and~\eqref{e:H1Ansatz}, we deduce 
\begin{equs}\label{e:H1Proof1}
\Big|\|r^\tau\|_{\fock{}{1}{\tau}}^2-\|(\cG^\tau_M)^{1/2}\phi\|^2_\tau\Big|\lesssim \Big(&\|\cN \, (- \nu_\tau \gensyx)^{1/2} s^\tau\|_\tau + \|\genamsh \cG^\tau_M \phi\|_\tau\Big)\times\\
&\times \Big(\|\cN \, (- \nu_\tau \gensyx)^{1/2} s^\tau\|_\tau + \|\genamsh \cG^\tau_M \phi\|_\tau+\|(\cG^\tau_M)^{1/2}\phi\|_\tau\Big)\,.
\end{equs}
 As a consequence, we are left to estimate the weighted norm of $s^\tau$. 
Upon taking $w\colon\N\to\R_+$ to be defined as $w(k)=n_0+k$, Proposition~\ref{p:AnisoH1Ansatz} 
gives 
\begin{equs}[e:BoundFinal]
\|[w(\cN - n_0)]^{1/2} \, (- \nu_\tau \gensyx)^{1/2} s^\tau\|_\tau &\leq 2 \|(- \gensyx[f^\tau_w])^{1/2} \, \cG_M^\tau \phi\|_\tau\\
&\lesssim \nu_\tau^{1/4}  \sqrt{|\log\nu_\tau|} \|(- \gensyx)^{1/2} \, \cG_M^\tau \phi\|_\tau
\end{equs}
which holds since the function $f^\tau_w(L^\tau)$, for $f^\tau_w$ as in~\eqref{e:WeightFunction} 
can be bounded uniformly in $x>0$ as 
\begin{equs}
f^\tau_w(L^\tau(x))&\lesssim \nu_\tau^{1/2} \sqrt{\nu_\tau + g^\tau(L^\tau(x))} \max\Big\{n_0\vee \log \sqrt{\frac{\nu_\tau+ g^\tau(L^\tau(x))}{\nu_\tau}}\Big\}\\
&\lesssim \nu_\tau^{1/2} \max\Big\{n_0\vee \log \nu_\tau^{-1/2}\Big\}\lesssim_{n_0} \nu_\tau^{1/2}  |\log \nu_\tau|
\end{equs}
where we used that $L^\tau$ is bounded uniformly in $x$
, and consequently so is $g^\tau(L^\tau)$, and, in the last step,
that $n_0$ is fixed. 
By plugging~\eqref{e:BoundFinal} into both~\eqref{e:L2Proof1} and~\eqref{e:H1Proof1} and 
recalling the definition of the seminorm in~\eqref{e:semi},~\eqref{e:L2} and~\eqref{e:H1} follow. 
\medskip

We now turn to~\eqref{e:L2Compact}, for which we proceed similarly to what was done in~\eqref{e:L2Proof1}. 
Assume $\phi$ is such that ${\rm Supp}(\hat\phi)\subset \{-\cS^\tau\leq\tau^\alpha\}$ for 
some $\alpha\in(0,1)$. Then, Proposition~\ref{p:H1Ansatz} and equations~\eqref{e:BoundFinal} and~\eqref{e:L2Ansatz} give
\begin{equs}
\|r^\tau-\cG^\tau_M\phi\|_\tau&\leq 
\|r^\tau-s^\tau\|_{\fock{}{1}{\tau}}+\|s^\tau-\cG^\tau_M\phi\|_\tau\\
&\lesssim \nu_\tau^{1/4}\sqrt{|\log\nu_\tau|}\|(-\gensyx)^{1/2}\cG^\tau_M\phi\|_\tau+\|\genamsh\cG^\tau_M\phi\|_\tau + \nu_\tau^{3/4} \|\phi\|_\tau\\
&\lesssim \nu_\tau^{1/4}\sqrt{|\log\nu_\tau|}\|(-\gensyx)^{1/2}\cG^\tau_M\phi\|_\tau+\nu_\tau^{3/4} \|\phi\|_\tau
\end{equs}
where the last step follows by~\eqref{e:GA+SmoothNorm}. 
To conclude, notice that, since $\cG^\tau_M$ is diagonal, then 
${\rm Supp}(\cF(\cG^\tau_M\phi))={\rm Supp}(\hat\phi)\subset \{-\cS^\tau\leq\tau^\alpha\}$ 
and on the latter set we have 
\begin{equ}
g^\tau_M(L^\tau(-\cS^\tau))\geq g^\tau_M(L^\tau(\tau^\alpha))\gtrsim_\alpha 1\,.
\end{equ}
Therefore, $(-\gensyx)^{1/2}\cG^\tau_M\lesssim_\alpha (\cG^\tau_M)^{1/2}\leq 1$, from which we get~\eqref{e:L2Compact} 
and thus conclude the proof of the statement. 
\end{proof}

\section{The superdiffusive Central Limit Theorem} \label{sec:CLT}

The goal of this and the next section is to show the main results of this paper, which, as we will soon see, 
are direct consequences of the analysis carried out in the previous. 

Here, we focus on Theorems~\ref{MT} and~\ref{th:cor}. 
As mentioned before, for the former we will exploit (a version of) the Trotter-Kato theorem which requires to show convergence of the
resolvent $(1-\gen)^{-1}$ to $(1-\geneff)^{-1}$. This is the content of the next Theorem which 
addresses point~\ref{i:pointIV} of Section~\ref{sec:HeuRM}. 
Before stating it, recall that, as mentioned in Section~\ref{sec:preliminaries},  
the two resolvents are defined on different spaces, namely $\L^\theta(\P^\tau)$ and $\L^\theta(\P)$, $\theta\in[1,\infty)$, 
which can be mapped to one another thanks to the surjective isometries $\iota^\tau$, $j^\tau$ introduced in Lemma~\ref{l:Emb}. 
%
%

\begin{theorem}\label{thm:StrongCvResL2}
For every $F\in\L^2(\P)$, we have
\begin{equ}[e:StrongCvResL2]
\iota^\tau (1 - \gen)^{-1} j^\tau F \xrightarrow[\tau \to \infty]{\L^2(\P)} (1 - \cL^{\mathrm{eff}})^{-1} F\,,
\end{equ}
where $\iota^\tau$ and $j^\tau$ are the maps defined in Lemma~\ref{l:Emb} and 
$\geneff$ is the operator introduced in Lemma~\ref{l:LimitGen}.
\end{theorem}
\begin{proof}
By Lemmas~\ref{l:Emb} and~\ref{lem:Contractivity}, and Lemma~\ref{l:LimitGen},
the operators $\iota^\tau (1 - \gen)^{-1} j^\tau$ and $(1 - \cL^{\mathrm{eff}})^{-1}$ are bounded, uniformly
in $\tau$, from $\L^2(\P)$ to itself, so that it suffices to verify~\eqref{e:StrongCvResL2} for $F$ belonging to a subset spanning a dense subspace
of $\mathbb L^2(\mathbb P)$, e.g. the set of $F$ of the form  $I^\tau_n(\phi)$ with $n\in\N$ and $\phi \in \fock{n}{}{}$ such that $\cF \phi$ has compact support  (recall the notation $I^\tau$ from Section \ref{sec:notation} for the isometry between Fock space and $L^2(\mathbb P)$).
Hence, we are left to verify that for such $n$ and $\phi$, we have
\begin{equ}[e:StrongCvResL2Bis]
\iota^\tau (1 - \gen)^{-1} j^\tau \phi \xrightarrow[\tau\to\infty]{\L^2(\P)} (1 - \cL^{\mathrm{eff}})^{-1} \phi\,.
\end{equ}
Notice that, by~\eqref{e:iotaeps}, for any (bounded) diagonal operator $\cD$ acting on $\fock{}{}{\tau}$ (see Definition~\ref{def:DiagOp}), the diagonal operator
$\iota^\tau \cD j^\tau$ acting on $\fock{}{}{\infty}$ has the same kernel as $\cD$, 
so that, with a slight abuse of notations, we denote it by $\iota^\tau \cD j^\tau = \cD$. Then we have
\begin{equs}
\|\iota^\tau &(1 - \gen)^{-1} j^\tau \phi - (1 - \cL^{\mathrm{eff}})^{-1} \phi\| \\
&\leq \|\iota^\tau (1 - \gen)^{-1} j^\tau \phi-\iota^\tau \cG^\tau_M j^\tau\phi\|+\|\cG^\tau_M\phi-(1 - \cL^{\mathrm{eff}})^{-1} \phi\| \, .
\end{equs}
By~\eqref{e:L2Ansatz}, the first term goes to $0$ as $\tau \to +\infty$. 
From the explicit expression of the kernels of $\cG_M^\tau$ and $\cL^{\mathrm{eff}}$, 
it is easy to see that also the second term vanishes as $\tau \to +\infty$, so that the proof is complete. 
\end{proof}

Thanks to the previous result, we are now ready to prove Theorems~\ref{MT} and~\ref{th:cor}. 

\begin{proof}[of Theorems~\ref{MT} and~\ref{th:cor}]
Let us immediately say that, given Theorem~\ref{MT}, Theorem~\ref{th:cor} follows immediately by taking 
$F^\tau_i(\eta^\tau) \eqdef f_i(\eta^\tau(\phi_i))$, for $i=1,\dots,k$, $\phi_1,\dots,\phi_k\in\cS(\R^2)$ 
and $f_1,\dots,f_n$ bounded continuous functions on $\R$.

Turning to Theorem~\ref{MT}, we split its proof in two parts. First, we will show strong convergence of the semigroup, and
more precisely that~\eqref{e:StrongCv} holds for any $\theta\in[1,\infty)$.
Then, we will use the latter to deduce the convergence in~\eqref{e:FinDimCv}.
\medskip

\noindent{\sl Part I: Convergence of the semigroup.}
Thanks to Lemma~\ref{l:Emb} and Theorem~\ref{thm:StrongCvResL2},
we are precisely in the setting of the
Trotter-Kato theorem~\cite[Theorem 4.2]{IK}, which implies that
for every $F\in\L^2(\P)$ and $T > 0$, it holds
\begin{equ} 
\sup_{t \in [0,T]} \|\iota^\tau P_t^\tau j^\tau F - P_t^{\mathrm{eff}} F\|_{\mathbb L^2(\mathbb P)} \xrightarrow[\tau \to +\infty]{} 0 \, .
\end{equ}
This can be rewritten in the following form. Let $(F_\tau)_\tau$ be sequence
such that $F_\tau \in\ L^2(\P^\tau)$ and $\|\iota^\tau F_\tau-F\|_{\L^2(\P)}\to 0$ as $\tau\to+\infty$. 
Then noticing that $\|\iota^\tau F_\tau-F\|_{\L^2(\P)} = \|F^\tau - j^\tau F\|_{\L^2(\P^\tau)}$ 
and using the contractivity of the semi-groups stated in 
Lemmas~\ref{lem:Contractivity} and~\ref{l:LimitGen}),
we also deduce that
\begin{equ} \label{e:StrongCvSemiGroupL2}
\sup_{t \in [0,T]} \|\iota^\tau P_t^\tau F^\tau - P_t^{\mathrm{eff}} F\|_{\mathbb L^2(\mathbb P)} \xrightarrow[\tau \to +\infty]{} 0 \, .
\end{equ}
Hence, we are left to extend the above to $T=\infty$ and any value of $\theta\in[1,\infty)$. 
Let us begin with the former. 

Without loss of generality, we can take $F^\tau$ and $F$ with mean zero. Let $T>0$ and note that,
for any $t\geq T$, the contractivity of $P^\tau$ and $P^\eff$ implies
\begin{equs}
\|\iota^\tau P_t^\tau F^\tau - P_t^\eff F\|_{\L^2(\P)} &\leq \|P_t^\tau F^\tau\|_{\L^2(\mathbb P^\tau)} + \|P_t^{\mathrm{eff}} F\|_{\mathbb L^2(\mathbb P)} \\
&\leq \|P_T^\tau F^\tau\|_{\L^2(\P^\tau)} + \|P_T^{\mathrm{eff}} F\|_{\mathbb L^2(\mathbb P)} \\
&\leq \|\iota^\tau P_T^\tau F^\tau - P_T^{\mathrm{eff}} F\|_{\mathbb L^2(\mathbb P)} + 2 \, \|P_T^{\mathrm{eff}} F\|_{\mathbb L^2(\mathbb P)} \, ,
\end{equs}
where we also used that $\iota^\tau$ is an isometry. Therefore,
\begin{equ}
\sup_{t \geq 0} \|\iota^\tau P_t^\tau F^\tau - P_t^{\mathrm{eff}} F\|_{\mathbb L^2(\mathbb P)} \leq \sup_{t \in [0,T]} \|\iota^\tau P_t^\tau F^\tau - P_t^{\mathrm{eff}} F\|_{\mathbb L^2(\mathbb P)} + 2 \, \|P_T^{\mathrm{eff}} F\|_{\mathbb L^2(\mathbb P)} \, .
\end{equ}
We can now take the $\limsup$ as $\tau\to+\infty$ at both sides so that, at the right hand side, the first term vanishes
and what is left goes to $0$ in the $T\to\infty$ limit by~\eqref{e:EffErgodicity}. 
Hence,~\eqref{e:StrongCv} holds for $\theta=2$.

To prove that it also holds for $\theta\in[1,\infty)$, let $(F^\tau)_\tau$ be
such that $F^\tau\in\L^\theta(\P^\tau)$ and $\|\iota^\tau F^\tau-F\|_{\L^\theta(\P)}\to 0$ as $\tau\to+\infty$.
Applying~\eqref{e:Contractivity} with $\theta$ therein equal to $+\infty$, it is immediate to see that
~\eqref{e:StrongCv} for $\theta=2$ extends to $\theta\in[1,\infty)$ if $F^\tau$ (and $F$)
is bounded by a constant, uniformly in $\tau$.
For general $F^\tau$ and $F$, let $M>0$ and $\chi_M(x)\eqdef (-M)\vee (M\wedge x)$.
Clearly, $\chi^M(F^\tau)$ and $\chi^M(F)$ are bounded and satisfy
$\|\iota^\tau \chi^M(F^\tau)-\chi^M(F)\|_{\L^{\theta}(\P)}\to 0$ as $\tau\to+\infty$.
Using once again the contractivity of the semigroups, we obtain the bound
\begin{equs}
\|\iota^\tau P_t^\tau F^\tau - P_t^{\mathrm{eff}} F\|_{\mathbb L^\theta(\mathbb P)}
&\leq \|\iota^\tau P_t^\tau \chi^M(F^\tau) - P_t^{\mathrm{eff}} \chi^M(F)\|_{\mathbb L^\theta(\mathbb P)} \\
&\quad + 2 \, \|F - \chi^M(F)\|_{\mathbb L^\theta(\mathbb P)} + \|\iota^\tau F^\tau - F\|_{\mathbb L^\theta(\mathbb P)}\, ,
\end{equs}
and we take the supremum in $t\geq 0$. Now, we can pass to the limit first in $\tau\to+\infty$,
so that the first and last terms vanish (the former since $\chi^M(F^\tau)$ and $\chi^M(F)$ are bounded
and the latter by construction) and then in $M\to\infty$ which ensures that also the second term
converges to $0$.
Therefore,~\eqref{e:StrongCv} holds for any $\theta\in[1,\infty)$.
\medskip

\noindent{\sl Part II: Convergence of the finite-dimensional laws.}
We first rewrite the conditional expectations in \eqref{e:FinDimCv}.
To do so, let $0=t_0\leq t_1\leq\dots\leq t_k$ and set
$\Delta t_i \eqdef t_{i} - t_{i-1}$ for $i \in \{1,...,k\}$.
Define recursively the functionals $G_i^\tau$ and $G_i^{\eff}$, $i \in \{1, ..., k+1\}$, according to
$G_{k+1}^\tau = G_{k+1}^{\mathrm{eff}} \eqdef 1$ and for every $i \in \{1,...,k\}$
\begin{equ} \label{e:DefMarkovDomino}
G^\tau_{i} = P^\tau_{\Delta t_i} \big(F_i^\tau  G^\tau_{i+1}\big) \qquad\mbox{and}\qquad G^{\mathrm{eff}}_{i} = P^{\mathrm{eff}}_{\Delta t_i} \big(F_i  G^{\mathrm{eff}}_{i+1}\big) \, .
\end{equ}
Note that both $\bigcap_{\theta \geq 1} \mathbb L^\theta(\mathbb P^{\tau})$ and
$\bigcap_{\theta \geq 1} \mathbb L^\theta(\mathbb P)$ are stable under multiplication
and the action of the respective semigroups.
Thus, the above recursion is well-defined and, for every $i\in\{1,\dots,k\}$,
$G_i^\tau \in \bigcap_{\theta \geq 1} \mathbb L^\theta(\mathbb P^{\tau})$ and
$G_i^{\mathrm{eff}}\in \bigcap_{\theta \geq 1} \mathbb L^\theta(\mathbb P)$.
The Markov property implies that the two terms in~\eqref{e:FinDimCv} can be written as
\begin{equ}\label{e:MarkovDomino}
\Exp^\tau_{u_0}\big[F_1^\tau\big(u^{\tau}_{t_1}\big) \, ... \, F_k^\tau \big(u^{\tau}_{t_k}\big)\big]
= \iota^\tau G_1^\tau (u_0)  \quad\text{and}\quad
\Exp^{\mathrm{eff}}_{u_0}\big[F_1\big(u^{\mathrm{eff}}_{t_1} \big) \, ... \, F_k\big(u^{\mathrm{eff}}_{t_k}\big)\big] = G_1^{\mathrm{eff}}(u_0) \, .
\end{equ}
Therefore, if we show that for every $i\in\{1,\dots, k+1\}$ it holds
$\|\iota^\tau G_i^\tau-G_i^\eff\|_{\L^\theta(\P)}\to0$ as $\tau\to+\infty$, for every $\theta\in[1,\infty)$,
then~\eqref{e:FinDimCv} follows at once. To prove it, we proceed by (backward) induction.
The result is obvious for $i = k+1$, since
$\iota^\tau G_{k+1}^\tau = 1 = G_{k+1}^{\eff}$.
Assume now it holds for $i+1$, and let us verify the same is true for $i$.
Since by assumption
$\|\iota^\tau F_i^\tau-F_i\|_{\L^\theta(\P)}\to 0$ for all $\theta\in[1,\infty)$,
and, by induction, the same holds with $G^\tau_{i+1}$ and $G^\eff_{i+1}$ in place of
$ F_i^\tau$ and $F_i$, the same must be true for their product.
By applying~\eqref{e:StrongCv}, we obtain 
\begin{equ}
G_i^\tau = P_{\Delta t_i}^\tau \big(F_i^\tau  G_{i+1}^\tau\big) \xrightarrow[\tau \to \infty]{\L^\theta(\P)} P_{\Delta t_i}^{\eff} \big(F_i G_{i+1}^{\eff}\big) = G_i^{\eff}
\end{equ}
which completes the proof.
\end{proof}

\section{The large-scale behaviour of the Diffusion Matrix}

\label{sec:D}
In this final section, we show Theorem~\ref{th:Diffusivity}, whose proof amounts to determine 
the asymptotic behaviour of the diffusion matrix $t\mapsto D(t)$ in~\eqref{e:FormalDiffusivity} as $t\to\infty$. 

To do so, let us first make 
a couple of remarks. By Proposition~\ref{p:GK} (for $\tau=1$), the only non-trivial entry of $D$ is $D_{1,1}$ 
which can be written in terms of the rescaled process $u^\tau$ as 
\begin{equ} \label{e:DiffRescaling}
D_{1,1}^\tau(s) = \nu_\tau D_{1,1}(\tau s) \, . \qquad s\geq 0\,.
\end{equ}
In particular, this means that Theorem~\ref{th:Diffusivity} follows provided we show that, 
for $s=1$, the left hand side converges as $\tau\to\infty$ to $C_{\mathrm{eff}}(\lambda)$ in~\eqref{e:EffConstant}. 
Now, before that, we need to manipulate the Green-Kubo formula in~\eqref{e:Dnew} 
so to unveil the connection with the results in the previous section.

%
%

\subsection{The diffusivity as the energy associated to a parabolic problem}

In order to simplify our analysis and ensure that all the quantities under study are well-defined, 
we approximate the spatial integral at the right hand side of~\eqref{e:Dnew} 
in such a way that 
the functional $\cM^\tau$  in~\eqref{e:cM} is written in terms of the operator $\genap$. 
Loosely speaking, $\cM^\tau[\eta^\tau]$ is the antiderivative in the 
$\fe_1$ direction of the nonlinearity $\cN^\tau[\eta^\tau]$, which, when tested against a test function $\psi\in\cS(\R^2)\cap\fock{1}{}{\tau}$, 
satisfies  $\cN^\tau[\eta^\tau](\psi)=\genap\eta^\tau(\psi)$. 
Thus, we would like to say that $\cM^\tau[\eta^\tau](\psi)$ is $\genap\eta^\tau(\partial_{\fe_1}^{-1}\psi)$, 
but this only makes sense provided ${\rm Supp}(\hat\psi)\cap (\fe_2\R)=\emptyset$.  

To do so, let us fix an even symmetric function $\psi \in \cS(\R^2)$  such that $\|\psi\|_{L^2(\R^2)} = 1$ 
and whose Fourier transform $\hat \psi$ has compact support in a ball of radius $1$ around the origin 
contained in $(\fe_2 \R)^c$, with $\fe_2 \R$ denoting the vertical axis.
For $\delta>0$, set $\psi^\delta(\cdot)\eqdef\delta\psi(\delta\,\cdot)$, so that the convolution of $\psi^\delta$ with itself 
converges to $1$ as $\delta\to 0$. 
Then, we hope that 
\begin{equs}[e:Hope]
\int_{\R^2}\Exp\big[\cM^{\tau}[u^{\tau}_r](x) \cM^{\tau}[u^{\tau}_0](0)\big] \, (\psi^\delta\ast\psi^\delta)(x)\dd x\stackrel{\delta\to0} \sim\int_{\R^2}\Exp\big[\cM^{\tau}[u^{\tau}_r](x) \cM^{\tau}[u^{\tau}_0](0)\big] \, \dd x \,.
\end{equs}
The advantage of the left hand side is that, by 
translation invariance and the symmetry of $\psi^\delta$, it equals  
\begin{equs}
\int_{\R^4}\Exp&\big[\cM^{\tau}[u^{\tau}_r](x-y) \cM^{\tau}[u^{\tau}_0](-y)\big] \, \psi^\delta(x-y)\psi^\delta(y)\dd x\dd y\\
&=\Exp\big[\cM^{\tau}[u^{\tau}_r](\psi^\delta) \cM^{\tau}[u^{\tau}_0](\psi^\delta)\big] =\frac{1}{\lambda^2\nu_\tau^{3/2}} \langle P^\tau_r \fm,\fm\rangle_\tau\label{e:IntGK}
\end{equs}
where in the last step we denoted by $\fm\in\fock{2}{}{\tau}\cap\core$ the kernel of $\lambda\nu_\tau^{3/4}\cM^{\tau,\delta}$, 
and $P^\tau$ is the semigroup of the rescaled SBE. 
Now, since by construction ${\rm Supp}(\hat\psi^\delta)\subset (\fe_2\R)^c$, $\partial_{\fe_1}^{-1} \psi^\delta$ is 
well-defined and a direct computation shows that 
\begin{equ}[e:mtaudelta]
\fm=\genap \partial_{\fe_1}^{-1} \psi^\delta\,.
\end{equ}

In the next lemma we show that the Green-Kubo formula~\eqref{e:Dnew} can be expressed 
in the form of the energy of the solution of the parabolic equation given by  
\begin{equ} \label{e:ODEPhi}
(\partial_t - \gen) \, \Phi^{\tau, \delta} = \fm \, ,\qquad \restriction{\Phi^{\tau, \delta}}{t=0} = 0
\end{equ}
which is explicit and given by 
\begin{equ}[e:ODEPhiExp]
\Phi_t^{\tau, \delta} = \int_0^t P_s^\tau \fm \, \dd s \, , \qquad t\geq 0. 
\end{equ}

\begin{proposition}\label{p:GKenergy}
For any $t,\tau>0$ fixed, the following equality holds 
\begin{equ} \label{e:GKenergy}
D^\tau_{1,1}(t) = \nu_\tau + \lim_{\delta \to 0} \, \frac{\cE^\tau_t(\Phi^{\tau, \delta})}{t} \, .
\end{equ}
where, for any $\delta>0$, $\Phi^{\delta,\tau}$ is the unique solution to~\eqref{e:ODEPhi} and, 
for any $(f_s)_{s\geq 0}\in  C^1([0,t], \fock{}{}{\tau}) \cap C([0,t], \mathrm{Dom}(\gen))$, 
the positive quadratic form $\cE^\tau_t$ is defined as
\begin{equ} \label{e:AprioriEnergy}
\cE^\tau_t(f) \eqdef  \|f_t\|_\tau^2 + 2\int_0^t \|(- \gensy)^{1/2} f_s\|_\tau^2 \, \dd s =
\|f_0\|_\tau^2 + 2\int_0^t \langle (\partial_s - \gen) f_s, f_s\rangle_\tau \, \dd s \, .
\end{equ}
\end{proposition}

\begin{proof}
Let us first briefly argue the second equality in~\eqref{e:AprioriEnergy}. 
A  proof similar to that of Lemma~\ref{lem:Apriori}, gives 
that for every $\phi\in\mathrm{Dom}(\gen)$, $\|(-\gensy)^{1/2} \phi\|_\tau^2 = \langle (-\gen) \, \phi, \phi\rangle_\tau$. 
Therefore,
\begin{equs}
  \int_0^t \langle (\partial_s - \gen) f_s, f_s\rangle_\tau \, \dd s =
  \int_0^t \langle \partial_s f_s, f_s\rangle_\tau \, \dd s
  + \int_0^t \langle (-\gen) f_s, f_s\rangle_\tau \, \dd s 
\end{equs}
from which the conclusion follows by integration by parts. 

Next, we address~\eqref{e:GKenergy}, which essentially boils down to show that~\eqref{e:Hope} holds. 
For this, note that $|\psi^\delta * \psi^\delta(x) - 1| = |\psi*\psi (\delta x) - \psi*\psi(0)| \lesssim \delta \, |x|$ 
and, thus, the polynomial decay of correlations in~\eqref{e:thesame} implies 
\begin{equ}
\Big|\mathbf E[\langle \cM^{\tau}[u_r^\tau], \psi^\delta\rangle \langle \cM^{\tau}[u_0^\tau], \psi^\delta \rangle] - \int_{\R^2} \mathbf E\big[\cM^{\tau}[u^{\tau}_r](x) \cM^{\tau}[u^{\tau}_0](0)\big] \, \dd x \Big| \lesssim_\tau \delta \, .
\end{equ}
Together with~\eqref{e:Dnew} and~\eqref{e:IntGK}, the previous gives 
\begin{equs}
D_{1,1}^\tau(t) &= \nu_\tau + \lim_{\delta \to 0} \frac{2 \lambda^2 \nu_\tau^{3/2}}{t} \int_0^t \dd s \int_0^s \dd r \, \mathbf E[\langle \cM^{\tau}[u_r^\tau], \psi^\delta\rangle \langle \cM^{\tau}[u_0^\tau], \psi^\delta\rangle] \\
&=\nu_\tau + \lim_{\delta \to 0} \frac{2 }{t} \int_0^t \dd s \int_0^s \dd r \, \langle P^\tau_r \fm,\fm\rangle_\tau = \nu_\tau + \lim_{\delta \to 0} \frac{2 }{t} \int_0^t \dd s \, \langle \phi^{\tau,\delta},\fm\rangle_\tau
\end{equs}
where in the last step we used~\eqref{e:ODEPhiExp}. 
Since by~\eqref{e:ODEPhi}, $\fm=(\partial_t - \gen) \, \Phi^{\tau, \delta}$, the second equality of~\eqref{e:AprioriEnergy} 
implies~\eqref{e:GKenergy} and the proof of the statement is complete. 
\end{proof}

To combine the previous proposition with the results from Section~\ref{sec:CVRes}, 
we need to understand the relation between the parabolic equation in~\eqref{e:ODEPhi} 
and the elliptic one given by 
\begin{equ}\label{e:elliptic}
(1 - \gen) \, r^{\tau, \delta} = \fm \, . 
\end{equ}
This is the content of the next (rather straightforward) lemma\footnote{See also~\cite[Section 9]{Morfe}, for a similar a similar analysis in a different context.}. 

\begin{lemma} \label{l:ParabolicElliptic}
Let $\Phi^{\tau,\delta}$ and $r^{\tau,\delta}$ be respectively the solutions of~\eqref{e:ODEPhi} and~\eqref{e:elliptic}. 
Then, for every $T > 0$, there exists a constant $C=C(T)>0$ such that 
\begin{equ} \label{e:ParabolicElliptic}
\sup_{t \in [0,T]} \cE^\tau_t(\Phi^{\tau, \delta} - {r}^{\tau, \delta}) \leq C \|r^{\tau, \delta}\|_\tau^2 \, .
\end{equ}
\end{lemma}

\begin{proof}
First, recall that $(\partial_t - \gen) \, \Phi^{\tau, \delta} = \fm$ while $(\partial_t - \gen) \, {r}^{\tau, \delta} = - \gen r^{\tau, \delta} = \fm - r^{\tau, \delta}$, since $r^{\tau, \delta}$ does not depend on $t$. Thus, by~\eqref{e:AprioriEnergy} and Cauchy-Schwarz, we deduce 
\begin{equs}
\cE^\tau_t(\Phi^{\tau, \delta} - {r}^{\tau, \delta}) &= \|r^{\tau, \delta}\|_\tau^2 + 2\int_0^t \langle r^{\tau, \delta}, \Phi^{\tau,\delta}_s - r^{\tau, \delta} \rangle_\tau \, \dd s \\
&\leq (1 + t) \, \|r^{\tau, \delta}\|_\tau^2 + \int_0^t  \|\Phi^{\tau,\delta}_s - r^{\tau, \delta}\|_\tau^2 \, \dd s \\
&\leq  (1 + t) \, \|r^{\tau, \delta}\|^2_\tau + \int_0^t \cE_{s, \tau}(\Phi^{\tau,\delta} - {r}^{\tau, \delta}) \, \dd s \, ,
\end{equs}
where the last step follows by the first equality in~\eqref{e:AprioriEnergy}. 
Then, Grönwall's lemma gives 
\begin{equ}
\cE^\tau_t(\Phi^{\tau,\delta} - {r}^{\tau, \delta}) \leq  (1 + t) \, e^{t} \, \|r^{\tau, \delta}\|_\tau^2 \, ,
\end{equ}
which concludes the proof.
\end{proof}

\subsection{Proof of Theorem~\ref{th:Diffusivity}}

In view of Proposition~\ref{p:GKenergy} and Lemma~\ref{l:ParabolicElliptic}, we are left to identify 
the limit as first $\delta\to0$ and then $\tau\to\infty$ of   
$\cE^\tau_t(r^{\tau,\delta})$ for $r^{\tau,\delta}$ the solution to the elliptic equation~\eqref{e:elliptic}. 
This is the point at which Theorem~\ref{thm:MainSec4} comes into play. To apply it, 
we need to estimate the right hand sides of~\eqref{e:L2} and~\eqref{e:H1} for $\phi$ given by $\fm$ in~\eqref{e:mtaudelta}. 
We want to control the (semi-)norms appearing therein using the results of Section~\ref{sec:apprG}, 
which are though stated in terms of the sharp operators $\genash_\pm$. Hence, let 
us consider the following slight modification of $\fm$ 
\begin{equ}[e:tfm]
\tfm \eqdef \genapsh \partial_{\fe_1}^{-1} \psi^\delta
\end{equ}
and show that its resolvent is close to $r^{\tau,\delta}$ in $\fock{}{1}{\tau}$. 

\begin{lemma}\label{l:DiffBoringLemma}
Let $r^{\tau,\delta}$ and $\tilde r^{\tau,\delta}$ be respectively the solutions to~\eqref{e:elliptic} and 
\begin{equ}[e:elliptictilde]
(1 - \gen) \, \tilde r^{\tau, \delta} = \tfm\,.
\end{equ}
Then, there exists a constant $C=C(\lambda)>0$, independent of $\delta$ such that 
\begin{equ}[e:rtilder]
\|r^{\tau, \delta}-\tilde r^{\tau, \delta}\|_{\fock{}{1}{\tau}}\leq C \nu_\tau^{1/2}\delta\,. 
\end{equ}
\end{lemma}

\begin{proof}
By~\eqref{e:Apriori} and the equations~\eqref{e:elliptic} and~\eqref{e:elliptictilde}, we have 
\begin{equs}
\|r^{\tau, \delta}-\tilde r^{\tau, \delta}\|_{\fock{}{1}{\tau}}^2&=\langle (1 - \gen)(r^{\tau, \delta}-\tilde r^{\tau, \delta}), r^{\tau, \delta}-\tilde r^{\tau, \delta}\rangle_\tau\\
&=\langle \fm-\tfm, r^{\tau, \delta}-\tilde r^{\tau, \delta}\rangle_\tau\leq \|\fm-\tfm\|_{\fock{}{-1}{\tau}}\|r^{\tau, \delta}-\tilde r^{\tau, \delta}\|_{\fock{}{1}{\tau}}
\end{equs}
so that we are left to control the $\fock{}{-1}{\tau}$-norm of the difference between $\fm$ and $\tfm$. 
Their definition in~\eqref{e:mtaudelta} and~\eqref{e:tfm} together 
with the fact that $\genap=\genapsh+\genafl_+$, implies that 
$\fm-\tfm = \genafl \partial_{\fe_1}^{-1} \psi^\delta$. Thus, 
\begin{equ}
 \|\genafl \partial_{\fe_1}^{-1} \psi^\delta\|^2_{\fock{}{-1}{\tau}} \leq \|\genafl \partial_{\fe_1}^{-1} \psi^\delta\|^2_{\tau}\lesssim \lambda^2 \nu_\tau^{3/2} \int_{\R^2} \dd p \,  |\hat \psi^\delta(p)|^2 \int_{\R^2} \dd q \, \1_{|\sqrt{R^\tau} q| \leq 2 \delta} \lesssim \lambda^2 \nu_\tau \delta^2
\end{equ}
where we used that $\|\psi^\delta\|_{L^2(\R^2)} = 1$ and 
that the volume of $\{q\colon|\sqrt{R^\tau} q| \leq 2 \delta\}$ is of order $\nu_\tau^{-1/2} \delta^2$. 
Hence,~\eqref{e:rtilder} follows at once. 
\end{proof}

Let us now analyse $\tfm$ and estimate all the (semi-)norms needed to conclude our proof. 

\begin{lemma}\label{l:DiffReplLemma}
Let $\tfm$ be defined according to~\eqref{e:tfm}. Then, there exists a constant $C=C(\lambda,M)>0$ such that 
\begin{equs}
\limsup_{\delta\to 0} \vertiii{\tfm}_{\tau,M}\vee\|\cG^\tau_M\tfm\|_\tau&\leq C \nu_\tau^{1/4}|\log\nu_\tau|\label{e:negl}\\
\limsup_{\delta \to 0} \Big|\|(\cG_M^\tau)^{1/2} \tfm\|_\tau^2 - \tfrac12C_\eff(\lambda)\Big| &\leq C \nu_\tau\label{e:DiffAnsGhalf}
\end{equs}
where $\vertiii{\cdot}_{\tau,M}$ is the (semi-)norm in~\eqref{e:semi}\,.
\end{lemma}

\begin{proof}
Let us begin by controlling $\|\cG^\tau_M\tfm\|_\tau$. Since the support of $\hat\psi$ is contained in a ball of radius $1$ away from $\fe_2\R$, 
that of $\cF(\partial_{\fe_1}^{-1} \psi^\delta)$ is contained in a ball of radius $\delta$ which is independent of $\tau$. 
Hence, we can apply~\eqref{e:GA+SmoothNorm}, which gives
\begin{equs}
\|\cG^\tau_M\tfm\|_\tau=\|\cG^\tau_M\genapsh \partial_{\fe_1}^{-1} \psi^\delta\|_\tau\lesssim \nu_\tau^{3/4} \|\partial_{\fe_1}^{-1} \psi^\delta\|_\tau\lesssim\nu_\tau^{3/4} \|\partial_{\fe_1}^{-1} \psi\|_{L^2(\R^2)}\lesssim \nu_\tau^{3/4}
\end{equs}
and in the third step we used that $\Theta^\tau\lesssim 1$ and in the last last step that $\hat \psi$ 
is supported away from $\fe_2\R$ and that $\psi$ is smooth, 
so that $\|\partial_{\fe_1}^{-1} \psi\|_{L^2(\R^2)}<\infty$. Note that, for $\tau$ large enough, the right hand side is controlled 
by the right hand side of~\eqref{e:negl}. 

Concerning $\vertiii{\tfm}_{\tau,M}$, we will separately bound the two summands in its definition. 
For the second, the definition of $\genapsh$ in~\eqref{e:A+sharp} ensures that 
$\cF \tfm$ is supported on modes $(q,q')$ satisfying $|q+q'| \leq \delta$, 
for which the kernel $\cK_{q,q'}^{-, \tau}$ in~\eqref{e:K+K-} of $\genamsh$ can be bounded by $\delta$. 
Hence, at a price which might blow up as $\tau\to\infty$, we have 
\begin{equs}
\|\genamsh \cG_M^\tau \tfm\|_\tau^2 &\lesssim_{\lambda, M, \tau} \int_{\R^2} \dd p \, \delta^2 \, |\hat \psi^\delta(p)|^2 \int_{\R^2} \dd q \, \Theta_\tau(q) \Theta_\tau(p-q) \lesssim_{\lambda, M, \tau} \delta^2 \,, 
\end{equs}
which means that, since we are taking $\delta\to 0$ first, the $\limsup_{\delta\to0}$ of the previous vanishes.  
For the first, we first apply the Recursive Replacement Lemma~\ref{p:2ndReplLemma} 
with $f=p^\tau_0\equiv\nu_\tau$ and $p^\tau_0$ as in~\eqref{e:DefPk}
(meaning $C_f=1$), and get 
\begin{equs}
\|(-\nu_\tau\gensyx)^{1/2}\cG^\tau_M\tfm\|_\tau&=\|(-\nu_\tau\gensyx)^{1/2}\cG^\tau_M\genapsh \partial_{\fe_1}^{-1} \psi^\delta\|_\tau\\
&\leq  \big\|(-\gensyx[p_1^\tau +  \nu_\tau  C^{(1)}_{M}])^{1/2} \partial_{\fe_1}^{-1} \psi^\delta\big\|_\tau\\
&\lesssim \nu_\tau^{1/2}\sqrt{|\log\nu_\tau|}\big\|(-\gensyx)^{1/2} \partial_{\fe_1}^{-1} \psi^\delta\big\|_\tau\\
&\lesssim \nu_\tau^{1/2}\sqrt{|\log\nu_\tau|}\|\psi^\delta\|_\tau\leq\nu_\tau^{1/2}\sqrt{|\log\nu_\tau|}\label{e:Boundnegl}
\end{equs}
where we used~\eqref{e:RelationPk} together with the definition of $\gensyx[\cdot]$ in~\eqref{e:Lwf} and the fact that 
\begin{equs}
p_1^\tau(L^\tau(-\cS^\tau))+\nu_\tau  C^{(1)}_{M} \leq \nu_\tau\Big(\log\sqrt{\frac{\nu_\tau+g^\tau(L^\tau(-\cS^\tau))}{\nu_\tau}}+1\Big)\lesssim \nu_\tau|\log\nu_\tau|\,.
\end{equs} 
Simplifying the $\nu_\tau^{1/2}$ at the right and left hand sides of~\eqref{e:Boundnegl}, 
we deduce that the 
first summand in~\eqref{e:semi} is bounded by the right hand side of~\eqref{e:negl} uniformly in $\delta>0$, 
and thus, the analysis of $\vertiii{\tfm}_{\tau,M}$ is concluded.  
\medskip

At last, we turn to~\eqref{e:DiffAnsGhalf}. Let us first note that, by~\eqref{e:Duality1} and the definition of $\tfm$ in~\eqref{e:tfm}, 
we have  
\begin{equ}
\|(\cG_M^\tau)^{1/2} \tfm\|_\tau^2=
\langle -\genamsh \cG_M^\tau \genapsh \partial_{\fe_1}^{-1} \psi^\delta, \partial_{\fe_1}^{-1} \psi^\delta \rangle_\tau\,.
\end{equ}
and the Replacement Lemma~\ref{p:1stReplLemma} gives 
\begin{equs}
\Big|\langle -\genamsh \cG_M^\tau \genapsh \partial_{\fe_1}^{-1} &\psi^\delta, \partial_{\fe_1}^{-1} \psi^\delta \rangle_\tau - \langle \cg^\tau_M\partial_{\fe_1}^{-1} \psi^\delta, \partial_{\fe_1}^{-1} \psi^\delta \rangle_\tau\Big|\\
&\lesssim \|(-\nu_\tau\gensyx)^{1/2} \partial_{\fe_1}^{-1} \psi^\delta\|_\tau^2 \leq \nu_\tau \|\psi^\delta\|_{L^2(\R^2)}=\nu_\tau\,. 
\end{equs}
As a consequence, we are left to analyse 
\begin{equ}
\| (\cg^\tau_M)^{1/2}\partial_{\fe_1}^{-1} \psi^\delta\|^2_\tau=\| g^\tau_M(L^\tau(-\cS^\tau))^{1/2}(-\gensyx)^{1/2}\partial_{\fe_1}^{-1} \psi^\delta\|_\tau^2=\tfrac{1}{2}\| g^\tau_M(L^\tau(-\cS^\tau))^{1/2} \psi^\delta\|_\tau^2\,,
\end{equ}
where the prefactor $1/2$ comes from the definition of $\gensyx$ in~\eqref{e:gensy}. 
Now, by construction, $|\hat{\psi^\delta}|^2$ converges to a Dirac mass at $0$ as $\delta \to 0$, 
which means that, as $\delta\to 0$ the right hand side converges to 
\begin{equ}
g^\tau_M(L^\tau(0))=\Big(\frac{3\lambda^2}{2\pi}\nu_\tau^{3/2}\log(1+\tau)+\nu_\tau^{3/2}\Big)^{2/3}-\nu_\tau=C_\eff(\lambda)+O(\nu_\tau)
\end{equ}
from which~\eqref{e:DiffAnsGhalf} follows at once. 
\end{proof}

Thanks to the previous lemmas, we can now complete the proof of Theorem~\ref{th:Diffusivity}. 

\begin{proof}[of Theorem~\ref{th:Diffusivity}]
The bulk of the proof consists of establishing that, as $\tau\to\infty$,  $D^\tau_{1,1}(1)$ converges to $C_\eff(\lambda)$, 
which, by~\eqref{e:GKenergy}, follows provided 
\begin{equ}[e:LimEnergy]
\lim_{\tau\to\infty}\lim_{\delta\to0}\cE^\tau_1(\Phi^{\tau,\delta})=C_\eff(\lambda)\,, 
\end{equ}
where $\Phi^{\tau,\delta}$ is the solution of the parabolic equation~\eqref{e:ODEPhi}.  

Let us first show that, in~\eqref{e:LimEnergy},  we can replace $\Phi^{\tau,\delta}$  
with the solution $\tilde r^{\tau,\delta}$ of~\eqref{e:elliptictilde} at a cost which is negligible with respect to $\delta$ and $\tau$. 
Recall the definition of $r^{\tau,\delta}$  in~\eqref{e:elliptic}, and note that 
\begin{equs}[e:comparison]
\cE^\tau_1(\Phi^{\tau,\delta}-\tilde r^{\tau\delta})&\lesssim \cE^\tau_1(\Phi^{\tau,\delta}- r^{\tau,\delta})+\cE^\tau_1(r^{\tau,\delta}-\tilde r^{\tau,\delta})\\
&\lesssim \|r^{\tau,\delta}\|_\tau^2+\|r^{\tau,\delta}-\tilde r^{\tau,\delta}\|^2_{\fock{}{1}{\tau}}\lesssim \|\tilde r^{\tau,\delta}\|_\tau^2+\|r^{\tau,\delta}-\tilde r^{\tau,\delta}\|_{\fock{}{1}{\tau}}^2
\end{equs}
where in the last step we simply applied the triangle inequality to bound $\|r^{\tau,\delta}\|_\tau$ by the right hand side, 
and in the second, we used~\eqref{e:ParabolicElliptic} and fact that, for a time independent 
$\phi\in {\rm Dom}(\cL^\tau)$,~\eqref{e:AprioriEnergy} implies that $\cE^\tau_1(\phi)\lesssim \|\phi\|_{\fock{}{1}{\tau}}^2$. 
Now, the second term at the right hand side is of order $\nu_\tau\delta^2$ in view of~\eqref{e:rtilder}, 
while for the first we invoke~\eqref{e:L2} (with $r^{\tau,\delta}$ replaced by  $\tilde r^{\tau,\delta}$), which gives 
\begin{equ}[e:rtildeL2bound]
\limsup_{\delta\to0}\|\tilde r^{\tau,\delta}\|_\tau\lesssim\limsup_{\delta\to0}\Big( \vertiii{\tfm}_{\tau,M}+\|\cG^\tau_M\tfm\|_\tau\Big)\lesssim \nu_\tau^{1/4}|\log\nu_\tau|
\end{equ} 
thanks to~\eqref{e:negl}. 

As a consequence,~\eqref{e:comparison} vanishes in the double limit $\lim_{\tau\to\infty}\lim_{\delta\to0}$ and we are left to study
\begin{equ}
\cE^\tau_1(\tilde r^{\tau,\delta})=\|\tilde r^{\tau,\delta}\|_\tau^2 + 2\int_0^1 \|(- \gensy)^{1/2} \tilde r^{\tau,\delta}\|_\tau^2 \, \dd s=\|\tilde r^{\tau,\delta}\|_\tau^2+2\|\tilde r^{\tau,\delta}\|_{\fock{}{1}{\tau}}^2\,.
\end{equ}
The $\limsup_{\delta\to0}\|\tilde r^{\tau,\delta}\|_\tau^2$ can be controlled as in~\eqref{e:rtildeL2bound}, 
while for the $\fock{}{1}{\tau}$-norm, by~\eqref{e:H1}  (with $r^{\tau,\delta}$ replaced by  $\tilde r^{\tau,\delta}$), we get 
\begin{equs}
\limsup_{\delta\to 0}&\Big|\|\tilde r^{\tau,\delta}\|_{\fock{}{1}{\tau}}^2- \tfrac12C_\eff(\lambda)\Big|\\
&\leq \limsup_{\delta\to 0}\Big|\|\tilde r^{\tau,\delta}\|_{\fock{}{1}{\tau}}^2-\|(\cG^\tau_M)^{1/2}\tfm\|^2_\tau\Big|+\Big|\|(\cG^\tau_M)^{1/2}\tfm\|^2_\tau- \tfrac12C_\eff(\lambda)\Big|\\
&\lesssim\limsup_{\delta\to0} \vertiii{\tfm}_{\tau,M}\Big(\vertiii{\tfm}_{\tau,M}+\|(\cG^\tau_M)^{1/2}\tfm\|_\tau\Big)+\nu_\tau\lesssim \nu_\tau^{1/4}|\log\nu_\tau|
\end{equs}
where we used both~\eqref{e:negl} and~\eqref{e:DiffAnsGhalf}. By putting the previous estimates together,~\eqref{e:LimEnergy} 
follows at once. 

In order to complete the proof of~\eqref{e:Diffusivity}, it suffices to exploit the scaling relation in~\eqref{e:DiffRescaling} 
and the definition of $\nu_\tau$ in~\eqref{e:nutau}, which give 
\begin{equ}\label{e:DiffFinalEstimate}
D_{1,1}(\tau) = C_{\mathrm{eff}}(\lambda) \, \nu_\tau^{-1} + O(\nu_\tau^{-3/4}\,|\log\nu_\tau|) = C_{\mathrm{eff}}(\lambda) (\log\tau)^{2/3} \Big(1+ O\Big(\frac{\log\log\tau}{(\log\tau)^{1/6}}\Big)\Big) \, . 
\end{equ}
\end{proof}

\begin{appendix}

\section{The regularised SBE: proofs}\label{a:BasicsSBE}

In this appendix, we collect some of the proofs of the statements concerning the solution to the
regularised SBE in~\eqref{e:SBEsystem}.
Since $\tau>0$ is fixed, we take it to be equal to $1$ and omit the corresponding index
so to lighten the notation. In particular, $\nu_\tau = 1$, $R_\tau$ is the identity matrix
and~\eqref{e:ScaledSPDE} reduces to~\eqref{e:SPDE1Strong}.
The case of general $\tau$ can be recovered by scaling.

We begin by recalling the definition and main properties of weighted Besov spaces
(for a thorough exposition we refer to~\cite[Ch. 4]{ET}).

\subsection{Weighted Besov spaces}\label{a:Besov}

Since we want to construct a solution to~\eqref{e:SBEsystem} in infinite volume,
we will need to consider weighted norms.
A function $\we\colon \R^2\to \R_+$ is said to be a {\it weight} if there exists a positive constant
$C>0$ such that
\begin{equ}[e:PropWeight]
C^{-1}\leq \sup_{|x-y|\leq 1}\frac{\we(x)}{\we(y)}\leq C\,.
\end{equ}
The weights we will be concerned with are of polynomial or exponential type, namely
\begin{equ}[e:weight]
\p_a(x)\eqdef (1+|x|)^{-a}\,,\qquad \e^\sigma_b(x)\eqdef \exp\big(-b(1+|x|)^\sigma\big)
\end{equ}
for $a,b\in\R_+$, $\sigma\in(0,1)$ and $x\in\R^2$.

We now turn to Besov spaces, whose definition is based on a {\it dyadic partition of unity}, i.e.
a family $(\phi_j)_{j\geq -1}\subset \CC^\infty_c(\R^2)$ of non-negative radial functions
such that
\begin{itemize}[noitemsep]
\item $\supp \phi_{-1}$ and $\supp \phi_0$ are respectively contained in a ball and in an annulus around the origin,
\item for $j\geq 0$, $\phi_j(\cdot)=\phi_0(2^{-j}\cdot)$
\item for all $x\in\R^2$, $\sum_{j\geq -1}\phi_j(x)=1$,
\item if $|j-j'|>1$, then $\supp \phi_j\cap\supp\phi_{j'}=\emptyset$.
\end{itemize}
We denote by $\PHI_j$ the inverse Fourier transform of $\phi_j$ and set $\Delta_j f\eqdef \PHI_j\ast f$
to be the $j$-th Littlewood-Paley block.

\begin{definition}\label{def:Besov}
Let $\we$ be a weight as in~\eqref{e:PropWeight},
$\alpha\in\R$ and $p,q\in[1,\infty]$. We define the weighted Besov space $\cB^\alpha_{p,q}(\we)$
as the space of distributions (if the weight is polynomial, and of ultra-distributions\footnote{For the definition of ultra-distributions and their need in the case of exponential weights, see e.g.~\cite[Sec. 2.2]{MP}} if exponential)
$f\in\CS'(\R^2)$ (or $\CS_\omega'(\R^2)$) such that
\begin{equ}[e:Besov]
\|f\|_{\cB^\alpha_{p,q}(\we)}\eqdef \Big\|\Big(2^{\alpha j}\|\we \Delta_j f\|_{L^p(\R^2)}\Big)_{j\geq -1}\Big\|_{\ell^q}=\Big(\sum_{j\geq -1}2^{\alpha j q}\|\we \Delta_j f\|_{L^p(\R^2)}^q\Big)^{\frac1q}<\infty\,.
\end{equ}
For $p=q=\infty$, we adopt the notation $\CC^\alpha(\we)\eqdef \cB^\alpha_{\infty,\infty}(\we)$.
\end{definition}

As given above, weighted Besov spaces are Banach and, for suitable choices of $\alpha,p,q$ and $\we$
they coincide with spaces the reader might be more familiar with.
In particular, if
$p=q=\infty$, $\alpha\in\R\setminus\N$ and $\we\equiv 1$,
$\CC^\alpha(1)$ is the space of $\alpha$-H\"older continuous functions.
If $p=q=2$, $\cB^{\alpha}_{2,2}(\we)=H^\alpha(\we)$ is the usual (weighted) Sobolev space and,
if $\alpha=0$ and $\we\equiv 1$, $\cB^{0}_{2,2}(1)=L^2(\R^2)$.

Besov spaces satisfy a number of properties that we will need in what follows. Let us summarise them
without proof.
\begin{itemize}[noitemsep]
\item {\it (Besov embedding)} (see e.g.~\cite[Thm. 4.2.3]{ET}) for any weights $\we_2 \lesssim \we_1$ such that $\lim_{|x|\to\infty}\we_2(x)/\we_1(x)=0$,
$1\leq p_1\leq p_2\leq\infty$, $1\leq q_1\leq q_2\leq\infty$, $\alpha_1\in\R$ and
$\alpha_2\leq \alpha_1-2(1/p_1+1/p_2)$, we have the continuous embedding
\begin{equ}[e:BesovEmb]
\cB^{\alpha_1}_{p_1,q_1}(\we_1)\subseteq \cB^{\alpha_2}_{p_2,q_2}(\we_2)\,.
\end{equ}
\item {\it (Semigroup estimates)} (see e.g.~\cite[Prop. 3.6]{MP}) Let $K_t\eqdef e^{\frac{t}{2}\Delta}$
be the semigroup associated to $\tfrac12\Delta$.
Then, for any
$\alpha\in\R$, $\beta\geq 0$ and weight $\we$, we have
\begin{equ}[e:Semigroup]
\|K_t f\|_{\CC^{\alpha+\beta}(\we)}\lesssim t^{-\frac{\beta}{2}}\|f\|_{\CC^{\alpha}(\we)}\,,
\end{equ}
where $K_t f \eqdef K_t \ast f$.

\item {\it (Product rule)} (see e.g.~\cite[Lem. 4.2]{MP}) For any $\alpha_1,\alpha_2\in\R$ and weights $\we_1,\we_2$, if $\alpha_1+\alpha_2>0$, then
\begin{equ}[e:Prod]
\| fg\|_{\CC^{\alpha_1\wedge \alpha_2}(\we_1\we_2)}\lesssim \| f\|_{\CC^{\alpha_1}(\we_1)}\| g\|_{\CC^{\alpha_2}(\we_2)}\,.
\end{equ}
\end{itemize}

\subsection{Approximation and convergence for the linear equation}\label{a:SHE}

Before turning to Burgers, we consider the stationary
solution to the linear stochastic heat equation (SHE) obtained
by dropping the non-linearity in~\eqref{e:SBEsystem}.

Let $(\eta,\vec{\xi})$ be a couple of independent
space and space-time Gaussian white noises on $\R^2$ and $\R\times\R^2$ respectively.
In order to ensure measurability of the solution (both for SHE and SBE),
we want to construct our approximating sequence
in such a way that all of its elements are coupled together using the same noises.
For this, starting from $(\eta,\vec{\xi})$ and
given $M\in\N$, we set
\begin{equ}[e:Periodised]
\eta_M\eqdef \sum_{y \in M\Z^2} \tau_y (\eta\1_{Q(M)})\,\qquad\text{and}\qquad \vec{\xi}_{M}(\dd s) \eqdef \sum_{y \in M \Z^2} \tau_y \,(\vec{\xi}\1_{Q(M)})
\end{equ}
where $Q(M)\eqdef[-M/2,M/2]^2$ and, for $y\in\R^2$, we denoted by $\tau_y$ the translation operator
on $\cS'(\R^2)$ (and defined by duality from $\tau_y f(\cdot)=f(\cdot-y)$ for $f\in\cS(\R^2)$).
Note that, thanks to spatial independence, $(\eta_M, \vec{\xi}_{M})$ coincide in law to
{\it periodic} independent space and space-time white noises with period $M\in\N$.

According to Definition~\ref{def:SolRegBurgers},
the unique mild solutions $X$ and $X^M$ of~\eqref{e:SBEsystem} with $\lambda=0$
driven by $(\eta^1,\vec{\xi}^{\,1})$ and by $(\eta^1_M,\vec{\xi}^1_M)$ are respectively given by
\begin{equ}[e:SHEFull&Per]
X_t = K_t \eta^1 + \int_0^t K_{t-s} \,  \mathrm{div} \, \vec{\xi}^{\,1} \, (\dd s)\quad\text{and}\quad X^M_t = K_t \eta^1_M + \int_0^t K_{t-s} \,  \mathrm{div} \, \vec{\xi}^{\,1}_{M} (\dd s)\,,
\end{equ}
for $K$ the usual heat kernel. Then, we have the following lemma.

\begin{lemma} \label{lem:CVSHE}
In the setting outlined above, let $X$ and $X^M$ be as in~\eqref{e:SHEFull&Per}.
Then, for any $\alpha, a, p, T > 0$ and $\beta \in [0,1/2)$, $X$ and $X^M$ almost surely
belong to $C_{\mathrm{loc}}^\beta(\R_+, \cC^\infty(\R^2))$ and there exists a constant $C=C(\rho)>0$
such that
\begin{equ} \label{e:BoundSHE}
\Exp\Big[\|X\|_{C^\beta([0,T], \cC^\alpha(p_a))}^p\Big] \vee \sup_{M}\Exp\Big[\|X^M\|_{C^\beta([0,T], \cC^\alpha(p_a))}^p\Big] \leq C \, .
\end{equ}
Furthermore,
\begin{equ} \label{e:CVSHE}
\lim_{M\to\infty}\Exp\Big[\|X - X^M\|_{C^\beta([0,T], \cC^\alpha(p_a))}^p\Big] =0\,.
\end{equ}
\end{lemma}

In the proof of the lemma (and in Lemma~\ref{l:Tightness}) we will need the following technical result.

\begin{lemma}\label{l:Paley}
For any $\gamma,\kappa>0$, $\bk\in\N^2$ and $m\in\{1,2\}$,
there exists a constant $C>0$ such that for every $j\geq -1$, every $r\geq 0$ and all $x\in\R^2$,
we have
\begin{equs}[e:Paley]
 \|\p_{-2\kappa} (\Delta_j K_{r} \partial^{|\bk|}_\bk\rho^{\ast m})(x-\cdot)\|_{L^2(\R^2)} \leq C 2^{-\gamma j}
\end{equs}
where $\Delta_j$ denotes the $j$-th Littlewood-Paley block, $\p_{-\kappa}$ is as in~\eqref{e:weight}
and $\partial_\bk^{|\bk|}=\partial_1^{k_1}\partial_2^{k_2}$. Moreover, for any $0\leq s\leq t$, 
we have
\begin{equ}[e:Paley2]
 \|\p_{-\kappa} (\Delta_j (K_{t}-K_s) \partial^{|\bk|}_\bk\rho)(x-\cdot)\|_{L^2(\R^2)}\leq C (t-s)2^{-\gamma j}\,.
\end{equ}
\end{lemma}
\begin{proof}
Notice that, upon taking $\kappa'>\kappa$ so that $\kappa'\in\N$,
we can interpret the weight $\p_{-\kappa'}$ as the inverse
Fourier transform of a derivative, thus the left hand side of~\eqref{e:Paley} is bounded above by
\begin{equ}
\|\p_{-\kappa'} (\Delta_j K_{r} \partial^{|\bk|}_\bk\rho^{\ast m})(x-\cdot)\|_{L^2(\R^2)}\lesssim \|\cF^{-1}((\Delta_j K_{r} \partial^{|\bk|}_\bk\rho^{\ast m})(x-\cdot))\|_{H^{\kappa'}(\R^2)}
\end{equ}
where, for any $p=(p_1,p_2)\in\R^2$, 
\begin{equ}
\cF^{-1}((\Delta_j K_{r} \partial^{|\bk|}_\bk\rho^{\ast m})(x-\cdot))(p)=\phi_j(p) \, e^{-r |p|^2} \, p_1^{k_1}p_2^{k_2} \, \hat{\rho}^m(p)\,e^{-2\pi i p\cdot x}\,.
\end{equ}
A similar argument holds for the left hand side of~\eqref{e:Paley2}, with the inverse Fourier transform
of the argument given by
\begin{equ}
\cF^{-1}(\Delta_j (K_{t}-K_s) \partial^{|\bk|}_\bk\rho(x-\cdot))(p)=\phi_j(p) \, (e^{-t |p|^2}-e^{-s|p|^2})\, p_1^{k_1}p_2^{k_2} \,\hat{\rho}(p)\,e^{-2\pi i p\cdot x}\,.
\end{equ}
Hence,~\eqref{e:Paley} and~\eqref{e:Paley2} follow by the fact that every factor in the above formulas is
(uniformly) smooth, that $\hat \rho$ and all its derivatives
decay faster than any polynomial, and that $\phi_j$ is supported on modes $p$ of order $2^j$.
\end{proof}

\begin{proof}[of Lemma~\ref{lem:CVSHE}]
Since the proof of~\eqref{e:BoundSHE} and~\eqref{e:CVSHE} follows the same steps with the second being
slightly more involved, we only focus on the latter.

We proceed as usual, namely, setting $\tilde X^M\eqdef X-X^M$, we control
the time increments of $\tilde X^M$ (and $\tilde X^M_0$, but the proof is similar and thus omitted),
apply first Besov's embedding~\eqref{e:BesovEmb}
and then Kolmogorov's criterion. 

As a first step, we need to consider the Littlewood-Paley blocks of these increments
and, using the coupling of $(\eta,\vec{\xi})$ and $(\eta^M,\vec{\xi}^M)$ provided by the
construction in~\eqref{e:Periodised}, we write them purely in terms of $(\eta,\vec{\xi})$.
To do so, let
$K^{j}_t\eqdef \Delta_jK_{t}\rho$, $j\geq-1$, and
notice that
\begin{equ}[e:CVSHEIdentity]
\Delta_j\tilde X^M_t(x)=\eta(\tilde K^{j,M}_{t,x})+\int_0^t\vec{\xi}(\dd s, \nabla\tilde K^{j,M}_{t-s,x})
\end{equ}
where
\begin{equ}
\tilde K^{j,M}_{t,x}(z)\eqdef \sum_{y\in M\Z^2_0} K^{j}_t(x-z)(\tau_y\1_{Q(M)})(z)-\sum_{y\in M\Z^2_0}\tau_{-y}\Big(K^{j}_t(x-z)(\tau_y\1_{Q(M)})(z)\Big)
\end{equ}
for $\Z^2_0\eqdef\Z^2\setminus\{0\}$. The
absence of the summand corresponding to $y=0$ is due to the difference.
The right hand sides of~\eqref{e:CVSHEIdentity} are well-defined as, even though
$\tilde K^{j,M}_{t,x}$ is not smooth (because of
the behaviour on $M\Z^2$), it belongs to $L^2(\R^2)$ for all $x\in\R^2$.
Indeed, for any $\bk\in\N^2$, we have
\begin{equs}[e:BoundL2Paley]
\|\partial^{|\bk|}_\bk\tilde K^{j,M}_{t,x}\|_{L^2(\R^2)}&\leq 2\sum_{y\in M\Z^2_0}\|\partial^{|\bk|}_\bk K^{j}_{t}(x-\cdot)(\tau_y\1_{Q(M)})\|_{L^2(\R^2)}\\
&\lesssim \sum_{y\in M\Z^2_0}\p_{\kappa}(y)\|\p_{-\kappa} \partial^{|\bk|}_\bk K^{j}_{t}(x-\cdot)\|_{L^2(\R^2)}\lesssim M^{-\kappa}2^{-j\gamma}
\end{equs}
where we have chosen $\kappa>0$ large enough so that $\p_{\kappa}$ is
summable and used~\eqref{e:Paley} in the last bound.

In view of~\eqref{e:BoundL2Paley}, we can now estimate to the time increments of $\Delta_j\tilde X^M$.
Let $0<a'<a$, $p>\max\{2/a', (1/2-\beta)^{-1}\}$ and $\alpha'\geq \alpha+2/p$. 
Then, for any $0\leq s\leq t$, we have 
\begin{equs}
\,&\Exp\Big[\|\p_{a'}\Delta_j(\tilde X_t^M-\tilde X_s^M)\|^p_{L^p(\R^2)}\Big]\\
&= \Exp\Big[\Big\|\p_{a'}\Big(\eta(\tilde K^{j,M}_{t,\cdot}-\tilde K^{j,M}_{s,\cdot})+\int_s^t\vec{\xi}(\dd r, \nabla\tilde K^{j,M}_{t-r,\cdot})+\int_0^s\vec{\xi}(\dd s, \nabla(\tilde K^{j,M}_{t-r,\cdot}-\tilde K^{j,M}_{s-r,\cdot}))\Big)\Big\|^p_{L^p(\R^2)}\\
&=\int_{\R^2}\p_{a'}(x)^p\Exp\Big[\Big(\eta(\tilde K^{j,M}_{t,x}-\tilde K^{j,M}_{s,x})+\int_s^t\vec{\xi}(\dd r, \nabla\tilde K^{j,M}_{t-r,x})\\
&\qquad\qquad\qquad\qquad\qquad+\int_0^s\vec{\xi}(\dd s, \nabla(\tilde K^{j,M}_{t-r,x}-\tilde K^{j,M}_{s-r,x}))\Big)^p\Big]\dd x\\
&\lesssim \int_{\R^2}\p_{a'}(x)^p\Big(\E[\eta(\tilde K^{j,M}_{t,x}-\tilde K^{j,M}_{0,x})^2]+\Exp\Big[\Big(\int_s^t\vec{\xi}(\dd s, \nabla\tilde K^{j,M}_{t-s,x})\Big)^2\Big]\\
&\qquad\qquad\qquad\qquad\qquad+\Exp\Big[\Big(\int_0^s\vec{\xi}(\dd s, \nabla(\tilde K^{j,M}_{t-r,x}-\tilde K^{j,M}_{s-r,x}))\Big)^2\Big]\Big)^{\frac{p}{2}}\dd x\\
&= \int_{\R^2}\p_{a'}(x)^p\Big(\|\tilde K^{j,M}_{t,x}-\tilde K^{j,M}_{s,x}\|_{L^2(\R^2)}^2+\int_s^t\|\nabla\tilde K^{j,M}_{t-s,x}\|_{L^2(\R^2)}^2\dd s\Big)^{\frac{p}{2}}\dd x\\
&\qquad\qquad\qquad\qquad\qquad+\int_0^s\|\nabla(\tilde K^{j,M}_{t-r,x}-\tilde K^{j,M}_{s-r,x}))\|_{L^2(\R^2)}^2\dd s\Big)^{\frac{p}{2}}\dd x\lesssim (t-s)^{\frac{p}{2}}M^{-\kappa p}2^{-j\gamma p}
\end{equs}
where in the third step we used Gaussian hypercontractivity and the independence of $\eta$ and $\vec{\xi}$,
in the fourth It\^o's isometry and in the last~\eqref{e:BoundL2Paley} for the second summand, and
the analogous bound for the first and third (which can be obtained arguing as in~\eqref{e:BoundL2Paley}
using~\eqref{e:Paley2} in place of~\eqref{e:Paley}).

Now, using Besov's embeddings~\eqref{e:BesovEmb} together with our assumption on $a',\alpha',p$
and $\gamma>\alpha'$, we get
\begin{equs}[e:LPboundSHEincrement]
\Exp\Big[&\|(\tilde X_t^M-\tilde X_s^M\big)\|^p_{\CC^{\alpha}(\p_a)}\Big]\lesssim \Exp\Big[\|(\tilde X_t^M-\tilde X_s^M)\|^p_{\cB^{\alpha'}_{p,p}(\p_{a'})}\Big] \\
&= \sum_{j\geq -1} 2^{j \alpha' p} \, \Exp\Big[\|\p_{a'} \Delta_j \, \big[ (\tilde X_t^M-\tilde X_s^M)\big]\|_{L^p(\R^2)}^p\Big]\lesssim_\kappa (t-s)^{p/2} M^{-\kappa} \, ,
\end{equs}
from which, an application of Kolmogorov's criterion gives~\eqref{e:CVSHE}.
\end{proof}

\subsection{The regularised SBE: Proposition~\ref{p:Global}}\label{a:SBE}

This section is the core of the present appendix as it contains the proof of Proposition~\ref{p:Global}.
The latter is based on two main ingredients: the derivation of suitable
weighted estimates for a stationary periodic approximation to the solution of~\eqref{e:SBEsystem}
(Lemma~\ref{l:Tightness}) and a deterministic apriori bound which allows to control
the difference of two solutions driven by possibly different noises (Lemma~\ref{l:AprioriLinear}).
We begin with the former.

\begin{lemma}\label{l:Tightness}
For $M\in\N$, let $\eta_M$ and $\vec{\xi}_M$ be as in~\eqref{e:Periodised}.
Then, 
the equation~\eqref{e:SBEsystem} driven by $(\eta^1_M, \vec{\xi}^1_M)$ admits a unique
mild solution $u^M$ according to Definition~\ref{def:SolRegBurgers} which 
is strong Markov, stationary, i.e. for every $t \geq 0$, $u^M_t \eqlaw \eta^1_M$, and skew-reversible.
Furthermore, for every $\alpha,a, p, T > 0$ and $\beta \in [0,1/2)$,
there exists a positive constant $C=C(\rho,\alpha,\beta,a, p, T)<\infty$ such that
\begin{equ}[e:aprioriM]
\sup_{M}\Exp\Big[\|u^M\|^p_{C^\beta([0,T], \cC^\alpha(\p_{a}))}\Big] \leq C \, .
\end{equ}
\end{lemma}
\begin{proof}
Existence, uniqueness (of weak solutions in
$C_\loc(\R_+,\cS'(\T^2_M))$, $\T^2_M$ being the two-dimensional torus of side $M$),
strong Markovianity, stationarity and skew-reversibility for fixed $M$ (essentially)
follow from~\cite[Lemma 2.1]{CGT}.
In the afore-mentioned reference, it is shown that such properties hold for $\tilde u^M$ which is the solution to
\begin{equ}[e:SBEcgt]
\partial_t \tilde u^M = \frac{1}{2} \Delta \tilde u^M + \lambda \, \tilde\cN^{1}[\tilde u^M] + (- \Delta)^{1/2} \, \xi_M\,,\qquad \tilde u^M_0\eqdef \eta_M
\end{equ}
on $\T^2_M$, with $\eta_M$ and $\xi_M$ space and space-time white noises on $\T^2_M$, and
\begin{equ}[e:Oldnonlin]
\tilde\cN^{1}[f](x)\eqdef [\rho\ast(\fe_1\cdot\nabla)(\rho\ast f)^2](x)=[\rho\ast\partial_1(\rho\ast f)^2](x)\,.
\end{equ}
Compared to our setting, the $\rho$ chosen therein is a cut-off function in Fourier space (so that
in particular it is not compactly supported in real space) and
the noise is formally different. That said, $(- \Delta)^{1/2} \, \xi_M\eqlaw \div\vec{\xi}_M$ and,
standard references (e.g.~\cite{DPZ}) guarantee existence and uniqueness of local in time weak solutions
and that these solutions are adapted and
strong Markov. On the other hand, the proof of global in time existence, stationarity
and skew-reversibility given in~\cite{CGT}
(see also~\cite[Lemma 3.3]{CT}) applies verbatim.

Therefore, setting $u^M\eqdef\rho\ast \tilde u^M$, $u^M$ is clearly smooth in space
(and thus belongs to the unweighted space of spatially periodic functions with period $M$,
$C_\loc(\R_+,\CC_{{\rm per}_M}^\alpha)$, for any $\alpha>0$),
and is the (necessarily unique) weak solution to~\eqref{e:SPDE1Strong}
(which is also mild, see Remark~\ref{rem:MildvsWeak})
driven by $(\eta^1_M, \vec{\xi}^1_M)$.
Further, it is a strong Markov process,
has the law of $\eta^1_M$ as invariant measure and is skew-reversible.
\medskip

Thus, identifying $u^M$ with its periodic extension to the whole of $\R^2$,
we are left to show that, for every $T>0$, it belongs to the weighted space $C^\beta([0,T],\CC^\alpha(\p_a))$
uniformly over $M$,
i.e. that~\eqref{e:aprioriM} holds. First,
let us write the mild formulation as
\begin{equs}
u^M_t &= K_t \eta^1_M + \lambda \int_0^t K_{t-s} \, \cN^1[u^M_s] \, \dd s+\int_0^t K_{t-s} \, \mathrm{div} \, \vec{\xi}^1_M(\dd s)=X_t^M+N_t^M
\end{equs}
in which $X^M$ collects the first and the last summand, i.e. the terms containing the noises,
and $N^M$ the second, i.e. that containing the nonlinearity.

Let $a',p$ and $\alpha'$ be such that $0<a'<a$, $p>\max\{2/a', (1/2-\beta)^{-1}$ and
$\alpha'\geq \alpha +2/p$. Let us begin by controlling the norm of the Littlewood-Paley
block of $N^M$ at a fixed time $t\geq 0$. Then,
\begin{equs}
\Exp\Big[\big\|\p_{a'} \, \Delta_j \, N^M_t \big\|_{L^p(\R^2)}^p\big]&=\lambda^p\int_{\R^2}\p_{a'}(x)^p \, \Exp\Big[\Big(\int_0^t \Delta_j K_{t-r} \, \cN^1[u^M_r](x) \, \dd r \Big)^p \Big] \, \dd x\\
&\lesssim t^p \, \E\Big[\sup_{r\in[0,t]} \Big|\Delta_j K_{r} \, \cN^1[\eta^1_M](0)\Big|^p\Big]
\end{equs}
where we first applied Jensen's inequality to the time integral, then we used stationarity of $u^M$ to replace it
with $\eta^1_M$, translation invariance to remove the spatial dependence from the quantity in the expectation
and the fact that $a'p>2$ to bound the spatial integral. Now, for $\kappa>0$ we have
\begin{equs}
\Big|\Delta_j K_{r} \cN^1[\eta^1_M](0)\Big|&=\Big|\langle \Delta_j K_{r} \partial_1\rho^{\ast2}, (\eta^1_M)^2\rangle\Big|\leq \|\p_{-\kappa} \Delta_j K_{r} \partial_1\rho^{\ast2}\|_{L^2(\R^2)}\|(\eta^1_M)^2\p_{\kappa}\|_{L^2(\R^2)}\\
&\lesssim \|\p_{-\kappa} \Delta_j K_{r} \partial_1\rho^{\ast2}\|_{L^2(\R^2)}\|\eta^1_M\|^2_{\CC^\alpha(\p_\kappa)}\lesssim 2^{-j\gamma}\|\eta^1_M\|^2_{\CC^\alpha(\p_\kappa)}
\end{equs}
where the third step follows by~\eqref{e:BesovEmb} and~\eqref{e:Prod} (up to take $\alpha$ sufficiently large)
while the last step holds for any $\gamma>0$ and $r\in[0,t]$,
and is a consequence of~\eqref{e:Paley}.
Therefore, applying Besov's embedding once again, combining the last two estimates and
choosing $\gamma>\alpha'$, we obtain
\begin{equs}
\Exp\Big[\|N^M_t \|_{\cC^\alpha(\p_a)}^p\Big] &\lesssim \Exp\Big[\|N^M_t \|_{\cB_{p,p}^{\alpha'}(\p_{a'})}^p\Big] =\sum_{j \geq -1} 2^{j \alpha' p} \, \Exp\Big[\Big\|\p_{a'} \, \Delta_j \, N^M_t \Big\|_{L^p(\R^2)}^p\Big] \\
&\lesssim t^p \, \E\Big[\|\eta^1_M\|^{2p}_{\CC^\alpha(\p_\kappa)}\Big] \sum_{j\geq -1} 2^{-j(\gamma-\alpha')p}\lesssim t^p
\end{equs}
in which we also used~\eqref{e:BoundSHE} with $t=0$.

Since $u^M$ is Markov, stationary and has stationary increments,
a bound on any time increment from $s$ to $t$, with $0\leq s\leq t$, follows
provided we have one for $s=0$. By the previous bound and~\eqref{e:BoundSHE}, we get
\begin{equs}
\Exp\Big[\|u^M_t-u^M_0 \|_{\cC^\alpha(\p_a)}^p\Big]^{\frac{1}{p}}\leq \Exp\Big[\|X^M_t-X^M_0 \|_{\cC^\alpha(\p_a)}^p\Big]^{\frac{1}{p}} +\Exp\Big[\|N^M_t \|_{\cC^\alpha(\p_a)}^p\Big]^{\frac{1}{p}}\lesssim_\beta t^{\beta}\,,
\end{equs}
for any $\beta\in[0,1/2)$.
By Kolmogorov's criterion, this implies~\eqref{e:aprioriM}
so that the proof of the lemma is complete.
\end{proof}

At this point we can move to the second step, namely show that~\eqref{e:SBEsystem} has
a unique solution and that the sequence $\{u^M\}_M$ converges to that solution.
As we will soon see, both the results follow from the next lemma.

\begin{lemma} \label{l:AprioriLinear}
Let $a, b, \alpha > 0$, $\sigma\in(0,1)$ such that $a < \sigma$.
Let $v \in C([0,T], \mathcal C^\alpha(e_b^\sigma))$ be such that there exist
$X, w \in C([0,T], \mathcal C^\alpha(\p_a))$ 
for which
\begin{equ}[e:EquationDifference]
v_t = X_t + \lambda \int_0^t K_{t-s} \, \cN^1[v_s,w_s] \, \dd s \, ,\qquad  \forall\, t \in [0,T]\,,
\end{equ}
where $\cN^1[f,g]$ is defined according to
\begin{equ}[e:BilinearForm]
\cN^1[f,g] \eqdef \big(\partial_1 \varrho^{\ast2}\big) \ast (fg) \, .
\end{equ}
Then, there exists a positive constant $C = C(\rho, \lambda, a,b, \sigma,T)<\infty$  such that $v$ satisfies
\begin{equ}[e:AprioriEstimateDifference]
\sup_{t \in [0,T]} \|v_t\|_{\cC^\alpha(\e_{b+t}^\sigma)} \leq C \|X\|_{C([0,T], \cC^\alpha(\p_a))}  \exp\Big(C \, \|w\|_{C([0,T],\cC^\alpha(\p_a))}^{\frac{1}{1 - a/\sigma}}\Big)
\end{equ}
\end{lemma}
\begin{proof}
The proof follows the strategy implemented in~\cite{HL} in the context of the parabolic Anderson equation
(which, contrary to our setting, is rough) and consists of taking advantage of the different weights imposed on
$v$ and $w$. The crucial ingredient is the (trivial) estimate
\begin{equ}[e:WeightsB]
\frac{\e_{b+t}(x)}{\p_a(x) \, \e_{b+s}(x)}=\frac{[(t-s)(1+|x|)^\sigma]^{\frac{a}{\sigma}}}{(t-s)^{\frac{a}{\sigma}}} e^{-(t-s)(1+|x|)^\sigma}\leq e^{-\frac{a}{\sigma}}\Big(\frac{a/\sigma}{t-s}\Big)^{\frac{a}{\sigma}} \, ,
\end{equ}
which holds uniformly over $x\in\R^2$, and where in the last step we used that $y\mapsto y^{a/\sigma} e^{-y}$ is bounded by $e^{-a/\sigma} (a/\sigma)^{a/\sigma}$.

From \eqref{e:EquationDifference}, we deduce
\begin{equs}[e:AprioriLinear]
\|v_t\|_{\cC^\alpha(\e_{b+t}^\sigma)} &\leq \|X_t\|_{\cC^\alpha(\e_{b+t}^\sigma)} + \lambda \int_0^t \|K_{t-s} \, \cN^1[v_s,w_s]\|_{\cC^\alpha(e_{b+t}^\sigma)} \, \dd s \\
&\lesssim \|X_t\|_{\cC^\alpha(\p_a)} + \lambda \int_0^t \|v_s w_s\|_{\cC^\alpha(e_{b+t}^\sigma)} \, \dd s \\
&\lesssim \|X_t\|_{\cC^\alpha(\p_a)} + \lambda \int_0^t (t-s)^{-a/\sigma} \|v_s\|_{\cC^\alpha(e_{b+s}^\sigma)} \|w_s\|_{\cC^\alpha(p_a)} \, \dd s \\
&\lesssim \|X_t\|_{\cC^\alpha(\p_a)} + \lambda \, \|w\|_{\cC([0,T],\cC^\alpha(p_a))} \int_0^t (t-s)^{-a/\sigma} \|v_s\|_{\cC^\alpha(e_{b+s}^\sigma)} \, \dd s \, ,
\end{equs}
where in the second step we used the convolution bound \eqref{e:Semigroup} while
in the third the product rule \eqref{e:Prod} and~\eqref{e:WeightsB}.
Now, recalling that $a/\sigma < 1$, Grönwall's lemma~\cite[Lemma 7.1.1]{D} implies the existence of
$C = C(\rho, \lambda, a, b, \sigma) > 0$ for which~\eqref{e:AprioriEstimateDifference} holds, thus the proof is concluded.
\end{proof}

We now have all the elements to complete the proof of Proposition~\ref{p:Global}.

\begin{proof}[of Proposition~\ref{p:Global}]

The proof of the statement is divided into three steps which amount to show that
1) the sequence $\{u^M\}_M$ converges to a unique limit $u$ which is a mild solution to~\eqref{e:SBEsystem}
and satisfies~\eqref{e:apriori},
2)~\eqref{e:SBEsystem} admits a unique solution in $C_\loc(\R_+,\CC^\alpha(\p_a))$, for $a,\alpha>0$
(so that it is necessarily given by $u$) and
3) that $u$ is a Markov process and is stationary.
Let us fix a number of parameters which will be used throughout, namely $T > 0$,
$\alpha, a, b, \sigma>0$, $\beta \in [0,1/2)$ such that $a < \sigma < 1$, and
$\alpha' > \alpha$, $\beta' \in (\beta, 1/2)$ and $a' \in (0,a)$.
\medskip

\noindent {\it 1) Convergence of $\{u^M\}_M$.} The result follows upon proving that the sequence is Cauchy
in the right (Banach) space.
Adopting the same notations as in the proof of Lemma~\ref{l:Tightness}, for $M,M'\in\N$ we have
\begin{equ}
u^M_t-u^{M'}_t= X^M_t-X^{M'}_t+\lambda\int_0^t K_{t-s} \, \cN^1[u^M_s - u^{M'}_s, u^M_s + u^{M'}_s] \, \dd s \, \dd s
\end{equ}
where $\cN^1[\cdot,\cdot]$ is defined according to~\eqref{e:BilinearForm}.
Hence, applying Lemma \ref{l:AprioriLinear} with $v=v^{M,M'} \eqdef u^M - u^{M'}$, $X =X^{M,M'} = X^M - X^{M'}$ and
$w=w^{M,M'} \eqdef u^M + u^{M'}$ leads to
\begin{equ}
\sup_{t\in[0,T]}\|v^{M,M'}\|_{\CC^\alpha(\e_{b+t}^\sigma)} \leq C\|X^{M,M'}\|_{C([0,T], \cC^\alpha(\p_a))}  \exp\Big(C \, \|w^{M, M'}\|_{C([0,T],\cC^\alpha(\p_a))}^{\frac{1}{1 - a/\sigma}}\Big)
\end{equ}
for some constant $C = C(\rho, \lambda, a, b, \sigma, T)$, which, together with~\eqref{e:aprioriM}
and~\eqref{e:CVSHE}, easily implies
\begin{equ} \label{e:CauchyProba1}
\lim_{M,M' \to +\infty} \mathbf E\Big[1 \wedge \sup_{t\in[0,T]}\|u^M - u^{M'}\|_{\cC^\alpha(e_{b+t}^\sigma)}\Big] = 0 \, .
\end{equ}
At this point, it remains to show that the convergence takes place
in a space with polynomial (instead of exponential) weight. To do so, note that
our choices of $\alpha,\alpha',\beta,\beta', a$ and $b$ ensure that the following embeddings hold
\begin{equ}
C^{\beta'}([0,T],\cC^{\alpha'}(\p_{a'})) \overset{c}{\hookrightarrow} C^\beta([0,T],\cC^\alpha(\p_a)) \hookrightarrow C([0,T],\cC^\alpha(e_{b+T}^\sigma)) \, ,
\end{equ}
the first even being compact and where the last space contains those functions such that the norm at the
left hand side of~\eqref{e:AprioriEstimateDifference} is finite.
Note that, for generic Banach spaces $X_1\subset X_2\subset X_3$ such that the first embedding is compact,
it is easy to see that for every $\delta>0$, there exists a constant $A_\delta>0$ such that
\begin{equ}
\|x\|_{X_2} \leq \delta \, \|x\|_{X_1} + A_\delta \, \|x\|_{X_3}\,.
\end{equ}
In particular, the above applies to $X_1=C^{\beta'}([0,T],\cC^{\alpha'}(\p_{a'}))$, $X_2= C^\beta([0,T],\cC^\alpha(\p_a))$ and
$X_3=C([0,T],\cC^\alpha(e_{b+T}^\sigma))$.
Therefore, invoking Lemma \ref{l:Tightness}, we deduce
\begin{equs}
\mathbf E\Big[1 &\wedge \|u^M - u^{M'}\|_{C^\beta([0,T],\cC^\alpha(\p_a))}\Big] \\
&\leq \delta \, \mathbf E\big[\|u^M - u^{M'}\|_{C^{\beta'}([0,T],\cC^{\alpha'}(\p_a))}\big] + A_\delta \, \mathbf E \big[1 \wedge \|u^M - u^{M'}\|_{C([0,T],\cC^\alpha(e_{b+T}))}\big] \\
&\lesssim \delta + A_\delta \, \mathbf E \big[1 \wedge \|u^M - u^{M'}\|_{C([0,T],\cC^\alpha(e_{b+T}))}\big] \, ,
\end{equs}
which, using \eqref{e:CauchyProba1} and that $\delta$ can be taken arbitrarily small, leads to
\begin{equ}[e:CauchyProba2]
\lim_{M, M' \to +\infty} \mathbf E\Big[1 \wedge \|u^M - u^{M'}\|_{C^\beta([0,T],\cC^\alpha(\p_a))}\Big] = 0 \, .
\end{equ}
The latter implies that the sequence $(u^M)_{M}$ converges in probability in $C^\beta([0,T],\cC^\alpha(\p_a))$
to a unique limit $u$ which, by~\eqref{e:aprioriM} and Fatou's lemma, satisfies~\eqref{e:apriori}, and,
up to passing to a subsequence, the convergence is almost sure.
A diagonal argument then guarantees that we can find an event of full $\Prob$-measure
in which the convergence takes place in $C^\beta([0,T],\cC^\alpha(\p_a))$ for every $T > 0$,
i.e. in $C_{\mathrm{loc}}^\beta(\R_+,\cC^\alpha(\p_a))$,
so that $\Prob$-a.s. $u\in C_{\mathrm{loc}}^\beta(\R_+,\cC^\alpha(\p_a))$.
Further, $u$ is $(\cF_t)_{t \geq 0}$-adapted because each $u^M$ is, and it
can be easily seen to satisfy~\eqref{e:SBEregMild} thanks to~\eqref{e:CVSHE}
and the continuity of the map
$C_{\mathrm{loc}}^\beta(\R_+,\cC^\alpha(\p_a))\ni f\mapsto \int_0^t K_{t-s}\cN^1[f_s]\dd s$.
In other words, $u$ is a mild solution to~\eqref{e:SBEsystem} driven by $(\eta,\vec{\xi})$.

%
\medskip

\noindent {\it 2) Path by path uniqueness. } It is an immediate consequence of the fact that
$\Prob$-a.s. $X, u\in C(\R_+, \cC^\alpha(\p_a))$ and Lemma~\ref{l:AprioriLinear} applied with
$v = u_1 - u_2$, $X \equiv 0$ and $w = u_1 + u_2$.
\medskip
%
%

\noindent {\it 3) Skew-reversibility, stationarity and Markovianity.} Since for every $t \geq 0$,
$u^M_t$ converges in probability to $u_t$ in $\cC^\alpha(\p_a)$ and skew-reversibility
holds for $u^M$ it also holds for the limit. Further, convergence in
probability implies convergence in law, so that, since by Lemma~\ref{l:Tightness},
$u_t^M \eqlaw \eta^1_M$, and $\eta^1_M$ converges in law to $\eta^1$
(because of Lemma \ref{e:CVSHE}, noticing that $X^M_0 = \eta^1_M$ and $X_0 = \eta^1$),
we deduce that $u_t$ and $\eta^1$ have the same law, which implies stationarity.

For the Markov property, note that $u$ is a measurable function of $(\eta,\vec{\xi})$, so that it can be written as
$u=\PHI(\eta,\vec{\xi})$. Now, defining the translation $\theta_\tau f\eqdef(f_{t+\tau})_{t\geq 0}$ for $\tau\geq 0$,
it is immediate to see that $\theta_\tau \PHI(\eta, \vec{\xi})$ and $\PHI(u(\tau), \theta_\tau \vec{\xi})$ are both
mild solutions to~\eqref{e:SBEsystem} driven by $(u_\tau,\theta_\tau\vec{\xi}^{\,1})$ so that, by uniqueness,
they coincide. In particular, since $u_\tau$ is $\cF_\tau$-measurable while $\theta_\tau \vec{\xi}^{\,1}$ is independent
of $\cF_\tau$ and has the same law as $\vec{\xi}^{\,1}$, for bounded functionals $F$ it holds
\begin{equ}
\mathbf E\big[F\big(\theta_\tau \PHI(\eta, \vec{\xi})\big)| \cF_\tau\big] = \mathbf E[F\big(\PHI(u_\tau, \vec{\xi})\big)] \, .
\end{equ}
which means that the process is Markov, hence the proof is completed.
\end{proof}

\subsection{Decay properties of a non-local PDE}\label{a:FK}

In this appendix, we determine decay properties of the non-local PDE~\eqref{e:Z}
derived in the proof of Proposition~\ref{p:corrdecay}.
Let us take a slightly more general approach,
namely consider the PDE
\begin{equ}[e:Zapp]
\partial_t Z=\Big(\tfrac12\Delta +L_{\PHI,f}\Big)Z+g Z\,,\qquad Z_0=h\,,
\end{equ}
where $Z=Z_t(x)$ for $(t,x)\in\R_+\times\R^2$, $f\ge0,g\ge0,h\ge0$ are locally bounded functions on $\R^2$,
and $L$ is the operator acting on functions $u\colon\R^2\to\R$ as
\begin{equ}[e:Lapp]
L_{\PHI,f}u(x)\eqdef f(x)\int\PHI(x-y) [u(y)-u(x)]\dd y
\end{equ}
for $\PHI$ a non-negative, compactly supported and even function.
Under suitable assumptions on $f,g,h$, the solution $Z$ decays exponentially fast, as the next lemma shows.

\begin{lemma}\label{l:decayZ}
Let $\PHI,f,g,h$ be functions on $\R^2$ such that $\PHI$ is non-negative, compactly supported and even,
$f,g$ are non-negative and such that there exists $\delta\in[1,1/2)$ for which
$\sup_x|f(x)\p_{\delta}(x)|\vee\sup_x|g(x)\p_{\delta}(x)|<\infty$, and $h$ is non-negative,
bounded and compactly supported in a ball
of radius $r_h>0$. Then, for every $T\geq 0$ and $\nu>0$
there exists a positive constant $C=C(\PHI,\delta,T, \nu)<\infty$
such that for all $|x|\geq 2 r_h$ and $t\leq T$, we have
\begin{equ}[e:decayZ]
Z_t(x)\leq C\|h\|_\infty e^{-C|x|}\,.
\end{equ}
\end{lemma}
\begin{proof}
Throughout the proof we will denote by $C$ a generic constant
depending only on the parameters in the statement and whose value might change
from line to line.

As mentioned in the proof of Proposition~\ref{p:corrdecay},
$L_{\PHI,f}$ is the generator of a pure jump Markov process on $\mathbb{R}^2$
that has rate $f(x) \PHI (y-x) \dd y$  of jumping from $x$ to
$\dd y$. Therefore, $(1 / 2) \Delta + L_{\PHI,f}$ is the generator of a Markov
process $Y$ on $\mathbb{R}^2$ that has a Brownian and a jump part, the
Brownian noise being independent from the Poisson jump noise. We call
$\P^Y_x$ the law of $Y$ with initial condition $x$.
The reader can then easily check that the following Feynman-Kac formula
holds
\begin{equ}
Z_t(x)=\E^Y_x\Big[h(Y_t) e^{\int_0^tg(Y_s)\dd s}\Big]
\end{equ}
which, by Cauchy-Schwarz, implies
\begin{equ}[e:CSpde]
Z_t(x)^2\leq \|h\|_\infty^2\,\P^Y_x(|Y_t|\leq r_h)\,\E^Y_x\Big[ e^{2\int_0^tg(Y_s)\dd s}\Big]
\end{equ}
and we will separately bound the last two factors at the right hand side starting from the first.

Let $x\in\R^2$ be such that $|x|\geq 2r_h$. As $Y$ starts at $x$, on the event $| Y_t | \leq r_h$,
$Y$ has to move at least at distance $|x|/2$ in the time-interval $[0,t]$.
Let $N$ be the number of jumps that the process does in $[0,t]$ starting
from positions smaller than $|x|+r_\PHI$ and $E$ be the event that $N>|x|/(4r_{\PHI})$, $r_{\PHI}$ being
the radius of the support of $\PHI$, which is finite by assumption.
We claim that
\begin{equ}[poissonsmall]
\mathbb{P}^Y_{x} (E) \lesssim e^{- C | x |}\,,
\end{equ}
for $C$ some constant depending on all the parameters in the statement, but independent of $x$.
To see why~\eqref{poissonsmall} holds, note that $N$ is dominated by a Poisson random variable $X$
of parameter $\lambda \lesssim t \sup \{ f (y), | y | \leq | x | + r_\PHI \} \lesssim t |x|^{2 \delta}$.
Now, an elementary estimate then gives that the probability of $X>k$ is upper bounded by $\lambda^k / k!$,
which, via an application of Stirling's formula, implies~\eqref{poissonsmall} since $2\delta<1$.
On the other hand, $Y$'s jumps are at most of size $r_\PHI$, hence, on $E^c$,
the contribution to $|Y_t-x|$ coming from the jumps counted by $N$ is at most $Nr_\PHI\leq |x|/4$ and
the Brownian noise must do the rest. Thus,
\begin{equ}[browniansmall]
\P^Y_x(|Y_t|\leq r_h\;;\;E^c)\leq \P_0^B\Big(\sup_{0 < s_1 < s_2 < t} |B_{s_2} - B_{s_1}|\geq |x|/4\Big) \lesssim e^{-C|x|^2}\,.
\end{equ}
We now turn to the second factor in~\eqref{e:CSpde} for which,
in view of~\eqref{poissonsmall} and~\eqref{browniansmall}, it suffices to show that
\begin{equ}[e:finalPDE]
\E^Y_x\Big[ e^{2\int_0^tg(Y_s)\dd s}\Big]\lesssim e^{C | x |^{\gamma}}
\end{equ}
for some $\gamma<1$.
Denote by $M$ the maximum of $|Y|$ in the interval $[0,t]$. Then, using
that $g$ grows at most as $|x|^{2\delta}$, we have
\begin{equs}
\E^Y_x\Big[ e^{2\int_0^t g(Y_s)\dd s}\Big]\lesssim \E^Y_x\Big[ e^{C t M^{2\delta}}\Big]\leq e^{C t |x|^{2\delta}}+
\sum_{n>2|x|}\P^Y_x(M\in[n-1,n]) \; e^{C t n^{2\delta}}\,.
\end{equs}
To control the probability of the event $\{M\in[n-1,n]\}$, we argue as above.
If $M\in[n-1,n]$, then either $Y$ has $n / (4 r_{\PHI})$ jumps from positions
with norm smaller than $n + r_\Phi$, or the Brownian noise induces a displacement of at least  $n / 2$ in time $t$.
The latter event has probability $\lesssim \exp (- n^2 / 2 t)$.
The probability of the  former event is upper bounded by the probability
that a Poisson random variable of parameter $\lambda \lesssim t n^{2\delta}$
is larger than $n / (4 r_\PHI)$, which decays super-exponentially in $n$, as mentioned above.
Therefore, $\mathbb{P}^Y_{x_0} (M \in [n - 1, n])$
decays at least exponentially in $n$ and the sum is small since $2\delta<1$.
It follows that~\eqref{e:finalPDE} holds for $\gamma=2\delta<1$, and the proof of the statement is complete.
\end{proof}

\section{The action of $\gen$: proofs} \label{a:Basics}

In this appendix we provide the proof of Lemma~\ref{l:ActionGen}, which will be split in two steps.
Let us begin with the following.

\begin{lemma}
For any $n\in\N$ and $\psi\in\cS(\R^{2n})\cap\fock{n}{}{\tau}$, it holds that
$\gen\psi=(\gensy+\genap+\genam)\psi$, where the operators at the
right hand side are defined according to formulas~\eqref{e:gensy}-\eqref{e:genam:fock}.
\end{lemma}
\begin{proof}
The proof follows a (by now) standard procedure which has been implemented in a number of references
for slightly different equations (see e.g.~\cite{GPGen, CES, CGT}),
so we refrain from carrying out the (long but) straightforward computations.
The main difference with respect to the literature listed above
is that our equation is defined on the full space and the Gaussian analysis is done
with respect to a regularised measure
but the bulk of the argument is exactly the same. At first, one considers functionals of the
form $I_n^\tau(h^{\otimes n})=n! H_n(\eta^\tau(h))$ (with $H_n$ the $n-th$ Hermite polynomial) for $h\in\cS(\R^2)$ and applies
It\^o's formula to $ H_n(u^\tau_t(h))$. Then, one reads off the drift of the latter from which
one deduces the action of $\gen$ on $I_n^\tau(h^{\otimes n})$. By basic calculus
on the $n$-th Wiener-It\^o integrals $I^\tau_n$ (and in particular the product formula~\cite[Proposition 1.1.3]{Nualart})
it is then possible to single out the contributions which respectively give rise to $\gensy$, $\genap$
and $\genam$. Hence, the equality $\gen=\gensy+\genap+\genam$ holds
on functionals of the form $h^{\otimes n}$ which then
can be extended to $\cS(\R^{2n})\cap\fock{n}{}{\tau}$ by
polarisation.
\end{proof}

In order to extend the definition of the operators $\gensy$, $\genap$
and $\genam$ to the whole of $\core$, show that they satisfy the stated symmetry
properties and verify that $\core$ is indeed a core, we will need the following lemma
which holds for $\tau>0$ {\it fixed}.

\begin{lemma}\label{l:GSC}
The operators $\gensy$, $\genap$ and
$\genam$ are well-defined and continuous on $\core$ and
there exists a constant $C=C(\tau)>0$ 
for all $\psi\in\core\cap\fock{n}{}{\tau}$
\begin{equs}
\|\gena_{\pm}\psi\|_\tau&\leq C n^{3/2}\|\psi\|_\tau\,.\label{e:ContA}
\end{equs}
Furthermore, the graded sector condition~\cite[Section 2.7.4]{KLO} holds for fixed $\tau$, i.e.
there exists a constant $C=C(\tau)>0$ such that, for all $\psi\in\fock{n}{1}{\tau}$, 
we have
\begin{equ}[e:GSC]
\|\gena_{\pm}\psi\|_{\fock{}{-1}{\tau}}\leq C \sqrt{n}\|\psi\|_{\fock{}{1}{\tau}}\,.
\end{equ}
\end{lemma}
\begin{remark}\label{rem:WeakvsStrong}
The dependence on $\tau$ in the graded sector condition~\eqref{e:GSC} can be explicitly tracked and it can
be shown that $C=C(\tau)\sim(\log\tau)^{1/2}$. That is, the graded sector condition is {\it not satisfied
uniformly in $\tau$} (see also the discussion in Remark~\ref{rem:GSCgenafl}).
\end{remark}
\begin{proof}
Concerning well-posedness and continuity, they are obvious for $\gensy$,
while for $\genap$ and $\genam$ the argument is similar so we will only focus on the former.
Let $\gamma\geq 0$ and $\psi\in\cS(\R^{2n})\cap\fock{n}{}{\tau}$. By Cauchy-Schwarz, we get
\begin{equs}
\|\genap \psi\|_{\fock{}{\gamma}{\tau}}^2 \lesssim_\tau &(n+1)n^2\, n! \int \Xi_n^\tau(\dd p_{1:n}) |\hat\psi(p_{1:n})|^2\times\label{b:ContM}\\
& \times\int |p_1|^2\big(1 + |\sqrt{R_\tau}(p_1-q,q,p_{2:n})|^2\big)^\gamma  \Theta_\tau(q) \Theta_\tau(p_1-q)\dd q\,.
\end{equs}
Now, bounding $|p_1|^2\lesssim |q|^2+|p_1-q|^2$ and using that $\Theta_\tau$ decays fast (recall its definition in~\eqref{e:Mollifiers}),
the integral over $q$ can be bounded above by a $\tau$-dependent constant times $(1+|\sqrt{R_\tau}p_{1:n}|^2)^{\gamma}$, 
from which we deduce
\begin{equ}
  \label{e:ContM}
\|\genap \psi\|_{\fock{}{\gamma}{\tau}} \lesssim_\tau n^{3/2} \, \|\psi\|_{\fock{}{\gamma}{\tau}}^2\,.
\end{equ}
In particular, by taking $\gamma=2$, we conclude that
the operator $\genap$ can be extended by continuity to a map from $\core$ to $\fock{}{\gamma}{\tau}\subset\fock{}{}{\tau}$,
and, by taking $\gamma=0$, satisfies~\eqref{e:ContA}.

At last it remains to argue~\eqref{e:GSC}. The same bound was obtained
in~\cite[Lemma 2.4]{CGT} in the periodic setting, but the estimate can be similarly shown in the
present context (for $\tau$ fixed!) and the proof is thus omitted.
\end{proof}

We are now ready to complete the proof of Lemma~\ref{l:ActionGen}.

\begin{proof}[of Lemma~\ref{l:ActionGen}]
We have shown above that on $\cS(\R^{2n})\cap\fock{n}{}{\tau}$,
$\gen$ coincides with $\gensy+\genap+\genam$. Lemma~\ref{l:GSC} guarantees that
the operators $\gensy, \genap$ and $\genam$ are continuous from $\core$ to $\fock{}{}{\tau}$
so that, since by Lemma~\ref{lem:Contractivity} $\gen$ is closed, we deduce that
$\core\subset{\rm Dom}(\gen)$ and that, on $\core$, $\gen=\gensy+\genap+\genam$.
The symmetry properties can be straightforwardly checked from the formulas~\eqref{e:gensy}-\eqref{e:genam:fock}
so that we only need to prove that $\core$ is a core for $\gen$, which amounts to show
that the closure of $\gen$ on $\core$ coincides with $\gen$.


At first, we claim that the closure of the range of $(1-\gen)$ restricted to $\core$ is dense in $\fock{}{}{\tau}$.
We will prove such claim at the end, but first we argue that it implies the result.
Let $x \in \mathrm{Dom}(\gen)$ (so that $(1-\gen)x\in\fock{}{}{\tau}$) and
let $(x_n)_{n \geq 1}$ be a sequence in $\core$ such that $(1 - \gen) x_n \to (1 - \gen) x$ in $\fock{}{}{\tau}$,
whose existence is guaranteed by the claim above.
To conclude, it only remains to show that $(x_n)_{n \geq 1}$ converges to $x$ in $\fock{}{}{\tau}$.
Notice that by~\eqref{e:BoundRes} (which is applicable since, for every $n$,
$x_n\in\core\subset\fock{}{1}{\tau}\cap{\rm Dom}(\gen)$), we have
\begin{equ}
\|x_n - x_m\|_\tau \lesssim\|(1 - \gen) x_n - (1 - \gen) x_m\|_\tau \, ,
\end{equ}
so that $(x_n)_{n \geq 1}$ is Cauchy in $\fock{}{}{\tau}$.
Let $\tilde x \in \fock{}{}{\tau}$ be the corresponding limit.
As $\gen$ is closed, $\tilde x \in \mathrm{dom}(\gen)$ and we have
$(1 - \gen) \tilde x = (1 - \gen) x$. Hence,
applying the resolvent to both sides we obtain $\tilde x = x$ as desired.
\medskip

We are left to prove the claim.
For $n\geq 1$, let us define the operator $\gen_{ n}$ acting on
$\fock{\leq n}{}{\tau} \eqdef \bigoplus_{k=0}^n \fock{k}{}{\tau}$ as
\begin{equ}
\cL^\tau_{ n} \eqdef P^{ n} \big(\gensy + \gena \big) P^{ n} \, ,
\end{equ}
where $P^{n}$ is the orthogonal projection of $\fock{}{}{\tau}$ onto $\fock{\leq n}{}{\tau}$.
Now, $-\gensy$ is positive self-adjoint while $P^n\gena P^{n}$
is skew-symmetric and bounded for $\tau$ fixed (see~\eqref{e:ContM}),
therefore $(1 - \gen_{ n})^{-1}$ is a bounded operator on $\fock{\leq n}{}{\tau}$,
taking values in $\fock{\leq n}{2}{\tau}(\subset\core)$.
Note that the inclusion holds
because of the projection: any element in $\fock{\leq n}{2}{\tau}$ is compactly supported
in chaos so that, in particular, it decays sufficiently fast.
Let us fix $\phi \in \fock{}{}{\tau}$ and set
$r^{(n)} \eqdef (1 - \gen_{ n})^{-1}P^n\phi \in\fock{\leq n}{2}{\tau}$.
A consequence of \cite[Proposition 2.8]{CGT}
is that for any $\alpha\geq 0$, there exists a constant $C=C(\tau)>0$ such that $r^{(n)}$ satisfies
\begin{equ}[e:aprioriBAD]
\sup_{k, n} \, (1+k^\alpha)\| r^{(n)}\|_\tau\leq C\,.
\end{equ}
Now, observe that
\begin{equ}
\|(1 - \gen) r^{(n)} - \phi\|_\tau \leq \|(\Id - P^{ n}) \phi\|_\tau + \|\genap r^{(n)}_n\|_\tau\lesssim \|(\Id - P^{ n}) \phi\|_\tau+n^{3/2}\|r^{(n)}_n\|_\tau\, ,
\end{equ}
where in the second bound we exploited~\eqref{e:ContA}. But now, the two terms
at the right hand side vanish as $n\to\infty$ - the first since $\phi\in\fock{}{}{\tau}$,
and the second by the polynomial decay in $n$ given by~\eqref{e:aprioriBAD}.
Thus, the proof is concluded.
\end{proof}

\section{The Replacement Lemmas: integral approximations estimates } \label{s:TechnicalEstimate}

In this appendix, we detail the main technical ingredient which is at the heart of the proof of 
the Replacement Lemmas in Section~\ref{sec:apprG}. 
Before stating it, let us introduce some useful notations which will be used throughout the section. 
For  $p_{1:n} \in \R^{2n}$, $q\in\R^2$, and $i \in \{1, ..., n\}$, we denote by  $q' \eqdef p_i - q$ and set 
\begin{equs}[e:NotationsRepl]
\Gamma^\tau_i \eqdef \tfrac{1}{2} \, |\sqrt{R_\tau} \, (q,q',&p_{1:n\setminus\{i\}})|^2 \, , \qquad \tilde \Gamma^\tau \eqdef |\sqrt{R_\tau} q|^2 + \tfrac{1}{2} |\sqrt{R_\tau} p_{1:n}|^2 \, , \\
\Gamma^{\fe_1}_i &\eqdef \tfrac{1}{2} \, |\fe_1 \cdot (q,q',p_{1:n\setminus\{i\}})|^2 \, .
\end{equs}
Notice that for any $a,b,c\in\R^2$ such that $a+b=c$, 
then  $|\sqrt{R_\tau} a|^2 + |\sqrt{R_\tau}b|^2 \asymp |\sqrt{R_\tau}a|^2 + |\sqrt{R_\tau}c|^2 \asymp |\sqrt{R_\tau}b|^2 + |\sqrt{R_\tau}c|^2$, 
and the same holds replacing  the scaled norm on $\R^2$ with $|\cdot|$. 
In particular, this ensures that, for any $i\in\{1,\dots,n\}$,  $\Gamma^\tau_i\asymp \tilde\Gamma^\tau$.

\begin{lemma} \label{lem:ReplTech}
Let $\PHI \colon \R^3 \to \R$ be a locally Lipschitz function such that there exists $C_\PHI > 0$ for which for every $t > 0$, $i \in \{1,2,3\}$
and $x_1, x_2, x_3 > 0$ where the partial derivatives of $\PHI$ are well-defined, it holds
\begin{equ} \label{e:ReplTechHyp1}
\PHI(t x_1, t x_2, x_3) = t^{-1} \, \PHI(x_1, x_2, x_3)
\end{equ}
\begin{equ} \label{e:ReplTechHyp2}
|\PHI(x_1, x_2, x_3)| \leq C_\PHI\frac{1 \vee \nu_\tau \tfrac{x_2}{x_1}}{x_1 + M \nu_\tau \, x_2} \, , \quad |\partial_i \PHI(x_1, x_2, x_3)| \leq C_\PHI\frac{1 \vee \nu_\tau \tfrac{x_2}{x_1}}{(x_1 + M \nu_\tau \, x_2) \, x_i} \, .
\end{equ}
Then, adopting the notations in~\eqref{e:NotationsRepl}, uniformly over $p_{1:n}\in\R^{2n}$, $i\in\{1,\dots,n\}$ and $M\geq 0$, we have 
\begin{equs}[e:ReplTech]
\Big|\frac{\lambda^2 \nu_\tau^{3/2}}{\pi^2} &\int_{\R^2}  \PHI\big(1 + \Gamma^\tau_i, \Gamma^{\fe_1}_i, L^\tau(\Gamma^\tau_i)\big) \, \cJ_{q,q'}^\tau \, \chi_{q, q'}^\sharp(p_{1:n}) \, \dd q \\
&- \int_0^{L^\tau\big(\tfrac12|\sqrt{R_\tau} p_{1:n}|^2\big)}  \int_0^{\nu_\tau^{-1}} \frac{\PHI(1,\varsigma,\ell)}{\pi \sqrt{\varsigma (1 - \nu_\tau \varsigma)}}\dd \varsigma \dd \ell  \Big| \lesssim  C_\PHI \lambda^2 \frac{1 \vee \log (1+M)}{1 \vee M^{1/2}} \, \nu_\tau \, ,
\end{equs}
where the map $L^\tau$ is given according to~\eqref{e:Leps}, and $\cJ_{q,q'}^\tau$ was introduced in~\eqref{e:diag}. 
\end{lemma}

\begin{remark} \label{rmk:ReplTech}
We will apply the previous statement to two specific functions, namely 
\begin{equ}
\PHI(\gamma, \gamma', \ell) = \frac{1}{\gamma + g_M^\tau(\ell) \, \gamma'}\,,\qquad\text{and}\qquad \PHI(\gamma, \gamma', \ell) = \frac{f(\ell) \, \gamma'}{(\gamma + g_M^\tau(\ell) \, \gamma')^2}\,,
\end{equ}
for some function $f$ on $\R_+$. 
The former will appear in the Replacement Lemma~\ref{p:1stReplLemma} and the latter in the Recursive Replacement 
Lemma~\ref{p:2ndReplLemma}. It is not hard to see that~\eqref{e:ReplTechHyp1} 
and~\eqref{e:ReplTechHyp2} hold for the first with $C_\PHI=1$, and the same is true for the second 
provided $f$ satisfies the assumptions of Lemma~\ref{p:2ndReplLemma}, with $C_\PHI=6 \, C_f$ and $C_f$ such that~\eqref{e:Assf} 
is verified. 
\end{remark}

In the proof of the above (and in that of Proposition~\ref{p:GA+Norm}), we will exploit 
the basic integral estimates summarised in the next lemma. 

\begin{lemma} \label{l:technical-integral}
Let $\gamma, A > 0$ and $M \geq 0$, then for every $\tau > 0$ it holds
\begin{equs} \label{e:Fact1}
\,&\int_{\R^2}\frac{\dd q}{(A^2 + |\sqrt{R_\tau} q|^2 + M \nu_\tau |\fe_1 \cdot q|^2)^2}\lesssim \nu_\tau^{-1/2} \frac{1}{A^2} \frac{1}{1 \vee M^{1/2}} \qquad\\
&\int_{\R^2}  \, \frac{\1_{\{\nu_\tau^{1/2} \, |\fe_1 \cdot q| \leq A\}}\;\dd q}{A^2 + |\sqrt{R_\tau} q|^2 + M \nu_\tau |\fe_1 \cdot q|^2} \lesssim \nu_\tau^{-1/2}\frac{1 \vee \log(1+M)}{1 \vee M^{1/2}} \, ,\qquad\label{e:Fact2}\\
&\int_{\R^2} \frac{\dd q}{|\sqrt{R_\tau} q|^2 + M \nu_\tau |\fe_1 \cdot q|^2} \Big[\Big(\frac{|\fe_1 \cdot q|}{A}\Big)^\gamma \wedge \Big(\frac{|\fe_1 \cdot q|}{A}\Big)^{-\gamma} \Big]\lesssim_\gamma \nu_\tau^{-1/2}\frac{1}{1 \vee M^{1/2}} \, ,\qquad\label{e:Fact3}\\
&\int_{\R^2} \frac{\dd q}{|\sqrt{R_\tau} q|^2 + M \nu_\tau |\fe_1 \cdot q|^2} \Big[\Big(\frac{|\sqrt{R_\tau} q|}{A}\Big)^\gamma \wedge \Big(\frac{|\sqrt{R_\tau} q|}{A}\Big)^{-\gamma}\Big]\lesssim_\gamma \nu_\tau^{-1/2}\frac{1}{1 \vee M^{1/2}}  \, .\qquad \label{e:Fact4}
\end{equs}
\end{lemma}

\begin{proof}
For the first three bounds, write $q=(x,y)\in\R^2$ and use the estimate 
$|\sqrt{R_\tau} q|^2 + M \nu_\tau |\fe_1 \cdot q|^2= \tfrac{\nu_\tau}2x^2 +\tfrac12y^2 + M \nu_\tau x^2\gtrsim y^2 +(1\vee M)\nu_\tau x^2$. 
Then, for~\eqref{e:Fact1}, we have 
\begin{equs}
\int_{\R^2}&\frac{\dd q}{(A^2 + |\sqrt{R_\tau} q|^2 + M \nu_\tau |\fe_1 \cdot q|^2)^2}\lesssim \int_\R\int_{\R}\frac{1}{A^4 +y^4+ (1\vee M)^2\nu_\tau^2 x^4}\dd x\dd y\\
&\lesssim \frac{\nu_\tau^{-1/2}}{1\vee M^{1/2}}\int_\R\frac{\dd y}{(A^4+y^4)^{3/4}}\lesssim \frac{\nu_\tau^{-1/2}}{1\vee M^{1/2}}\int_\R\frac{\dd y}{A^3+y^3}\lesssim \frac{\nu_\tau^{-1/2}}{A^2(1\vee M^{1/2})}
\end{equs}
where in the second and the last step we used the elementary bound 
\begin{equ}[e:Ele1]
\int_{\R} \frac{\dd u}{a |u|^\alpha + b} \lesssim_\alpha \frac{1}{a^{1/\alpha} \, b^{1-1/\alpha}} 
\end{equ}
which holds for any $\alpha>1$, uniformly over $a,b>0$. 

Concerning~\eqref{e:Fact2}, we proceed similarly, i.e.
\begin{equs}
\int_{\R^2} & \, \frac{\1_{\{\nu_\tau^{1/2} \, |\fe_1 \cdot q| \leq A\}}\;\dd q}{A^2 + |\sqrt{R_\tau} q|^2 + M \nu_\tau |\fe_1 \cdot q|^2}\lesssim \int_{\R}\int_{\R} \frac{\1_{\{\nu_\tau^{1/2} \, |x| \leq A\}}}{A^2 +y^2+ (1\vee M)\nu_\tau x^2}\dd x\dd y\\
&\lesssim \int_\R\frac{\1_{\{\nu_\tau^{1/2} \, |x| \leq A\}}}{\sqrt{A^2 + (1\vee M)\nu_\tau x^2}}\dd x\lesssim \int_\R\frac{\1_{\{\nu_\tau^{1/2} \, |x| \leq A\}}}{A + (1\vee M^{1/2})\nu_\tau^{1/2} |x|}\dd x\lesssim \nu_\tau^{-1/2}\frac{1 \vee \log(1+M)}{1 \vee M^{1/2}}
\end{equs}
the integral in $y$ being controlled via~\eqref{e:Ele1} with $\alpha=2$ and the last via a direct computation. 

For~\eqref{e:Fact3}, we use once again~\eqref{e:Ele1} with $\alpha=2$ on the integral in $y$, so that  
\begin{equs}[e:Fact3Proof]
\int_{\R^2}& \frac{\dd q}{|\sqrt{R_\tau} q|^2 + M \nu_\tau |\fe_1 \cdot q|^2} \Big[\Big(\frac{|\fe_1 \cdot q|}{A}\Big)^\gamma \wedge \Big(\frac{|\fe_1 \cdot q|}{A}\Big)^{-\gamma} \Big]\\
&\lesssim \int_{\R^2} \frac{(|x|/A)^\gamma \wedge (|x|/A)^{-\gamma} }{y^2 + (M\vee 1) \nu_\tau x^2} \dd y\dd x\lesssim\frac{\nu_\tau^{-1/2}}{1\vee M^{1/2}}\int_{\R_+} \frac{(x/A)^\gamma \wedge (x/A)^{-\gamma} }{x}\dd x
\end{equs}
from which the bound follows by splitting the cases $x<A$ and $x>A$. 

At last, in the integral at the left hand side of~\eqref{e:Fact4} we first apply the change of variables $\tilde q=\sqrt{R_\tau}q$ 
and then pass to polar coordinates, so that it equals
\begin{equs}
\nu_\tau^{-1/2}&\int_{\R^2}\frac{(|q|/A)^\gamma \wedge (|q|/A)^{-\gamma}}{|q|^2+M|\fe_1\cdot q|^2}\dd q\\
&\lesssim \nu_\tau^{-1/2}\int_0^{2 \pi} \frac{\dd \theta}{1 + (M\vee 1) \cos(\theta)^2}\int_{\R_+}\frac{(r/A)^\gamma \wedge (r/A)^{-\gamma} }{r}\dd r \lesssim_\gamma \frac{\nu_\tau^{-1/2}}{1 \vee M^{1/2}}\,:
\end{equs}
the integral over the radial component was bounded as in~\eqref{e:Fact3Proof}, while that of the angular one 
can be obtained by a direct computation. 
\end{proof}

We are now ready to prove Lemma~\ref{lem:ReplTech}. 

\begin{proof}[of Lemma~\ref{lem:ReplTech}]
The bulk of the proof consists of showing that   
\begin{equs}[e:TechReplLemmaGoal]
\frac{\lambda^2 \nu_\tau^{3/2}}{\pi^2}\bigg|& \int_{\R^2}  \PHI\big(1 + \Gamma^\tau_i, \Gamma^{\fe_1}_i, L^\tau(\Gamma^\tau_i)\big) \, \cJ_{q,q'}^\tau \, \chi_{q, q'}^\sharp(p_{1:n}) \, \dd q \\
&- \int_{\R^2}  \, \frac{\PHI\big(1 + \tilde \Gamma^\tau, \tfrac{1 + \tilde \Gamma^\tau}{|\sqrt{R_\tau} q|^2} |\fe_1 \cdot q|^2, L^\tau(\tilde \Gamma^\tau)\big)}{1 + (1 + \tilde \Gamma^\tau)\tau^{-1}}\dd q \bigg| \lesssim  C_\PHI \lambda^2 \frac{1 \vee \log (1+M)}{1 \vee M^{1/2}} \, \nu_\tau\,.
\end{equs}
To see how the latter implies~\eqref{e:ReplTech}, let us manipulate the second term above, which we denote by $I_\PHI$. 
Using the scaling relation~\eqref{e:ReplTechHyp1} and the symmetry of the integrand with respect to 
the inversion $q\mapsto-q$, we see that $I_\PHI$ equals  
\begin{equ}
\frac{\lambda^2 \nu_\tau^{3/2}}{\pi^2} \int_{\R^2} \frac{\PHI\big(1, \tfrac{|\fe_1 \cdot q|^2}{|\sqrt{R_\tau} q|^2}, L^\tau(\tilde \Gamma^\tau)\big)}{(1 + \tilde \Gamma^\tau) \big(1 + (1 + \tilde \Gamma^\tau)\tau^{-1}\big)} \, \dd q = \frac{4 \lambda^2 \nu_\tau^{3/2}}{\pi^2} \int_{\R_+^2} \frac{\PHI\big(1, \frac{|\fe_1 \cdot q|^2}{|\sqrt{R_\tau} q|^2}, L^\tau(\tilde \Gamma^\tau)\big)}{(1 + \tilde \Gamma^\tau) \big(1 + (1 + \tilde \Gamma^\tau)\tau^{-1}\big)} \, \dd q \, .
\end{equ}
Then, we perform two change of variables, first $(\sigma,\varsigma) \eqdef (\tilde\Gamma^\tau, |\fe_1 \cdot q|^2/|\sqrt{R_\tau} q|^2)$, 
for which the determinant of the Jacobian is $\dd \sigma \dd \varsigma = 4 \sqrt{\varsigma (1 - \nu_\tau \varsigma)} \, \dd q$, 
and then $\ell \eqdef L^\tau(\sigma)$, for which instead 
$\dd \ell = \frac{\lambda^2 \nu_\tau^{3/2}}{\pi} [(1 + \sigma) (1 + (1 + \sigma) \tau^{-1})]^{-1}\dd \sigma$. 
As a consequence we obtain 
\begin{equs}[e:ChangeofVariables]
I_\PHI &= \frac{\lambda^2 \nu_\tau^{3/2}}{\pi^2}\int_{0}^{\nu_\tau^{-1}} \int_{|\sqrt{R_\tau} p_{1:n}|^2/2}^{+\infty} \frac{\PHI\big(1, \varsigma, L^\tau(\sigma)\big)}{(1 + \sigma) \big(1 + (1 + \sigma)/\tau\big)}\dd\sigma  \frac{\dd \varsigma}{\sqrt{\varsigma(1-\nu_\tau\varsigma)}}  \\
&= \int_0^{L^\tau\big(\tfrac12|\sqrt{R_\tau} p_{1:n}|^2\big)}  \int_0^{\nu_\tau^{-1}} \frac{\PHI(1,\varsigma,\ell)}{\pi \sqrt{\varsigma (1 - \nu_\tau \varsigma)}}\dd \varsigma \dd \ell \, ,
\end{equs}
which indeed coincides with the second integral at the left hand side of \eqref{e:ReplTech}.
\medskip

We are thus left to show~\eqref{e:TechReplLemmaGoal}. Heuristically, the reason why we expect it to hold 
is that only the regions $|\fe_1 \cdot q|, |\fe_1 \cdot q'| \gg 1 \vee |p_{1:n}|$ and $|q|,|q'| \ll \tau^{-1}$ 
should contribute and, on such regions, the integrands are indeed close. 
To formalise this idea, let $Q_- \eqdef \{\nu_\tau^{1/2} |\fe_1 \cdot q| \leq 10 \, (1 \vee |\sqrt{R_\tau} p_{1:n}|)\}$ 
and $Q_+$ be its complement. Assume the following estimates hold 
\begin{equs} 
&\int_{Q_-} \dd q \, \frac{C}{1 + |\sqrt{R_\tau} q|^2 + |\sqrt{R_\tau} p_{1:n}|^2 + M \nu_\tau |\fe_1 \cdot q|^2} \lesssim C_\PHI \, \frac{1 \vee \log(M)}{1 \vee M^{1/2}} \, \nu_\tau^{-1/2}\,, \label{e:ReplLemmaStep1}\\
&\int_{Q_+} \dd q \, \frac{C}{|\sqrt{R_\tau} q|^2 + M \nu_\tau |\fe_1 \cdot q|^2} \, \Big|\cJ_{q,q'}^\tau - \frac{1}{1 + (1 + \tilde \Gamma^\tau)\tau^{-1}} \Big| \lesssim C_\PHI \, \frac{\nu_\tau^{-1/2}}{1 \vee M^{1/2}} \, , \label{e:ReplLemmaStep2}\\
&\int_{Q_+} \dd q \, \Big|\PHI(1 + \Gamma^\tau_i, \Gamma^{\fe_1}_i, L^\tau(\Gamma^\tau_i)) - \PHI(1 + \tilde \Gamma^\tau, \tfrac{1 + \tilde \Gamma^\tau}{|\sqrt{R_\tau} q|^2} |\fe_1 \cdot q|^2, L^\tau(\tilde \Gamma^\tau))\Big| \lesssim C_\PHI \, \frac{\nu_\tau^{-1/2}}{1 \vee M^{1/2}}.   \qquad\quad\label{e:ReplLemmaStep3}
\end{equs} 
In order to deduce~\eqref{e:TechReplLemmaGoal}, we begin by splitting 
the first integral al the left hand side of~\eqref{e:TechReplLemmaGoal} 
in the regions $Q_-$ and $Q_-$. For the integral over $Q_-$, we bound $\PHI$ via the first inequality in~\eqref{e:ReplTechHyp2}, 
use that $1\vee \nu_\tau\frac{\Gamma^{\fe_1}_i}{1+\Gamma^\tau_i}\lesssim 1$, and 
control what is left with~\eqref{e:ReplLemmaStep1}.  
For the integral over $Q_+$, note that, on $Q_+$, the definition of $\chi^\sharp$ in~\eqref{e:ChiSharp} 
ensures that $\chi^\sharp_{q,q'}(p_{1:n}) = 1$. 
Then, we replace $\cJ_{q,q'}^\tau$ with $[1+(1+\tilde\Gamma^\tau)\tau^{-1}]^{-1}$, by adding and subtracting 
the corresponding term, and estimate the error by treating $\PHI$ as above and exploiting~\eqref{e:ReplLemmaStep2}.  
At last, we want to substitute  $(\Gamma^\tau_i, \Gamma^{\fe_1}_i)$ with 
$(\tilde\Gamma^\tau, \tfrac{1 + \tilde \Gamma^\tau}{|\sqrt{R_\tau} q|^2} |\fe_1 \cdot q|^2)$. Once again, 
we add and subtract the corresponding term, and 
control the difference by bounding $[1+(1+\tilde\Gamma^\tau)\tau^{-1}]^{-1}\leq1$ 
and using~\eqref{e:ReplLemmaStep3}. These last two steps are instrumental to 
reducing the integrand to a form in which the change of variables in~\eqref{e:ChangeofVariables} 
is possible. 
\medskip

Therefore, we are left to show~\eqref{e:ReplLemmaStep1}-\eqref{e:ReplLemmaStep3}. 
The first is the easiest since it can be deduced by simply removing $|\sqrt{R_\tau} p_{1:n}|^2$ from the denominator and then 
applying~\eqref{e:Fact2} from Lemma \ref{l:technical-integral}.

For~\eqref{e:ReplLemmaStep2}, recall the definition of $\rho_\tau$ and $\Theta_\tau$ in~\eqref{e:Mollifiers}, 
and that $\cJ_{q,q'}^\tau=\Theta_\tau(q) \Theta_\tau(q') \Theta_\tau(p_i)$. 
Since $\hat{\rho}_\tau \in \cS(\R^2)$ and $2 \pi \hat{\rho}_\tau(0) = 1$, we have 
\begin{equs}
\big|\cJ_{q,q'}^\tau - 1\big| \lesssim \tau^{-1/2}\max\{|\sqrt{R_\tau} q|, |\sqrt{R_\tau} q'|, |\sqrt{R_\tau} p_i|\}\lesssim \tau^{-1/2}\sqrt{R_\tau} q|,\\
\cJ_{q,q'}^\tau \lesssim \frac{1}{\tau^{-1/2}\max\{|\sqrt{R_\tau} q|, |\sqrt{R_\tau} q'|, |\sqrt{R_\tau} p_i|\}}\lesssim \frac{1}{\tau^{-1/2}|\sqrt{R_\tau} q|} \, ,
\end{equs}
where we also used that on $Q_+$, $|\sqrt{R_\tau} q| \asymp \max\{|\sqrt{R_\tau} q|, |\sqrt{R_\tau} q'|, |\sqrt{R_\tau} p_i|\}$. 
Therefore, 
\begin{equs}
\Big|\cJ_{q,q'}^\tau-\frac{1}{1 + (1+\tilde \Gamma)\tau^{-1}}\Big|\leq& \1_{\{\tau^{-1/2} |\sqrt{R_\tau} q| \leq 1\}}\Big(\big|\cJ_{q,q'}^\tau-1\big|+\Big|1 - \frac{1}{1 + (1+\tilde \Gamma)\tau^{-1}}\Big|\Big)\\
&+\1_{\{\tau^{-1/2} |\sqrt{R_\tau} q| > 1\}}\Big(\cJ_{q,q'}^\tau+\frac{1}{1 + (1+\tilde \Gamma)\tau^{-1}}\Big)\\
\lesssim& \tau^{-1/2} |\sqrt{R_\tau} q| \, \1_{\{\tau^{-1/2} |\sqrt{R_\tau} q| \leq 1\}} + \frac{\1_{\{\tau^{-1/2} |\sqrt{R_\tau} q| > 1\}}}{\tau^{-1/2} |\sqrt{R_\tau} q|}\,.
\end{equs}
Plugging the above estimate into the left hand side of~\eqref{e:ReplLemmaStep2}, 
we can upper bound it (modulo an absolute constant) by 
\begin{equ}
\int_{\R^2} \dd q \, \bigg(\frac{\tau^{-1/2} |\sqrt{R_\tau} q| \1_{\{\tau^{-1/2} |\sqrt{R_\tau} q| \leq 1\}}}{|\sqrt{R_\tau} q|^2 + M \nu_\tau |\fe_1 \cdot q|^2}  + \frac{(\tau^{-1/2} |\sqrt{R_\tau} q|)^{-1} \1_{\{\tau^{-1/2} |\sqrt{R_\tau} q| > 1\}}}{|\sqrt{R_\tau} q|^2 + M \nu_\tau |\fe_1 \cdot q|^2}\bigg) \lesssim \frac{  \nu_\tau^{-1/2}}{1 \vee M^{1/2}}\,,
\end{equ}
the last step being a consequence of~\eqref{e:Fact4}, applied with $\gamma=1$. 

It remains to prove~\eqref{e:ReplLemmaStep3}. Let us first control the integrand as  
\begin{equs}
|\PHI(1&+\Gamma^\tau_i, \Gamma^{\fe_1}_i, L^\tau(\Gamma^\tau_i)) - \PHI(1+\tilde\Gamma^\tau, \tfrac{1+\tilde \Gamma^\tau}{|\sqrt{R_\tau} q|^2} |\fe_1 \cdot q|^2, L^\tau(\tilde \Gamma^\tau))| \\
\leq& |\PHI(1+\Gamma^\tau_i, \Gamma^{\fe_1}_i, L^\tau(\Gamma^\tau_i)) - \PHI(1+\tilde\Gamma^\tau, \Gamma^{\fe_1}_i, L^\tau(\tilde \Gamma^\tau))| \\
&+ |\PHI(1+\tilde\Gamma^\tau, \Gamma^{\fe_1}_i, L^\tau(\tilde \Gamma^\tau)) - \PHI(1+\tilde\Gamma^\tau, \tfrac{\tilde \Gamma^\tau}{|\sqrt{R_\tau} q|^2} |\fe_1 \cdot q|^2, L^\tau(\tilde \Gamma^\tau))| \\
&+ |\PHI(1+\tilde\Gamma^\tau, \tfrac{\tilde \Gamma^\tau}{|\sqrt{R_\tau} q|^2} |\fe_1 \cdot q|^2, L^\tau(\tilde \Gamma^\tau)) - \PHI(1+\tilde\Gamma^\tau, \tfrac{1+\tilde \Gamma^\tau}{|\sqrt{R_\tau} q|^2} |\fe_1 \cdot q|^2, L^\tau(\tilde \Gamma^\tau))| \\
\lesssim& \frac{C_\PHI}{ |\sqrt{R_\tau} q|^2 + M \nu_\tau |\fe_1 \cdot q|^2}\bigg(\frac{|\Gamma^\tau_i-\tilde\Gamma^\tau|}{|\sqrt{R_\tau} q|^2}+\frac{\big|\Gamma^{\fe_1}_i-\frac{\tilde \Gamma^\tau}{|\sqrt{R_\tau} q|^2}|\fe_1\cdot q|^2\big|}{|\fe_1\cdot q|^2}+\frac{1}{|\sqrt{R_\tau} q|^2}\bigg)\,, 
\end{equs}
where in the last step we used mean value theorem, assumption~\eqref{e:ReplTechHyp2} on the gradient of $\PHI$, 
which in particular gives 
\begin{equs}
|\partial_\gamma \PHI(\gamma, \gamma', L^\tau(\gamma))| 
&\leq |\partial_1 \PHI(\gamma, \gamma', L^\tau(\gamma))| + |\partial_3 \PHI(t, \gamma', L^\tau(\gamma))| \times |\partial_\gamma L^\tau(\gamma)| \\
&\lesssim \frac{ C_\PHI(1 \vee \nu_\tau \frac{\gamma'}{\gamma})}{(\gamma + M \nu_\tau \gamma') \gamma} + \frac{ C_\PHI(1 \vee \nu_\tau \frac{\gamma'}{\gamma})}{(\gamma + M \nu_\tau \gamma') L^\tau(\gamma)} \frac{L^\tau(\gamma)}{\gamma} \lesssim C_\PHI\frac{ 1 \vee \nu_\tau \frac{\gamma'}{\gamma}}{(\gamma + M \nu_\tau \gamma') \gamma} \,, 
\end{equs}
that $\Gamma^\tau_i\asymp\tilde\Gamma^\tau\gtrsim |\sqrt{R_\tau} q|^2$ 
and that, on $Q_+$, $\Gamma^{\fe_1}_i\asymp \tfrac{\tilde \Gamma^\tau}{|\sqrt{R_\tau} q|^2}|\fe_1\cdot q|^2\gtrsim |\fe_1 \cdot q|^2$. 
Now, 
\begin{equs}
|\Gamma^\tau_i - \tilde \Gamma^\tau| &= |\langle \sqrt{R_\tau} p_i, \sqrt{R_\tau} q\rangle| \leq |\sqrt{R_\tau} p_i| \, |\sqrt{R_\tau} q|\,,\\
|\Gamma^{\fe_1}_i - |\fe_1 \cdot q|^2| &\leq |\fe_1 \cdot p_i| |\fe_1 \cdot q| + \frac{1}{2} |\fe_1 \cdot p_{1:n}|^2 \, , \\
\Big||\fe_1 \cdot q|^2 - \frac{1+\tilde \Gamma}{|\sqrt{R_\tau} q|^2} |\fe_1 \cdot q|^2\Big| &\leq \frac{(1+|\sqrt{R_\tau} p_{1:n}|^2) \, |\fe_1 \cdot q|^2}{|\sqrt{R_\tau} q|^2} \, ,
\end{equs}
so that, putting everything together, we obtain 
\begin{equs}
\int_{Q_+}& \dd q \, \Big|\PHI(1 + \Gamma^\tau_i, \Gamma^{\fe_1}_i, L^\tau(\Gamma^\tau_i)) - \PHI(1 + \tilde \Gamma^\tau, \tfrac{1 + \tilde \Gamma^\tau}{|\sqrt{R_\tau} q|^2} |\fe_1 \cdot q|^2, L^\tau(\tilde \Gamma^\tau))\Big| \\
&\lesssim \int_{Q_+}  \frac{C_\PHI\, \dd q }{ |\sqrt{R_\tau} q|^2 + M \nu_\tau |\fe_1 \cdot q|^2}\bigg(\frac{|\sqrt{R_\tau} p_i|}{|\sqrt{R_\tau} q|}+\frac{|\fe_1 \cdot p_i|}{|\fe_1 \cdot q|} + \frac{|\fe_1 \cdot p_{1:n}|^2}{|\fe_1 \cdot q|^2} + \frac{1+|\sqrt{R_\tau} p_{1:n}|^2}{|\sqrt{R_\tau} q|^2}\bigg)\\
&\lesssim C_\PHI \, \frac{\nu_\tau^{-1/2}}{1 \vee M^{1/2}}
\end{equs}
which holds thanks to~\eqref{e:Fact3} and~\eqref{e:Fact4} in Lemma~\ref{l:technical-integral}. 
Therefore, the proof of~\eqref{e:ReplLemmaStep3} 
(and of the lemma) is complete. 
\end{proof}

\end{appendix}

\section*{Acknowledgments} 
We wish to thank Scott Armstrong, Harry Giles, Massimiliano Gubinelli, Felix Otto and Peter Morfe 
for enlightnening discussions.
G.~C. gratefully acknowledges financial support
via the UKRI FL fellowship ``Large-scale universal behaviour of Random
Interfaces and Stochastic Operators'' MR/W008246/1. The research of Q.~M. and F.~T. was funded 
by the Austrian Science Fund (FWF) 10.55776/F1002. 
For open access purposes, the authors have applied a CC BY public copyright license to any author accepted manuscript version arising from this submission.
\bibliography{refs}
\bibliographystyle{Martin}

\end{document}